\documentclass[final,3p,times]{elsarticle}

\usepackage{amssymb}
\usepackage{bm}
\usepackage{amsmath}
\usepackage{mathrsfs}
\usepackage{float}
\usepackage{caption}
\usepackage{latexsym,amsfonts,amssymb}
\usepackage{graphicx,graphics}
\usepackage[algosection]{algorithm2e}
\usepackage{amsthm}
\usepackage{epstopdf}
\usepackage{subfigure}
\usepackage[misc]{ifsym}
\usepackage{multirow}
\usepackage{epstopdf}
\usepackage[colorlinks,
linkcolor=blue,
anchorcolor=blue,
citecolor=blue]{hyperref}
\biboptions{numbers,sort&compress}%%%%%%%%%%%ѹËõÒýÓÃÐòºÅ

\newtheorem{thrm}{Theorem}[section]
\newtheorem{lem}{Lemma}[section]

\newtheorem{remark}{Remark}[section]

\begin{document}

\begin{frontmatter}

%% Title, authors and addresses

%% use the tnoteref command within \title for footnotes;
%% use the tnotetext command for theassociated footnote;
%% use the fnref command within \author or \address for footnotes;
%% use the fntext command for theassociated footnote;
%% use the corref command within \author for corresponding author footnotes;
%% use the cortext command for theassociated footnote;
%% use the ead command for the email address,
%% and the form \ead[url] for the home page:
%% \title{Title\tnoteref{label1}}
%% \tnotetext[label1]{}
%% \author{Name\corref{cor1}\fnref{label2}}
%% \ead{email address}
%% \ead[url]{home page}
%% \fntext[label2]{}
%% \cortext[cor1]{}
%% \affiliation{organization={},
%%             addressline={},
%%             city={},
%%             postcode={},
%%             state={},
%%             country={}}
%% \fntext[label3]{}

\title{Improved uniform error bounds  for long-time dynamics of the  high-dimensional nonlinear space
fractional sine-Gordon equation with weak nonlinearity}

%% use optional labels to link authors explicitly to addresses:
%% \author[label1,label2]{}
%% \affiliation[label1]{organization={},
%%             addressline={},
%%             city={},
%%             postcode={},
%%             state={},
%%             country={}}
%%
%% \affiliation[label2]{organization={},
%%             addressline={},
%%             city={},
%%             postcode={},
%%             state={},
%%             country={}}

\author[1]{Junqing Jia}
\ead{matjiajq@sdu.edu.cn}
\author[2]{Xiaoqing Chi\corref{cor1}}
\ead{cxqwh@sdu.edu.cn}
\author[1]{Xiaoyun Jiang}
\ead{wqjxyf@sdu.edu.cn}
\cortext[cor1]{Corresponding author. }
\address[1]{School of Mathematics, Shandong University, Jinan 250100, PR China}
\address[2]{ School of Mathematics and Statistics, Shandong University, Weihai 264209, PR China.}

%%\affiliation{organization={},%Department and Organization
%            addressline={},
%            city={},
%            postcode={},
%            state={},
%            country={}}

\begin{abstract}

In this paper, we derive the improved uniform error bounds for the long-time dynamics of the $d$-dimensional $(d=2,3)$ nonlinear space fractional sine-Gordon equation (NSFSGE). The nonlinearity strength of the NSFSGE is characterized by $\varepsilon^2$ where $0<\varepsilon \le 1$ is a dimensionless parameter.  The second-order time-splitting method is applied to the temporal discretization and the Fourier  pseudo-spectral method is used for the spatial discretization. To obtain the explicit relation between the numerical errors and the parameter $\varepsilon$,  we introduce the regularity compensation oscillation technique to the convergence analysis of fractional models. Then we establish the improved uniform error bounds $O\left(\varepsilon^2 \tau^2\right)$ for the semi-discretization scheme and $O\left(h^m+\varepsilon^2 \tau^2\right)$ for the full-discretization scheme up to the long time at $O(1/\varepsilon^2)$.  Further, we extend the time-splitting  Fourier pseudo-spectral method to the complex NSFSGE as well as the oscillatory complex NSFSGE, and the improved uniform error bounds for them are also given. Finally, extensive numerical examples  in two-dimension  or three-dimension are provided to support the theoretical analysis. The differences in dynamic behaviors between the fractional sine-Gordon equation and classical  sine-Gordon equation are also discussed.

\end{abstract}

%%Graphical abstract
%\begin{graphicalabstract}
%%\includegraphics{grabs}
%\end{graphicalabstract}

%%Research highlights
%\begin{highlights}
%\item Research highlight 1
%\item Research highlight 2
%\end{highlights}


\end{frontmatter}

%% \linenumbers

%% main text
\section{Introduction}
Nonlinear wave equations and their dynamic properties explain the rich and colorful natural phenomena reasonably, such as the wave propagation, smooth scattering, emission and absorption in electromagnetism, superelastic material \cite{2403,2404,2402,2401}. Some well-known nonlinear wave equations are  Schr\"odinger equations,  Klein-Gordon equations, sine-Gordon equations, Korteweg-deVries equation, Burgers equation and so on. As an important class of nonlinear hyperbolic equations, the sine-Gordon equation has many soliton solutions. So it is widely used in biophysics, fluid motion,  quantum mechanics, nonlinear media super transport, plasma physics and other scientific fields \cite{1}. A lot of studies have been done for the analytical analysis and numerical research on the  soliton solutions of the sine-Gordon \cite{i1,4,2405}.

In classical nonlinear partial differential equations, the state of a point is directly affected by its nearest points. However, in many practical physical phenomena, the state of one point can be affected by points farther away, which is called the non-locality and remote correlation. Considering that fractional operators have nonlocality and memory properties, fractional models have great advantages in describing the above phenomena. Some studies have shown that space fractional models can relatively accurately describe complex phenomena such as anomalous diffusion transport, remote interactions and fractal dispersion \cite{24010, 2406}. In this view, more and more classical models have been expressed in fractional systems \cite{i2,8,7,5}. The space fractional sine-Gordon equation is an extension of the classical sine-Gordon equation, which is an important dynamic model with remote interactions in nonlinear science \cite{16,9001,9002}. For instance,  Korabel et al. \cite{9001} discovered  the soliton-like and breather-like patterns of the fractional sine-Gordon equation. Mac\'ias-D\'iaz \cite{9002} numerically proved the existence of nonlinear supertransmission in the space fractional sine-Gordon system.

%Resently, in order to describe wave propagation in complex media well, many partial differential equations  in the study of wave phenomena are extended to consider non-local effects \cite{2406}.
%This is mainly because fractional models can effectively simulate various physical phenomena and  especially have stronger advantages in engineering processes, biological system, abnormal diffusion and non-exponential relaxation models.

In this paper, we consider the following high-dimensional nonlinear space fractional sine-Gordon equation (NSFSGE)
\begin{equation}
\left\{\begin{array}{l}
\partial_{t t} \psi(\mathbf{x}, t)+(-\Delta)^{\alpha/2} \psi(\mathbf{x}, t)+\sin (\psi(\mathbf{x}, t))=0, \quad \mathbf{x} \in \Omega, \quad t>0, \\
\psi(\mathbf{x}, 0)=\varepsilon \psi_0(\mathbf{x}), \quad \partial_t \psi(\mathbf{x}, 0)=\varepsilon \psi_1(\mathbf{x}), \quad \mathbf{x} \in \bar{\Omega},
\end{array}\right.
\label{a0}
\end{equation}
with periodic boundary conditions. $\psi:=\psi(\mathbf{x}, t)$ is a real-valued function, $\mathbf{x}$ is the spatial coordinate, $t$ is the time variable, and $\Omega=\prod_{i=1}^{d}(a_i,b_i)\subset \mathbb{R}^d (d=2,3)$ is a compact domain.  $\psi_0(\mathbf{x})$, $\psi_1(\mathbf{x})$ are two known real initial functions  independent of $\varepsilon$. $\varepsilon \in(0,1]$ is a dimensionless parameter to  depict the nonlinearity strength of Eq. (\ref{a0}).  $(-\Delta)^{^{\alpha/2} }(1<\alpha \leq 2)$ is the space fractional Laplacian, which is defined by \cite{901,902}
$$
(-\Delta)^{^{\alpha/2} } \psi(\mathbf{x}, t)=\mathscr{F}^{-1}\left(|\bm{\omega}|^\alpha \widehat{\psi}(\bm{\omega}, t)\right), \quad \widehat{\psi}(\bm{\omega}, t)=\mathscr{F} \psi(\bm{\omega}, t),
$$
where $\bm{\omega}$ is the $d$-dimensional vector, $\mathscr{F}(\psi)$ denotes the Fourier transform of $z$ and  $\mathscr{F}^{-1}$ stands for its inverse transform. Considering  that $\Omega$ is a bounded domain, the fractional Laplace operator  $(-\Delta)^{^{\alpha/2} }(1<\alpha \leq 2)$  can be  defined by finite Fourier series as \cite{tg1, 15}
$$
(-\Delta)^{^{\alpha/2} } \psi(\mathbf{x}, t)=\sum_{\mathbf{k} \in \mathbb{Z}^d}|\mu\mathbf{k}|^\alpha\widehat{\psi}_{\mathbf{k}} \mathrm{e}^{i \mu\mathbf{k}(\mathbf{x}-a)},$$
here  $i$ is the imaginary unit, $|\mu\mathbf{k}|^\alpha=(\sum_{i=1}^{d}\left|\mu_i\right|^2)^{\alpha/2}$, $\mu\mathbf{k}(\mathbf{x}-a)=\sum_{i=1}^{d}\mu_i(x_i-a_i)$, $\mu_i= \frac{2\pi k_i}{b_i-a_i}$, where $x_i$, $k_i$ represent the component of $\mathbf{x}$, $\mathbf{k}$, respectively. The Fourier coefficients can be given as
$$
\widehat{\psi}_{\mathbf{k}}=\frac{1}{\prod_{i=1}^{d}(b_i-a_i)}\int_{\Omega}\psi(\mathbf{x}, t)\mathrm{e}^{-i \mu\mathbf{k}(\mathbf{x}-a)}d\mathbf{x}.
$$
When $\alpha=2$, the NSFSGE (\ref{a0}) is reduced to the classical sine-Gordon equation.

When $0<\varepsilon\ll 1$, we introduce $ u(\mathbf{x}, t)= \psi(\mathbf{x}, t)/\varepsilon$. The NSFSGE (\ref{a0}) with $O(\varepsilon)$ initial data and $O(1)$ nonlinearity can be turned  to the following NSFSGE: %with $O(1)$ initial data and $O(\varepsilon)$ nonlinearity:
\begin{equation}
\left\{\begin{array}{l}
\partial_{t t} u(\mathbf{x}, t)+(-\Delta)^{\alpha/{2}} u(\mathbf{x}, t)+\frac{1}{\varepsilon} \sin (\varepsilon u(\mathbf{x}, t))=0, \quad \mathbf{x} \in \Omega, \quad t>0, \\
%u(a,t)=u(b,t), \partial_{x}u(a,t)=\partial_{x}u(b,t),\\
u(\mathbf{x}, 0)=u_0(\mathbf{x}), \quad \partial_t u(\mathbf{x}, 0)=u_1(\mathbf{x}), \quad \mathbf{x} \in\bar{\Omega}.
\end{array}\right.
\label{a1}
\end{equation}
where $u_0(\mathbf{x})=\psi_0(\mathbf{x})$, $u_1(\mathbf{x})=\psi_1(\mathbf{x})$. For  Eq. (\ref{a1}), we use the Taylor expansion $\sin(u)=u-\frac{u^3}{6}+O(u^5)$ and give the leading order behevior of the solution is
\begin{equation}\label{a2}
\partial_{t t} u(\mathbf{x}, t)+(-\Delta)^{\alpha/{2}} u(\mathbf{x}, t)+u(\mathbf{x}, t)-\frac{\varepsilon^2 u^3(\mathbf{x}, t)}{6}=0,
\end{equation}
which shows the lifespan of Eq. (\ref{a1}) is at least up to $O(1/\varepsilon^2)$ according to  the research for the nonlinear Klein-Gordon equation provided with the cubic nonlinearlity \cite{903,904}. Note that the long-time dynamics of the NSFSGE (\ref{a0}) is equivalent to that of the  NSFSGE (\ref{a1}). Then we only discuss the error estimation of the long-time dynamics for the NSFSGE (\ref{a1}) with weak nonlinearity at time $O(1/\varepsilon^2)$ and the relevant conclusions can be easily generalized to Eq. (\ref{a0}).

When $\varepsilon=1$, the NSFSGE (\ref{a1}) is  in the standard classical regime and there are many  numerical studies  for it \cite{16,15, 10, 12, 14}.  Ran and Zhang \cite{12} proposed a compact difference scheme with convergence accuracy of fourth-order in space and second-order in time to solve the NSFSGE. Alfimov et al. \cite{14} numerically explored the breather-like solution of the NSFSGE by using the rotating wave approximation method.  A dissipation-preserving Fourier pseudo-spectral method was given for the NSFSGE with damping in  \cite{15}. Fu et al. \cite{16} proposed a linearly implicit structure-preserving numerical scheme for the NSFSGE, in which the stability and convergence in the maximum norm of the numerical scheme were given.   However, those numerical method and error estimation are generally effective up to the time at $O(1)$. When $0<\varepsilon\ll1$, the lifespan of the solution for the NSFSGE (\ref{a1}) is up to the time at $O(1/\varepsilon^2)$. An important work is to extend those classical error bounds for the NSFSGE (\ref{a1}) up to the time at $O(1/\varepsilon^2)$ instead of $O(1)$, i.e., the long-time error analysis.

%%%The authors \cite{14} numerically explored the breather-like solution of fractional sine-Gordon equation by using the rotating wave approximation method. Dehghan et al. \cite{17} developed the meshless method based on the radial basis functions and collocation approach for solving the time fractional nonlinear sine-Gordon.
%%%%%However, in long-time simulations, the errors generated by low order time discretization schemes will gradually accumulate and become larger over time, resulting in inaccurate numerical approximation.

The long-time dynamics of the classical  nonlinear equation with weak nonlinearity has attracted much attention \cite{2100, 19, 48, 18, 21,20}.  Cohen et al. \cite{18} used a modulated Fourier expansion  to carry out  the long-time analysis of the nonlinearly perturbed wave equations. Bao et al. \cite{19} established improved uniform error bounds on a second-order time-splitting method for the nonlinear Schr\"odinger equation up to time at $O(1/\varepsilon^2)$ by introducing a new technique of regularity compensation oscillation. The paper \cite{21}  presented the long time error analysis of the fourth-order compact finite difference methods for the nonlinear Klein-Gordon equation.  An exponential wave integrator Fourier pseudo-spectral method for  the long-time dynamics of the nonlinear Klein-Gordon  equation was developed in \cite{20}. As far as we know, the existing research on long-time dynamics is focused on integer  equations and there is  limited research on long-time dynamics of  nonlinear fractional equations.
For the long-time dynamics of nonlinear fractional  wave equations, the behaviors of the plane wave solutions are significantly different from that of classical equations and the fractional order can  affect  the width and height of the solution \cite{2101,2102}. So it is  an interesting and challenging research topic for proposing an  effective numerical method and  error estimation for long-time nonlinear fractional models.
%%Cohen et al. \cite{18} used a modulated Fourier expansion  to carry out  the long-time analysis of the nonlinearly perturbed wave equations. Bao et al. \cite{19} introduced a new technique of regularity compensation oscillation  to get the improved uniform error bounds for the time-splitting methods for the long-time dynamics of the nonlinear Schr\"odinger equation.  Bao et al. \cite{1801} established the uniform error bounds of time-splitting spectral methods for the long-time dynamics of the nonlinear Klein-Gordon equation, which gives the explicit dependency between errors and $\varepsilon$.

The aim of this paper is to obtain the improved uniform error bounds  in $H^{\alpha/2}$-norm  of the high-dimensional NSFSGE (\ref{a1}) up to the time at
$T_{\varepsilon}=T/\varepsilon^2$ with $T$ fixed. The inherent differences in dynamic behaviors between the fractional sine-Gordon equation and classical  sine-Gordon equation are discussed. The difficulties are the numerical analysis of fractional operators and how to give the explicit dependence of the errors and the parameter $\varepsilon$ in fractional equations. The main contributions of this paper can be listed as

$\bullet$ To obtain the uniform error bounds for the long-time dynamics,  one of the tricks used here is that a linear part from the sine function  of the NSFSGE is separated. Then we transform the NSFSGE to an equivalent relativistic nonlinear space fractional Schr\"odinger equation (NSFSE). With the help of the Strang splitting technique \cite{22,23}, we decompose the NSFSE into a linear equation and a nonlinear equation, and the second-order time-splitting  Fourier pseudo-spectral  method is developed.
%A time-splitting Fourier pseudo-spectral  method  are developed to solve the NSFSGE, which can obtain the second-order convergence accuracy in time and  the spectral convergence accuracy in space.

$\bullet$ We introduce the  regularity compensation oscillation technique  to the convergence analysis of fractional models, which the high-frequency modes are controlled by regularity and the low-frequency modes are analyzed by phase cancellation as well as energy method, then we give the improved uniform error bounds  $O\left(\varepsilon^2 \tau^2\right)$ for the semi-discretization scheme and $O\left(h^m+\varepsilon^2 \tau^2\right)$ for the full-discretization scheme. The explicit relation of the  errors and the parameter $\varepsilon$ is specified.

$\bullet$ Compared with the existing references \cite{48, 49},  we prove the error bounds for the fully discretization directly by the mathematical induction  without comparing with the semi-discretization in time, so the bound of the semi-discrete numerical solution $\lVert\varphi^{[n]}\rVert_{H^{m+\alpha/2}}$  is not needed in this paper. %This proof is different form \cite{48, 49}.
%In\cite{48, 49}, the improved error bounds for the fully discretization are mainly based on the error bounds of the semi-discretization and the bound of $\lVert\varphi^{[n]}\rVert_{H^{m+\alpha/2}}$ is necessary to obtain the space order $h^m$. However, the bound of $\lVert\varphi^{[n]}\rVert_{H^{m+\alpha/2}}$ is not available for the nonlinear equation. Here, we prove the error bounds for the fully discretization directly by the mathematical induction, in which the bound of $\lVert\varphi^{[n]}\rVert_{H^{m+\alpha/2}}$ is not needed.

%$\bullet$  We extend the time-splitting Fourier pseudo-spectral method and the improved uniform error bounds  to the complex NSFSGE and the oscillatory complex  NSFSGE. Some numerical examples in two-dimension or three-dimension  to support theoretical analysis and the differences of in dynamic behaviors between the fractional sine-Gordon equation and classical  sine-Gordon equation are discussed.

The organization of this paper is as follows. In Section \ref{sec:2}, we give some notation and use a second-order time-splitting  Fourier pseudo-spectral method to obtain the numerical scheme of the NSFSGE (\ref{a1}). Section \ref{sec:3} gives the improved uniform error bounds for the semi-discretization scheme and full discretization scheme by  the regularity compensation oscillation  technique, respectively. In Section \ref{sec:4}, we extend the time-splitting Fourier pseudo-spectral method and corresponding error bounds to the complex NSFSGE as well as the oscillatory complex NSFSGE. Section \ref{sec:5}  verifies the validity of the numerical scheme and error bounds through some specific examples. Finally, we give some conclusions in Section \ref{sec:6}.
\section{Numerical method}
\label{sec:2}
For convenience, we only give the numerical scheme and error estimation
for the NSFSGE (\ref{a1}) in the two-dimension (2D). It is similar for the NSFSGE (\ref{a1}) in the three-dimension (3D). In 2D, we further assume $\Omega=(a,b)^2$ for simplification,  Eq. (\ref{a1}) with periodic boundary conditions can be expressed  as
\begin{equation}
\left\{\begin{array}{l}
\partial_{t t} u(x, y, t)+(-\Delta)^{\alpha/2} u(x, y, t)+\frac{1}{\varepsilon} \sin (\varepsilon u(x, y, t))=0, \quad (x,y) \in \Omega, \quad t>0,\\
u(x, y, 0)=u_0(x, y), \quad \partial_t u(x, y, 0)=u_1(x, y), \quad (x,y) \in \bar{\Omega}.
\end{array}\right.
\label{b1}
\end{equation}
%In this section, we adopt the time-splitting method  in time and Fourier pseudo-spectral method  in space to discretize Eq.(\ref{b1}).
 In the remainder of this paper, let $C>0$ denote positive constants independent of the time step $\tau$, the space step $h$ and $\varepsilon$. The notation $A\lesssim B$ is adopted to denote $|A|\le CB$.
\subsection{Preliminary}
%Let $\tau>0$ be the time step and $h=(b-a) / N$ be the space step, where $N$ is an even positive integer. The grid points are denoted as
%$$
%t_n:=n \tau, \quad n=0,1,2, \cdots,
%$$
%$$
%x_p:=a+p h, \quad y_q:=a+q h, \quad p,q=0,1,2, \cdots, N.
%$$
Let $L^2(\Omega)$ be the standard $L^2$-space, $C_{p e r}(\Omega)=\{u \in C(\bar{\Omega}) \mid \text{$u$ is periodic} \}$  and
$$
 Y_N:=\operatorname{span}\left\{\mathrm{e}^{i \mu_{k}(x-a)+i\mu_{l}(y-a)}: \mu_j=\frac{2 \pi j}{b-a}, k,l\in T_N\right\}, $$
 $$
 T_N=\left\{j\mid j=-\frac{N}{2},-\frac{N}{2}+1, \ldots, \frac{N}{2}-1 \right\},
 $$
here $N$ is an even positive integer. Denote $P_N: L^2(\Omega) \rightarrow Y_N$ as the standard $L^2$-projection operator, $I_N: C_{p e r}(\Omega) \rightarrow Y_N$ as the trigonometric interpolation operator, i.e.,
$$
P_N u(x, y)=\sum_{k,l \in T_N} \widehat{u}_{kl} \mathrm{e}^{i \mu_{k}(x-a)+i\mu_{l}(y-a)}, \quad I_N u(x, y)=\sum_{k,l \in T_N} \tilde{u}_{kl} \mathrm{e}^{i \mu_{k}(x-a)+i\mu_{l}(y-a)}, \quad (x, y) \in \bar{\Omega},
$$
where
$$
\widehat{u}_{kl}=\frac{1}{(b-a)^2} \int_{\Omega} u(x,y) \mathrm{e}^{-i \mu_k(x-a)-i \mu_l(y-a)} {d} {\Omega}, \quad \tilde{u}_{kl}=\frac{1}{N^2} \sum_{p=0}^{N-1}\sum_{q=0}^{N-1} u_{pq} \mathrm{e}^{-i \mu_k\left(x_p-a\right)-i \mu_l\left(y_q-a\right)},
$$
and $u_{pq}$  is defined as $u\left(x_p, y_q\right)$.

For $m \geq 0, H^m(\Omega)$ is the space of functions $u(x, y) \in L^2(\Omega)$, the $H^m$-norm $\|\cdot\|_m$ is defined by
$$
\|u\|_m^2=\sum_{k,l \in \mathbb{Z}}\left(1+\left|\mu_k\right|^2+\left|\mu_l\right|^2\right)^m\left|\widehat{u}_{kl}\right|^2, \quad \text { for } \quad u(x, y)=\sum_{k,l \in \mathbb{Z}} \widehat{u}_{kl} \mathrm{e}^{i \mu_k(x-a)+i \mu_l(y-a)},
$$
where $\widehat{u}_{kl}(k,l \in \mathbb{Z})$ denotes the Fourier coefficients of $u(x,y)$. Furthermore, $H^0(\Omega)=$ $L^2(\Omega)$ and  $\|\cdot\|$ denotes the $L^2-$norm $\|\cdot\|_0$. We denote $L^{\infty}([0,T];H^m(\Omega))$ with $T>0$ as the space of the functions $u:[0,T]\to H^m(\Omega)$ such that
%\begin{equation*}
%L^{\infty}(0,T;H^m(\Omega))=\bigg\{u:(0,T)\to H^m(\Omega) \bigg| \lVert v\rVert_{L^{2}(Y)}= \big( \int^T_0 \lVert v\rVert^2_Y dt\big)^{\frac{1}{2}}<+\infty.\bigg\},
%\end{equation*}
%and is endowed with the norm
$$
\lVert u\rVert_{L^{\infty}([0,T];H^m(\Omega))}= \mathop{esssup}\limits_{0\le t\le T}\lVert u(t)\rVert_m<+\infty.
$$

Define the operator $\langle\nabla\rangle_\alpha=\sqrt{1+(-\Delta)^{{\alpha}/{2}}}$ as
$$
\langle\nabla\rangle_\alpha u(x, y)=\sum_{k,l \in \mathbb{Z}} \sqrt{1+\left(\left|\mu_k\right|^2+\left|\mu_l\right|^2\right)^{\alpha/2}} \widehat{u}_{kl} \mathrm{e}^{i \mu_k(x-a)+i \mu_l(y-a)}
$$
and the inverse operator $\langle\nabla\rangle_\alpha^{-1}$ as
$$
\langle\nabla\rangle_\alpha^{-1} u(x, y)=\sum_{k, l \in \mathbb{Z}} \frac{\widehat{u}_{kl}}{\sqrt{1+\left(\left|\mu_k\right|^2+\left|\mu_l\right|^2\right)^{\alpha/2}}} \mathrm{e}^{i \mu_k(x-a)+i \mu_l(y-a)}.
$$
It is easy to get
$$
\left\|\langle\nabla\rangle_\alpha^{-1} u\right\|_s \lesssim\|u\|_{s-\alpha / 2} \lesssim\|u\|_s, \quad\left\|\langle\nabla\rangle_\alpha u\right\|_s \lesssim\|u\|_{s+\alpha / 2} \lesssim\|u\|_{s+1} .
$$

For Eq. (\ref{b1}),  we  first separate a linear part $u(x,y,t)$ from $\frac{1}{\varepsilon}\sin({\varepsilon u(x,y,t)})$  and get
\begin{equation}
\partial_{t t} u(x,y, t)+\langle\nabla\rangle_{{\alpha}}^2 u(x,y, t)+\frac{1}{\varepsilon} \sin (\varepsilon u(x,y, t))-u(x,y,t)=0, \quad (x,y )\in \Omega, \quad t>0.
\label{b2}
\end{equation}
Introducing $v(x,y, t)=\partial_t u(x,y, t)$ and $\varphi(x,y, t)=u(x,y, t)-i\langle\nabla\rangle_\alpha^{-1} v(x,y, t)$, we can get $\partial_t v(x,y, t)=i\langle\nabla\rangle_\alpha\partial_t \varphi(x,y, t)-i\langle\nabla\rangle_\alpha v(x,y, t)$ and $u(x,y, t)=\varphi(x,y, t)-i\langle\nabla\rangle_\alpha^{-1} v(x,y, t)$. Then Eq. (\ref{b2}) is equivalent to the following nonlinear space fractional Schr\"odinger equation (NSFSE) with periodic boundary conditions:
\begin{equation}
\left\{\begin{array}{lll}
i \partial_t \varphi(x,y, t)+\langle\nabla\rangle_\alpha \varphi(x,y, t)+\langle\nabla\rangle_\alpha^{-1}f(\frac{\varphi+\bar{\varphi}}{2})(x,y, t)=0, & (x,y) \in \Omega, \quad t>0, \\
%\varphi(a,y, t)=\varphi(b,y, t),\quad \partial_x \varphi(a, y, t)=\partial_x \varphi(b,y, t),  \quad t > 0,\\
%\varphi(x,a, t)=\varphi(x,b, t),\quad \partial_y \varphi(x, a, t)=\partial_x \varphi(x,b, t), \quad t > 0,\\
\varphi(x,y, 0)=\varphi_0(x,y):=u_0(x,y)-i\langle\nabla\rangle^{-1} u_1(x,y), \quad (x,y) \in \bar{\Omega},
\end{array}\right.
\label{b3}
\end{equation}
where $f(u)=\frac{1}{\varepsilon} \sin (\varepsilon u(x,y, t))-u(x,y,t)$, $\bar{\varphi}$ is the complex conjugate of $\varphi$. Note that the solution of Eq. (\ref{b1}) is
\begin{equation}
u(x,y, t)=\frac{1}{2}(\varphi(x,y, t)+\bar{\varphi}(x,y, t)), \quad v(x,y, t)=\frac{i}{2}\langle\nabla\rangle_\alpha(\varphi(x,y, t)-\bar{\varphi}(x,y, t)) .
\label{b31}
\end{equation}

\subsection{Numerical scheme}

%Denote $\varphi(t):=\varphi(x, y, t)$, the exact solution to (\ref{b3}) can be obtained by using the variation-of-constants formula,
%\begin{equation}
%\varphi\left(t_n+\tau\right)=\mathrm{e}^{i \tau\langle\nabla\rangle \alpha} \varphi\left(t_n\right)+ \int_0^\tau \mathrm{e}^{i(\tau-s)\langle\nabla\rangle_\alpha} F\left(\varphi\left(t_n+s\right)\right) \mathrm{d} s,
%\label{b4}
%\end{equation}
%where the nonlinear operator $F$ has the following expression
%$$F(\varphi)=i\langle\nabla\rangle_\alpha^{-1}f(\frac{\varphi+\bar{\varphi}}{2}),\quad f(\frac{\varphi+\bar{\varphi}}{2})=\frac{1}{\varepsilon}\sin(\frac{\varepsilon(\varphi+\bar{\varphi})}{2})-\frac{\varphi+\bar{\varphi}}{2}.$$

By the Strang splitting technique, the relativistic NSFSE (\ref{b3}) can be  decomposed into a linear part and a nonlinear part.  The linear part is
\begin{equation}
\left\{\begin{array}{l}
\partial_t \varphi(x,y, t)=i\langle\nabla\rangle_\alpha \varphi(x,y, t),\\
\varphi(x,y, 0)=\varphi_0(x,y),
\end{array}\right.
\label{b41}
\end{equation}
and  it can be solved exactly in phase space  as
\begin{equation}
\varphi(x, y, t)=\phi_T^t(\varphi_0)=\mathrm{e}^{it\langle\nabla\rangle_\alpha}\varphi_0(x, y),\quad t\ge 0.
\label{b5}
\end{equation}
The nonlinear part is
\begin{equation}
\left\{\begin{array}{l}
\partial_t \varphi(x, y, t)=F(\varphi(x,y, t)),\\
\varphi(x,y, 0)=\varphi_0(x,y),
\end{array}\right.
\label{b42}
\end{equation}
where the nonlinear operator $F$ has the following expression
$$F(\varphi)=i\langle\nabla\rangle_\alpha^{-1}f(\frac{\varphi+\bar{\varphi}}{2}),\quad f(\frac{\varphi+\bar{\varphi}}{2})=\frac{1}{\varepsilon}\sin(\frac{\varepsilon(\varphi+\bar{\varphi})}{2})-\frac{\varphi+\bar{\varphi}}{2}.$$  Then (\ref{b42}) can be integrated exactly in time as
\begin{equation}
\varphi(x,y, t)=\phi_V^t(\varphi_0)=\varphi_0(x,y)+tF(\varphi_0(x, y)),\quad t\ge 0.
\label{b6}
\end{equation}

 %The the evolution operator for the linear part $\partial_t \varphi(x,y, t)=i\langle\nabla\rangle_\alpha \varphi(x,y, t)$  with initial data $\varphi(x,y, 0)=\varphi_0(x,y)$ is given by
%\begin{equation}
%\varphi(x,y, t)=e^{it\langle\nabla\rangle_\alpha}\varphi_0(x,y),\quad t\ge 0,
%\label{b5}
%\end{equation}
%and the nonlinear part $\partial_t \varphi(x, y, t)=F(\varphi(x,y, t))$ with initial data $\varphi(x,y, 0)=\varphi_0(x,y)$ can be integrated in time as
%\begin{equation}
%\varphi(x,y, t)=\varphi_0(x,y)+tF(\varphi_0(x, y)),\quad t\ge 0.
%\label{b6}
%\end{equation}

Let $\tau>0$ denote the time step  and $t_n:=n \tau, n=0,1,2, \cdots.$
We define $\varphi^{[n]}:=\varphi^{[n]}(x,y)$ as the time numerical approximation of $\varphi\left(x, y, t_n\right)$.  Based on the Strang splitting, the second-order time-splitting method for the relativistic NSFSE (\ref{b3}) can be given as
\begin{equation}
\varphi^{[n+1]}=S_{\tau}(\varphi^{[n]})=\phi_T^{\frac{\tau}{2}}\circ\phi_V^{\tau}\circ\phi^{\frac{\tau}{2}}_T(\varphi^{[n]})=\mathrm{e}^{i \tau\langle\nabla\rangle_{\alpha}} \varphi^{[n]}+ \tau \mathrm{e}^{i \frac{\tau\langle\nabla\rangle_{\alpha}}{2}} F\left(\mathrm{e}^{i \frac{\tau\langle\nabla\rangle_{\alpha}}{2}} \varphi^{[n]}\right),
\label{b7}
\end{equation}
and $\varphi^{[0]}=\varphi_0=u_0-i\langle\nabla\rangle_\alpha^{-1} u_1$. Hence the time discrete scheme for Eq. (\ref{a1}) is
\begin{equation}
u^{[n+1]}=\frac{1}{2}\left(\varphi^{[n+1]}+\overline{\varphi^{[n+1]}}\right), \quad v^{[n+1]}=\frac{i}{2}\langle\nabla\rangle_\alpha\left(\varphi^{[n+1]}-\overline{\varphi^{[n+1]}}\right), \quad n=0,1, \ldots,
\label{b8}
\end{equation}
with $u^{[0]}=u_0(x, y)$ and $v^{[0]}=u_1(x, y)$. Here, $u^{[n]}:=u^{[n]}(x, y)$ and $v^{[n]}:=v^{[n]}(x, y)$ are the numerical approximations of $u\left(x, y, t_n\right)$ and $v\left(x, y, t_n\right)$, respectively.

Next, we use the Fourier pseudo-spectral method \cite{24} for the spatial approximation. Let  $h=(b-a) / N$ be the space step. The grid points are denoted as $ x_p:=a+p h, \quad y_q:=a+q h, \quad p,q=0,1,2, \cdots, N.$  %%%and  $\varphi^n=\left(\varphi_0^n, \varphi_1^n, \ldots, \varphi_N^n\right)^T \in \mathbb{C}^{N+1}$
Denote $\varphi_{pq}^n$ as the numerical approximation of $\varphi\left(x_p, y_q, t_n\right)$,   $\varphi^n=\left(\varphi_{00}^n, \varphi_{10}^n,\ldots, \varphi_{N0}^n, \ldots, \varphi_{0N}^n, \ldots, \varphi_{NN}^n, \right)^T \in \mathbb{C}^{(N+1)\times (N+1)}$, where $\varphi_{0q}^n=\varphi_{Nq}^n=\varphi_{p0}^n=\varphi_{pN}^n$, then a
time-splitting Fourier pseudo-spectral (TSFP) method for (\ref{b3}) is
\begin{equation}\label{b9}
\begin{split}
\varphi_{pq}^{(1)}&=\sum_{k,l \in T_N} \mathrm{e}^{i \frac{\tau \delta_{kl}}{2}} (\widetilde{\varphi^n})_{kl} \mathrm{e}^{i \mu_k\left(x_p-a\right)+
i \mu_l\left(y_q-a\right)},\\
\varphi_{pq}^{(2)}&=\varphi_{pq}^{(1)}+\tau F_{pq}^n, \quad
F_{pq}^n=\sum_{k,l \in T_N}\frac{i}{\delta_{kl}}\left(\widetilde{f(\frac{\varphi^{(1)}+\overline{\varphi^{(1)}}}{2})}\right)_{kl}\mathrm{e}^{i \mu_k\left(x_p-a\right)+
i \mu_l\left(y_q-a\right)},\\
\varphi_{pq}^{n+1}&=\sum_{k,l \in T_N} \mathrm{e}^{i \frac{\tau \delta_{kl}}{2}} (\widetilde{\varphi^{(2)}})_{kl} \mathrm{e}^{i \mu_k\left(x_p-a\right)+
i \mu_l\left(y_q-a\right)},
\end{split}
\end{equation}
with $\varphi_{pq}^0=u_0\left(x_p,y_q\right)-i \sum\limits_{k,l \in T_N} \frac{\widetilde{\left(u_1\right)_{kl}}}{\delta_{kl}} \mathrm{e}^{i \mu_k\left(x_p-a\right)+i \mu_l\left(y_q-a\right)}$ and $\delta_{kl}=\sqrt{1+\left(\mu_k^2+\mu_l^2\right)^{\alpha/2}}$.

%the fully discrete scheme for (\ref{b3}) is
%\begin{equation}
%\varphi_{pq}^{n+1}=\sum_{k,l \in T_N} \mathrm{e}^{i \tau \delta_{kl}} \widetilde{\varphi}_{kl}^n \mathrm{e}^{i \mu_k\left(x_p-a\right)+
%i \mu_l\left(y_q-a\right)}+\tau \sum_{k,l \in T_N}
%\mathrm{e}^{i \frac{\tau\delta_{kl}}{2}}\widetilde{F}_{kl}^n\mathrm{e}^{i \mu_k\left(x_p-a\right)+i \mu_l\left(y_q-a\right)},
%\label{b9}
%\end{equation}
%with $\varphi_{pq}^0=z_0\left(x_p,y_q\right)-i \sum\limits_{k,l \in T_N} \frac{\widetilde{\left(z_1\right)_{kl}}}{\delta_{kl}} \mathrm{e}^{i \mu_k\left(x_p-a\right)+i \mu_l\left(y_q-a\right)}$, where
%$$
%\delta_{kl}=\sqrt{1+\left(\left|\mu_k\right|^2+\left|\mu_l\right|^2\right)^{\alpha/2}}, \quad \widetilde{F_{kl}^{n}}=\frac{i}{\delta_{kl}}\left(\widetilde{f(\frac{\varphi_{pq}+\overline{\varphi_{pq}}}{2})}\right)_{kl}.
%$$

%Denote $u^n=\left(u_0^n, u_1^n, \ldots, u_N^n\right)^T \in \mathbb{R}^{N+1}$ and $v^n=$ $\left(v_0^n, v_1^n, \ldots, v_N^n\right)^T \in \mathbb{R}^{N+1}$.
Combining $(\ref{b31})$ and $(\ref{b9})$, the fully discrete scheme for (\ref{b1}) by the TSFP method is
\begin{equation}\label{b10}
\begin{split}
u_{pq}^{n+1}&=\frac{1}{2}\left(\varphi_{pq}^{n+1}+\overline{\varphi_{pq}^{n+1}}\right), \quad  n=0,1, \ldots,\quad  p,q=0,1, \cdots, N,
\\
v_{pq}^{n+1}&=\frac{i}{2} \sum_{k,l \in T_N} \delta_{kl}\left[(\widetilde{\varphi^{n+1}})_{kl}-(\widetilde{\overline{{\varphi^{n+1}}}})_{kl}\right] \mathrm{e}^{i \mu_k\left(x_{p}-a\right)+i \mu_l\left(y_{q}-a\right)},
\end{split}
\end{equation}
with $u_{pq}^0=u_0(x_p, y_q)$ and $v_{pq}^0=u_1(x_p, y_q)$. Here, $u_{pq}^n$ and $v_{pq}^n$ are the numerical approximations of $u\left(x_{p},y_q, t_n\right)$ and $v\left(x_{p}, y_q, t_n\right)$ respectively.

\section{Error estimation}
\label{sec:3}
We supposes the exact solution $u:=u(x,y, t)$ of Eq. (\ref{b1}) up to the time at $T_{\varepsilon}=T / \varepsilon^2$ with $T>0$  satisfies the following assumptions:
\begin{equation}\label{c0}
\text{(A)} \quad \|u\|_{L^{\infty}\left(\left[0, T_{\varepsilon}\right] ; H^{m+\alpha/2}\right)} \lesssim 1, \quad\left\|\partial_t u\right\|_{L^{\infty}\left(\left[0, T_{\varepsilon}\right] ; H^m\right)} \lesssim 1, \quad m \geq 0.
\end{equation}
Note that the  assumption (A) is equivalent to the regularity of $\|\varphi\|_{L^{\infty}\left(\left[0, T_{\varepsilon}\right] ; H^{m+\alpha/2}\right)} \lesssim 1$.

Considering that $f(u)=\frac{1}{\varepsilon} \sin (\varepsilon u)-u$, we can write  $f\left(\frac{\varphi+\bar{\varphi}}{2}\right)$ with the $O\left(\varepsilon^2\right)$ dominant term by the Taylor expansion as
\begin{equation}
f\left(\frac{\varphi+\bar{\varphi}}{2}\right)=-\frac{\varepsilon^2}{48}(\varphi+\bar{\varphi})^3+\varepsilon^4 r(\varphi)=: \varepsilon^2 h(\varphi)+\varepsilon^4 r(\varphi).
\label{c1}
\end{equation}
Define the function $H$ as
$$
H(\varphi)=\varepsilon^2 i\langle\nabla\rangle^{-1}_{\alpha} h(\varphi),
$$
and let
\begin{equation}\label{c101}
H_t: \varphi \mapsto  \mathrm{e}^{-i t\langle\nabla\rangle_{\alpha}} H\left( \mathrm{e}^{i t\langle\nabla\rangle_{\alpha}} \varphi\right), \quad t \in \mathbb{R}.
\end{equation}

Next we will give the improved uniform error bounds for the semi-discretization scheme (\ref{b7})-(\ref{b8}) and the full-discretization scheme (\ref{b9})-(\ref{b10}) up to the time at $T_{\varepsilon}$ respectively.
\subsection{Improved uniform error bounds for the semi-discretization scheme}
%Firstly, we have the following estimates for the local truncation error.
\begin{lem}
For $0<\varepsilon \leq 1$, the local error of the  time  semi-discretization scheme (\ref{b7}) by the Strang splitting for the relativistic NSFSE (\ref{b3}) can be obtained as
\begin{equation}
\mathcal{E}^n:=S_{\tau}(\varphi(t_n))-\varphi(t_{n+1}) =\mathcal{H}\left(\varphi\left(t_n\right)\right)+\mathcal{R}^n, \quad n=0,1, \ldots,
\label{c2}
\end{equation}
%\tau e^{i \frac{\tau\langle\nabla\rangle_{\alpha}}{2}} F\left(e^{i \frac{\tau\langle\nabla\rangle_{\alpha}}{2}} \varphi(t_n)\right)-\int_0^\tau \mathrm{e}^{i(\tau-s)\langle\nabla\rangle_\alpha} F\left(\varphi\left(t_n+s\right)\right)ds
with
\begin{equation}
\mathcal{H}\left(\varphi\left(t_n\right)\right)= \mathrm{e}^{i\tau\langle\nabla\rangle_{\alpha}}\left( \tau H_{\frac{\tau}{2}}\left( \varphi(t_n)\right)
-\int_0^\tau H_{s}\left(\varphi\left(t_n\right)\right)ds\right),
\label{c3}
\end{equation}
%\begin{equation}
%\mathcal{H}\left(\varphi\left(t_n\right)\right)=\tau e^{i \frac{\tau\langle\nabla\rangle_{\alpha}}{2}} H\left(e^{i \frac{\tau\langle\nabla\rangle_{\alpha}}{2}} \varphi(t_n)\right)-\int_0^\tau \mathrm{e}^{i(\tau-s)\langle\nabla\rangle_\alpha} H\left(e^{is\langle\nabla\rangle_{\alpha}}\varphi\left(t_n\right)\right)ds,
%\label{c3}
%\end{equation}
then we have the following estimates  under the assumption (A),
\begin{equation}
\left\|\mathcal{H}\left(\varphi\left(t_n\right)\right)\right\|_{\alpha/2} \lesssim \varepsilon^2 \tau^3, \quad\left\|\mathcal{R}^n\right\|_{\alpha/2} \lesssim \varepsilon^4 \tau^3.
\label{c301}
\end{equation}
%where $0<\varepsilon \leq 1$ and  $\varphi(t):=\varphi(x, y, t)$.
\label{l1}
\end{lem}
\begin{proof}
Recalling  the Duhamel's principle,  the exact solution for the relativistic NSFSE (\ref{b3}) can be written as
\begin{equation}
\varphi\left(t_n+\tau\right)=\mathrm{e}^{i \tau\langle\nabla\rangle \alpha} \varphi\left(t_n\right)+ \int_0^\tau \mathrm{e}^{i(\tau-s)\langle\nabla\rangle_\alpha} F\left(\varphi\left(t_n+s\right)\right) {d} s.
\label{c302}
\end{equation}
By (\ref{b7}) and (\ref{c302}), we obtain
\begin{equation}
\begin{split}
\mathcal{E}^n=&\tau  \mathrm{e}^{i \frac{\tau\langle\nabla\rangle_{\alpha}}{2}} F\left( \mathrm{e}^{i \frac{\tau\langle\nabla\rangle_{\alpha}}{2}} \varphi(t_n)\right)-\int_0^\tau \mathrm{e}^{i(\tau-s)\langle\nabla\rangle_\alpha} F\left(\varphi\left(t_n+s\right)\right)ds\\
=&\tau  \mathrm{e}^{i \frac{\tau\langle\nabla\rangle_{\alpha}}{2}} H\left( \mathrm{e}^{i \frac{\tau\langle\nabla\rangle_{\alpha}}{2}} \varphi(t_n)\right)-\int_0^\tau \mathrm{e}^{i(\tau-s)\langle\nabla\rangle_\alpha} H\left(\varphi\left(t_n+s\right)\right)ds\\&+\tau  \mathrm{e}^{i \frac{\tau\langle\nabla\rangle_{\alpha}}{2}}\varepsilon^4 i\langle\nabla\rangle_{\alpha}^{-1}r\left( \mathrm{e}^{i \frac{\tau\langle\nabla\rangle_{\alpha}}{2}} \varphi(t_n)\right)-\int_0^\tau \mathrm{e}^{i(\tau-s)\langle\nabla\rangle_\alpha} \varepsilon^4 i
\langle\nabla\rangle_{\alpha}^{-1} r\left(\varphi\left(t_n+s\right)\right)ds\\
=&\tau  \mathrm{e}^{i \frac{\tau\langle\nabla\rangle_{\alpha}}{2}} H\left( \mathrm{e}^{i \frac{\tau\langle\nabla\rangle_{\alpha}}{2}} \varphi(t_n)\right)-\int_0^\tau \mathrm{e}^{i(\tau-s)\langle\nabla\rangle_\alpha} H\left(\varphi\left(t_n+s\right)\right)ds+\mathcal{R}^n_1\\
=&\tau  \mathrm{e}^{i \frac{\tau\langle\nabla\rangle_{\alpha}}{2}} H\left( \mathrm{e}^{i \frac{\tau\langle\nabla\rangle_{\alpha}}{2}} \varphi(t_n)\right)
-\int_0^\tau \mathrm{e}^{i(\tau-s)\langle\nabla\rangle_\alpha} H\left(\mathrm{e}^{is\langle\nabla\rangle_\alpha} \varphi(t_n)+\int_0^s \mathrm{e}^{i(s-\sigma)\langle\nabla\rangle_\alpha} F\left(\varphi\left(t_n+\sigma\right)\right)d\sigma\right)ds+\mathcal{R}^n_1\\
=&\tau  \mathrm{e}^{i \frac{\tau\langle\nabla\rangle_{\alpha}}{2}} H\left( \mathrm{e}^{i \frac{\tau\langle\nabla\rangle_{\alpha}}{2}} \varphi(t_n)\right)-\int_0^\tau \mathrm{e}^{i(\tau-s)\langle\nabla\rangle_\alpha} H\left( \mathrm{e}^{is\langle\nabla\rangle_{\alpha}}\varphi\left(t_n\right)\right)ds+\mathcal{R}^n_1+\mathcal{R}^n_2\\
=&\mathcal{H}\left(\varphi\left(t_n\right)\right)+\mathcal{R}^n.
\end{split}
\label{c4}
\end{equation}
where $\mathcal{H}\left(\varphi\left(t_n\right)\right)$ is defined in (\ref{c3}) and $\mathcal{R}^n=\mathcal{R}_1^n+\mathcal{R}_2^n$.

Due to that the operator $\mathrm{e}^{i t\langle\nabla\rangle_{\alpha}}$ preserves the  $H^m-$norm($m>0$),   under the assumption (A), we obtain
$$
\left\|\mathcal{H}\left(\varphi\left(t_n\right)\right)\right\|_{\alpha/2} \lesssim \tau^3\left\|\partial_s^2\left( \mathrm{e}^{i (\tau-s)\langle\nabla\rangle_\alpha} H\left( \mathrm{e}^{i s\langle\nabla\rangle_\alpha} \varphi\left(t_n\right)\right)\right)\right\|_{\alpha/2} \lesssim \varepsilon^2 \tau^3\left\|\varphi\left(t_n\right)\right\|_{\alpha/2+1},
$$
and $\left\|\mathcal{R}^n\right\|_{\alpha/2} \lesssim \varepsilon^4 \tau^3$. The proof of the error bounds (\ref{c301}) is finished .
\end{proof}

The improved uniform error bounds for the semi-discretization (\ref{b7})-(\ref{b8}) are given as follows.
\begin{thrm}
For $0<\varepsilon \leq 1$, let $0<\tau_0<1$ be small enough and independent of $\varepsilon$. When $0<\tau<\gamma \frac{\pi(b-a)^{\alpha / 2} \tau_0^{\alpha / 2}}{2 \sqrt{\tau_0^\alpha(b-a)^\alpha+2^{{3\alpha}/2} \pi^\alpha\left(1+\tau_0\right)^\alpha}}$, where $\gamma \in(0,1)$ is a fixed constant, the following improved uniform error bound can be obtained under the assumption (A),
\begin{equation}
\left\|u\left(x, y, t_n\right)-u^{[n]}\right\|_{\alpha / 2}+\left\|\partial_t u\left(x, y, t_n\right)-v^{[n]}\right\| \lesssim \varepsilon^2 \tau^2+\tau_0^{m+\alpha / 2}, \quad 0 \leq n \leq T_{\varepsilon} / \tau.
\label{c5}
\end{equation}
Especially, if the exact solution is sufficiently smooth, e.g., $u, \partial_t u \in H^{\infty}$, the last term $\tau_0^{m+\alpha / 2}$ decays exponentially fast and can be ignored for small enough $\tau_0$, then the improved uniform error bound for small enough $\tau$ is
\begin{equation}
\left\|u\left(x, y, t_n\right)-u^{[n]}\right\|_{\alpha / 2}+\left\|\partial_t u\left(x, y, t_n\right)-v^{[n]}\right\| \lesssim \varepsilon^2 \tau^2, \quad 0 \leq n \leq T_{\varepsilon} / \tau.
\label{c6}
\end{equation}
\label{thrm1}
\end{thrm}
\begin{proof}
Let the numerical error function $e^{[n]}:=e^{[n]}(x, y)\left(n=0,1, \ldots, T_{\varepsilon} / \tau\right)$  be
\begin{equation}
e^{[n]}:=\varphi^{[n]}-\varphi\left(t_n\right) .
\label{c6}
\end{equation}
Based on the assumption (A), we  will prove that there exists $\tau_c>0$ and  $0<\tau<\tau_c$ such that the following estimates hold,
\begin{equation}
\left\|e^{[n]}\right\|_{\alpha/2} \leq C\left(\varepsilon^2 \tau^2+\tau_0^{m+\alpha/2}\right), \quad\left\|\varphi^{[n]}\right\|_{\alpha/2} \leq K+1, \quad 0 \leq n \leq {T_\varepsilon}{/\tau},
\label{c7}
\end{equation}
here $K=\left\|\varphi\right\|_{L^{\infty}([0,T_{\varepsilon}]; H^{\alpha/2})}$ is a constant depending on $T$. The mathematical induction is adopted  to prove the estimates. It is obvious to obtain  the error bound  (\ref{c5}) holds for the case $n=0$ with $\varphi^{[0]}=\varphi_0$. Then we suppose (\ref{c7}) hods for all $0 \leq n \leq p \leq {T_\varepsilon}{/\tau}-1$,  and  we will prove (\ref{c7})  holds for the case $n=p+1$.

From (\ref{b7}) and $(\ref{c2})$, we get the following error equation,
\begin{equation}
e^{[n+1]}=S_{\tau}(\varphi^{[n]})-S_{\tau}(\varphi(t_n))+\mathscr{E}^n=\mathrm{e}^{i \tau\langle\nabla\rangle_\alpha} e^{[n]}+Q^n+\mathscr{E}^n, \quad 0 \leq n \leq T_{\varepsilon} / \tau,
\label{c701}
\end{equation}
where
\begin{equation}
Q^{n}=\tau  \mathrm{e}^{i \frac{\tau\langle\nabla\rangle_{\alpha}}{2}}\left(F\left( \mathrm{e}^{i \frac{\tau\langle\nabla\rangle_{\alpha}}{2}} \varphi^{[n]}\right)-F\left( \mathrm{e}^{i \frac{\tau\langle\nabla\rangle_{\alpha}}{2}} \varphi\left(t_n\right)\right)\right).
\label{c702}
\end{equation}
Then
\begin{equation}
e^{[n+1]}=\mathrm{e}^{i(n+1) \tau\langle\nabla\rangle_\alpha} e^{[0]}+\sum_{j=0}^n \mathrm{e}^{i(n-j) \tau\langle\nabla\rangle _\alpha}\left(Q^j+\mathscr{E}^j\right)=\sum_{j=0}^n \mathrm{e}^{i(n-j) \tau\langle\nabla\rangle_\alpha}\left(Q^j+\mathscr{E}^j\right).
\label{c8}
\end{equation}
By the assumption (A) and (\ref{c7}) for $n \leq p$, we have
\begin{equation}
\left\|Q^n\right\|_{\alpha/2} \lesssim \tau\left\|F\left( \mathrm{e}^{i \frac{\tau\langle\nabla\rangle_{\alpha}}{2}} \varphi^{[n]}\right)-F\left( \mathrm{e}^{i \frac{\tau\langle\nabla\rangle_{\alpha}}{2}} \varphi\left(t_n\right)\right)\right\|_{\alpha/2} \lesssim \varepsilon^2 \tau\left\| {e}^{[n]}\right\|_{\alpha/2}, \quad n \leq p.
\label{c9}
\end{equation}
From (\ref{c301}), Eq. (\ref{c8}) can be written as
\begin{equation}
\left\|e^{[n+1]}\right\|_{\alpha/2} \lesssim \varepsilon^2 \tau^2+\varepsilon^2 \tau \sum_{j=0}^n\left\|e^{[j]}\right\|_{\alpha/2}+\left\|\sum_{j=0}^n  \mathrm{e}^{i(n-j) \tau\langle\nabla\rangle_{\alpha}} \mathcal{H}\left(\varphi\left(t_j\right)\right)\right\|_{\alpha/2}.
\label{c10}
\end{equation}

To analysis the last term on the right hand side (RHS) of (\ref{c10}) and gain the order $O\left(\varepsilon^2\right)$, we introduce the regularity compensation oscillation (RCO)  technique \cite{19}, whose key idea is a summation-by-parts procedure combined with spectrum cutoff and phase cancellation.
%which  the high-frequency modes are controlled by regularity and the low-frequency modes are analyzed by phase cancellation as well as energy method.
Multiplying the last term in (\ref{c10}) by $\mathrm{e}^{i(n-1) \tau\langle\nabla\rangle_\alpha}$, then the last term becomes
$$
\left\|\sum_{j=0}^n \mathrm{e}^{-i(j+1) \tau\langle\nabla\rangle_\alpha} \mathcal{H}\left(\varphi\left(t_j\right)\right)\right\|_{\alpha / 2} .
$$
By (\ref{b3}), we get $\partial_t \varphi(x, y, t)-i\langle\nabla\rangle_\alpha \varphi(x, y, t)=F(\varphi(x, y, t))=O\left(\varepsilon^2\right)$,  then we define the `twisted variable' as
\begin{equation}
\xi(x, y, t)=\mathrm{e}^{-i t\langle\nabla\rangle_\alpha} \varphi(x, y, t), \quad t \geq 0.
\label{c11}
\end{equation}
Note that it satisfies the equation $\partial_t \xi(x, y, t)= \mathrm{e}^{-i t\langle\nabla\rangle_\alpha} F\left(\mathrm{e}^{i t\langle\nabla\rangle \alpha} \xi(x, y, t)\right)$. By the assumption (A), we get $\|\xi\|_{L^{\infty}\left(\left[0, T_{\varepsilon}\right] ; H^{m+\alpha / 2}\right)} \lesssim 1$ and $\left\|\partial_t \xi\right\|_{L^{\infty}\left(\left[0, T_{\varepsilon}\right] ; H^{m+\alpha / 2}\right)} \lesssim \varepsilon^2$ with
\begin{equation}
\left\|\xi\left(t_{n+1}\right)-\xi\left(t_n\right)\right\|_{m+\alpha / 2} \lesssim \varepsilon^2 \tau, \quad 0 \leq n \leq T_{\varepsilon} / \tau-1.
\label{c12}
\end{equation}

\textbf{Step 1.} Choose the cut-off parameter for Fourier modes. Let
$\tau_0 \in(0,1)$  and  $N_0=2\left\lceil 1 / \tau_0\right\rceil \in \mathbb{Z}^{+}$ (where $\lceil\cdot\rceil$ is the ceiling function) with $1 / \tau_0 \leq N_0 / 2<1+1 / \tau_0$, then only the Fourier modes with $-\frac{N_0}{2} \leq l \leq \frac{N_0}{2}-1$ $(|l| \leq \frac{1}{\tau_0}$, low frequency modes) in a spectral projection are considered. Recalling $H_t$ from (\ref{c101}) as
\begin{equation}\label{c121}
H_t(\varphi)
=-\frac{\varepsilon^2i}{48}\mathrm{e}^{-i t\langle\nabla\rangle_\alpha}\langle\nabla\rangle_\alpha^{-1}\left(\mathrm{e}^{i t\langle\nabla\rangle_\alpha} \varphi+\mathrm{e}^{-i t\langle\nabla\rangle _\alpha} \bar{\varphi}\right)^3.
\end{equation}
%\begin{equation}
%H_t(\varphi)=H_t\left(\mathrm{e}^{i t\langle\nabla\rangle_\alpha} \xi\right)=\mathrm{e}^{-i t\langle\nabla\rangle_\alpha}H\left(\mathrm{e}^{i t\langle\nabla\rangle_\alpha} \mathrm{e}^{i t\langle\nabla\rangle_\alpha} \xi\right)
%=-\frac{\varepsilon^2i}{48}\mathrm{e}^{-i t\langle\nabla\rangle_\alpha}\langle\nabla\rangle_\alpha^{-1}\left(\mathrm{e}^{i t\langle\nabla\rangle_\alpha} \xi+\mathrm{e}^{-i t\langle\nabla\rangle _\alpha} \bar{\xi}\right)^3.
%\label{c13}
%\end{equation}
With the assumption (A), the properties of operators $\mathrm{e}^{i t\langle\nabla\rangle_\alpha}$ and $\langle\nabla\rangle_\alpha^{-1}$, we obtain the following estimates
\begin{equation}
\left\|H_t\left(\mathrm{e}^{i t_{j}\langle\nabla\rangle_\alpha} \xi\left(t_{j}\right)\right)\right\|_{m+\alpha} \lesssim \varepsilon^2\left\|\xi\left(t_{j}\right)\right\|_{m+\alpha / 2}^3 \lesssim \varepsilon^2, \quad 0 \leq j \leq T_{\varepsilon} / \tau .
\label{c14}
\end{equation}
Since
\begin{equation}
\begin{split}
\mathcal{H}\left(\varphi\left(t_{j}\right)\right)=&\mathcal{H}\left(\mathrm{e}^{i t_{j}\langle\nabla\rangle_\alpha} \xi\left(t_{j}\right)\right) \\
=&\mathcal{H}\left(\mathrm{e}^{i t_{j}\langle\nabla\rangle_\alpha} \xi\left(t_{j}\right)\right)-\mathcal{H}\left(\mathrm{e}^{i t_{j}\langle\nabla\rangle_\alpha} P_{N_0} \xi\left(t_{j}\right)\right) \\
& +\mathcal{H}\left(\mathrm{e}^{i t_{j}\langle\nabla\rangle_\alpha} P_{N_0} \xi\left(t_{j}\right)\right)-P_{N_0} \mathcal{H}\left(\mathrm{e}^{i t_{j}\langle\nabla\rangle_\alpha}\left(P_{N_0} \xi\left(t_{j}\right)\right)\right) \\
& +P_{N_0} \mathcal{H}\left(\mathrm{e}^{i t_{j}\langle\nabla\rangle_\alpha}\left(P_{N_0} \xi\left(t_{j}\right)\right)\right),
\label{c15}
\end{split}
\end{equation}
%\begin{equation}
%\begin{split}
% \mathcal{H}\left(\varphi\left(t_{k}\right)\right)=&\mathcal{H}\left(\mathrm{e}^{i t_{k}\langle\nabla\rangle_\alpha} \xi\left(t_{k}\right)\right) \\
%=&\mathcal{H}\left(\mathrm{e}^{i t_{k}\langle\nabla\rangle_\alpha} \xi\left(t_{k}\right)\right)-P_{N_0} \mathcal{H}\left(\mathrm{e}^{i t_{k}\langle\nabla\rangle_\alpha} \xi\left(t_{k}\right)\right) \\
%& +P_{N_0} \mathcal{H}\left(\mathrm{e}^{i t_{k}\langle\nabla\rangle_\alpha} \xi\left(t_{k}\right)\right)-P_{N_0} \mathcal{H}\left(\mathrm{e}^{i t_{k}\langle\nabla\rangle_\alpha}\left(P_{N_0} \xi\left(t_{k}\right)\right)\right) \\
%& +P_{N_0} \mathcal{H}\left(\mathrm{e}^{i t_{k}\langle\nabla\rangle_\alpha}\left(P_{N_0} \xi\left(t_{k}\right)\right)\right),
%\label{c15}
%\end{split}
%\end{equation}
by (\ref{c3}), (\ref{c14}) and the properties of projection operator\cite{24}, it holds
\begin{equation}
\left\|\mathcal{H}\left(\mathrm{e}^{i t_{j}\langle\nabla\rangle_\alpha} \xi\left(t_{j}\right)\right)-\mathcal{H}\left(\mathrm{e}^{i t_{j}\langle\nabla\rangle_\alpha} P_{N_0} \xi\left(t_{j}\right)\right)\right\|_{\alpha / 2} \lesssim \varepsilon^2\tau\left\| P_{N_0}\xi(t_j)- \xi(t_j)\right\|_{\alpha/2}\lesssim \varepsilon^2\tau\tau_0^{m+\alpha / 2},
\label{c16}
\end{equation}
and
\begin{equation}
\left\|\mathcal{H}\left(\mathrm{e}^{i t_{j}\langle\nabla\rangle_\alpha} P_{N_0} \xi\left(t_{j}\right)\right)-P_{N_0} \mathcal{H}\left(\mathrm{e}^{i t_{j}\langle\nabla\rangle_\alpha}\left(P_{N_0} \xi\left(t_{j}\right)\right)\right)\right\|_{\alpha / 2} \lesssim \varepsilon^2\tau\tau_0^{m+\alpha / 2}.
\label{c16}
\end{equation}
On the basis of the  above estimates, we have for $n\le p$,
\begin{equation}
\left\|e^{[n+1]}\right\|_{\alpha / 2} \lesssim \tau_0^{m+\alpha / 2}+\varepsilon^2\tau^2+\varepsilon^2 \tau \sum_{j=0}^n\left\|e^{[j]}\right\|_{\alpha / 2}+\left\|R^n\right\|_{\alpha / 2},
\label{c17}
\end{equation}
where
\begin{equation}
R^n=\sum_{j=0}^n  \mathrm{e}^{-i(j+1) \tau\langle\nabla\rangle_{\alpha}} P_{N_0} \mathcal{H}\left( \mathrm{e}^{i t_j\langle\nabla\rangle_{\alpha}}\left(P_{N_0} \xi\left(t_j\right)\right)\right).
\label{c18}
\end{equation}

\textbf{Step 2.} Analyze the low Fourier modes term $R^n$. We present the following decomposition for the nonlinear function $H(\cdot)$,
\begin{equation}
H(\varphi)=\sum_{q=1}^4 H^{q}(\varphi), \quad H^{q}(\varphi)=-\frac{\varepsilon^2}{48} i\langle\nabla\rangle^{-1}_{\alpha} h^{q}(\varphi), \quad q=1,2,3,4,
\label{c19}
\end{equation}
with $ h^1(\varphi)=\varphi^3, h^2(\varphi)=3\bar{\varphi} \varphi^2, h^3(\varphi)=3 \bar{\varphi}^2 \varphi,  h^4(\varphi)= \bar{\varphi}^3$. For $s\in \mathbb{R}$ and $q=1,2,3,4$, introducing $ H_t^q(\varphi(t_j))=\mathrm{e}^{-it\langle\nabla\rangle_{\alpha}} H^q\left(\mathrm{e}^{it\langle\nabla\rangle_{\alpha}}\varphi(t_j)\right)$ and $ \mathcal{H}^q\left(\varphi\left(t_n\right)\right)= \mathrm{e}^{i\tau\langle\nabla\rangle_{\alpha}}\left( \tau H^q_{\frac{\tau}{2}}\left( \varphi(t_n)\right)
-\int_0^\tau H^q_{s}\left(\varphi\left(t_n\right)\right)ds\right)$, we have
\begin{equation}
R^n=\sum^{4}_{q=1}R_q^n, \quad R_q^n=\sum_{j=0}^n \mathrm{e}^{-it_{j+1}\langle\nabla\rangle_{\alpha}} P_{N_0} \mathcal{H}^q\left( \mathrm{e}^{i t_j\langle\nabla\rangle_{\alpha}}\left(P_{N_0} \xi\left(t_j\right)\right)\right).
\label{c20}
\end{equation}
Since the analysis of $R_q^n(q=1,2,3,4)$ are analogous, we just give the case of $R_1^n$ $\left(0 \leq n \leq T_{\varepsilon} / \tau-1\right)$.

Give two index sets  $I_k^{N_0}$, $I_l^{N_0}$ as
\begin{align}
I_k^{N_0}&=\left\{\left(k_1, k_2, k_3\right) \mid k_1+k_2+k_3=l, k_1, k_2, k_3 \in T_{N_0}\right\}, \quad k \in T_{N_0},
\\
I_l^{N_0}&=\left\{\left(l_1, l_2, l_3\right) \mid l_1+l_2+l_3=l, l_1, l_2, l_3 \in T_{N_0}\right\}, \quad l \in T_{N_0},
\label{c21}
\end{align}
the following expansion is obtained based on  $P_{N_0} \varphi\left(t_j\right)=\sum_{k,l \in {T}_{N_0}} \widehat{\varphi}_{kl}\left(t_j\right) e^{i \mu_k(x-a)+i \mu_l(y-a)}$ :
\begin{equation}
\begin{split}
&  \mathrm{e}^{-i t_{j+1}\langle\nabla\rangle_{\alpha}} P_{N_0}\left(\mathrm{e}^{i \tau\langle\nabla\rangle_{\alpha}}H^1_{s}\left( \mathrm{e}^{i
t_j\langle\nabla\rangle_{\alpha}}P_{N_0} \xi\left(t_j\right)\right)\right) \\
& =-\varepsilon^2\sum_{k, l \in {T}_{N_0}} \sum_{\left(k_1, k_2, k_3\right) \in {I}_k^{N_0},\atop \left(l_1, l_2, l_3\right) \in {I}_l^{N_0}} \frac{i}{48 \delta_{kl}} \mathcal{G}_{k,l}^{j}(s)  \mathrm{e}^{i \mu_k(x-a)+i \mu_l(y-a)},
\end{split}
\label{c22}
\end{equation}
where the coefficients $\mathcal{G}_{k,l}^j(s)$ are functions of $s$ defined as
\begin{equation}\label{c23}
\mathcal{G}_{k,l}^j(s)= \mathrm{e}^{-i\left(t_j+s\right)\delta_{kl}^{'}}     \widehat{\xi}_{k_1l_1}\left(t_j\right) \widehat{\xi}_{k_2l_2}\left(t_j\right) \widehat{\xi}_{k_3l_3}\left(t_j\right),
\end{equation}
with $\delta_{kl}^{'}=\delta_{kl}-\delta_{k_1l_1}-\delta_{k_2l_2}-\delta_{k_3l_3}$. Thus, we get
\begin{equation}\label{c24}
R_1^n=-\frac{i \varepsilon^2}{48} \sum_{j=0}^n \sum_{k,l \in {T}_{N_0}} \sum_{\left(k_1, k_2, k_3\right) \in {I}_k^{N_0},\atop \left(l_1, l_2, l_3\right) \in {I}_l^{N_0}} \frac{1}{\delta_{kl}} \Lambda_{k,l}^j  \mathrm{e}^{i \mu_k(x-a)+i \mu_l(y-a)},
\end{equation}
and
\begin{equation}\label{c25}
\Lambda_{k,l}^j  =-\tau \mathcal{G}_{k,l}^j(\tau/2)+\int_0^\tau \mathcal{G}_{k,l}^j(s) ds  =r_{k,l}  \mathrm{e}^{-i t_j \delta_{kl}^{'}} c_{k,l}^j,
\end{equation}
here the coefficients $c_{k,l}^k$ and $r_{k,l}$ are difined by
%\begin{equation}
\begin{align}\label{c26}
c_{k,l}^j&=\widehat{\xi}_{k_1l_1}\left(t_j\right) \widehat{\xi}_{k_2l_2}\left(t_j\right) \widehat{\xi}_{k_3l_3}\left(t_j\right),
\\
\label{c261}
r_{k,l}&=-\tau e^{-i\tau\delta_{kl}^{'}/2}+\int_0^{\tau}e^{-is\delta_{kl}^{'}}ds=O\left(\tau^3(\delta_{kl}^{'})^2\right).
\end{align}
%\end{equation}
We suppose $\delta_{kl}^{'} \neq 0$ $(r_{k,l}=0 ~\text{if} ~\delta_{kl}^{'}=0)$. For $k, l \in T_{N_0}$ and $\left(k_1, k_2, k_3\right) \in I_k^{N_0}, \left(l_1, l_2, l_3\right) \in I_l^{N_0}$, the following inequality holds,
$$
\left|\delta_{kl}^{'}\right| \leq 4 \delta_{\frac{N_0}{2}\frac{N_0}{2}}=4 \sqrt{1+\left(\left|\mu_{\frac{N_0}{2}}\right|^2+\left|\mu_{\frac{N_0}{2}}\right|^2\right)^{\alpha/2}}<4 \sqrt{1+\frac{2^{3\alpha/2} \pi^\alpha\left(1+\tau_0\right)^\alpha}{\tau_0^\alpha(b-a)^\alpha}},
$$
which shows if $0<\tau\le\gamma\frac{\pi(b-a)^{\alpha/2}\tau_0^{\alpha/2}}{2\sqrt{\tau_0^{\alpha}(b-a)^{\alpha}+2^{3\alpha/2}\pi^{\alpha}
(1+\tau_0)^{\alpha}}}:=\tau_0^{\gamma}(0<\tau_0,\gamma<1)$ then,
\begin{equation}\label{c27}
\frac{\tau}{2}\left|\delta_{kl}^{'}\right| \leq \gamma \pi.
\end{equation}
Denoting $S_{k,l}^n=\sum_{j=0}^n\mathrm e^{-it_j\delta_{kl}^{'}}(n\ge0)$, for $0<\tau<\tau_0^\gamma$, since $\frac{\sin(t)}{t}$ is bounded and decreasing for $t\in[0, \gamma \pi)$, then we have the following inequality:
\begin{equation}\label{c28}
\left|S_{k,l}^n\right| \leq \frac{1}{\left|\sin \left(\tau \delta_{kl}^{'} / 2\right)\right|} \leq \frac{C}{\tau\left|\delta_{kl}^{'}\right|}, \quad C=\frac{2 \gamma \pi}{\sin (\gamma \pi)}.
\end{equation}
Adopting summation by parts, we derive from (\ref{c25}) that
\begin{equation}\label{c29}
\sum_{j=0}^n \Lambda_{k,l}^j=r_{k,l}\left[\sum_{j=0}^{n-1} S_{k,l}^j\left(c_{k,l}^j-c_{k,l}^{j+1}\right)+S_{k,l}^n c_{k,l}^n\right],
\end{equation}
with
\begin{equation}\label{c30}
\begin{aligned}
c^j_{k,l}-c^{j+1}_{k,l}
=&\left(\widehat{\xi}_{k_1l_1}\left(t_j\right)-\widehat{\xi}_{k_1l_1}\left(t_{j+1}\right)\right)\widehat{\xi}_{k_2l_2}\left(t_j\right)\widehat{\xi}_{k_3l_3}\left(t_j\right)
\\
&+\widehat{\xi}_{k_1l_1}\left(t_{j+1}\right)\left(\widehat{\xi}_{k_2l_2}\left(t_j\right)-\widehat{\xi}_{k_2l_2}\left(t_{j+1}\right)\right)\widehat{\xi}_{l_3l_3}\left(t_j\right) \\
&+\widehat\xi_{k_1l_1}\left(t_{j+1}\right)\widehat\xi_{k_2l_2}\left(t_{j+1}\right)\left(\widehat\xi_{l_3l_3}\left(t_j\right)-\widehat\xi_{l_3l_3}\left(t_{j+1}\right)\right).
\end{aligned}
\end{equation}

According to (\ref{c261}) and (\ref{c28})-(\ref{c30}), we get
\begin{equation}\label{c31}
\begin{split}
\left|\sum_{j=0}^n \Lambda_{k,l}^j\right| \lesssim & \tau^2\left|\delta_{k,l}^{'}\right| \sum_{j=0}^{n-1}\left(\left|\widehat{\xi}_{k_1l_1}\left(t_j\right)-\widehat{\xi}_{k_1l_1}\left(t_{j+1}\right)\right|\left|\widehat{\xi}_{k_2l_2}\left(t_j\right)\right|\left| \widehat{\xi}_{k_3l_3}\left(t_j\right)\right|\right.\\
& +\left|\widehat{\xi}_{k_1l_1}\left(t_{j+1}\right)\right|\left|\widehat{\xi}_{k_2l_2}\left(t_j\right)-\widehat{\xi}_{k_2l_2}\left(t_{j+1}\right)\right|\left|\widehat{\xi}_{k_3l_3}\left(t_j\right)\right| \\
&\left.+\left|\widehat{\xi}_{k_1l_1}\left(t_{j+1}\right)\right|\left|\widehat{\xi}_{k_2l_2}\left(t_{j+1}\right)\right|\left|\widehat{\xi}_{k_3l_3}\left(t_j\right)-\widehat{\xi}_{k_3l_3}\left(t_{j+1}\right)\right|\right) \\& +\tau^2\left|\delta_{k,l}^{'}\right|\left|\widehat{\xi}_{k_1l_1}\left(t_n\right)\right|\left|\widehat{\xi}_{k_2l_2}\left(t_n\right)\right|\left|\widehat{\xi}_{k_3l_3}\left(t_n\right)\right|.
\end{split}
\end{equation}
For $k, l \in T_{N_0},\left(k_1, k_2, k_3\right) \in I_k^{N_0}, \left(l_1, l_2, l_3\right) \in I_l^{N_0}$, we have
\begin{equation}\label{c32}
\big|\delta_{k,l}^{'}\big| \lesssim \sqrt{ 1+\biggl(\biggl|\sum_{j=1}^3\mu_{k_j}\biggr|^2+\biggl|\sum_{j=1}^3\mu_{l_j}\biggr|^2\biggr)^{\alpha/2}}
+\sum_{j=1}^3\sqrt{1+\left(\lvert\mu_{k_j}\rvert^2+\rvert\mu_{l_j}\rvert^2\right)^{\alpha/2}} \lesssim \prod_{j=1}^3\sqrt{1+\left(\lvert\mu_{k_j}\rvert^2+\rvert\mu_{l_j}\rvert^2\right)^{\alpha/2}} ,
\end{equation}
thus it holds
\begin{equation}\label{c33}
\begin{aligned}
& \left\|R_1^n\right\|_{\alpha / 2}^2 \lesssim \varepsilon^4 \sum_{k,l \in T_{N_0}}\left|\sum_{\left(k_1, k_2, k_3\right) \in {I}_k^{N_0},\atop \left(l_1, l_2, l_3\right) \in {I}_l^{N_0}} \sum_{j=0}^n \Lambda_{k,l}^j\right|^2 \\
\lesssim & \varepsilon^4 \tau^4\left\{\sum_{k,l \in T_{N_0}}\left(\sum_{\left(k_1, k_2, k_3\right) \in {I}_k^{N_0},\atop \left(l_1, l_2, l_3\right) \in {I}_l^{N_0}}\left|\widehat{\xi}_{k_1l_1}\left(t_n\right)\right|\left|\widehat{\xi}_{k_2l_2}\left(t_n\right)\right|\left|\widehat{\xi}_{k_3l_3}\left(t_n\right)\right| \prod_{j=1}^3\sqrt{1+\left(\lvert\mu_{k_j}\rvert^2+\rvert\mu_{l_j}\rvert^2\right)^{\alpha/2}} \right)^2\right. \\
& +n \sum_{j=0}^{n-1} \sum_{k,l \in T_{N_0}}\left[\left(\sum_{\left(k_1, k_2, k_3\right) \in {I}_k^{N_0},\atop\left(l_1, l_2, l_3\right) \in I_l^{N_0}}\left|\widehat{\xi}_{k_1l_1}\left(t_j\right)-\widehat{\xi}_{k_1l_1}\left(t_{j+1}\right)\right|\left|\widehat{\xi}_{k_2l_2}\left(t_j\right)\right|\left|\widehat{\xi}_{k_3l_3}\left(t_j\right)\right| \prod_{j=1}^3\sqrt{1+\left(\lvert\mu_{k_j}\rvert^2+\rvert\mu_{l_j}\rvert^2\right)^{\alpha/2}} \right)^2\right. \\
& +\left(\sum_{\left(k_1, k_2, k_3\right) \in {I}_k^{N_0},\atop \left(l_1, l_2, l_3\right) \in {I}_l^{N_0}}\left|\widehat{\xi}_{k_1l_1}\left(t_{j+1}\right)\right|\left|\widehat{\xi}_{k_2l_2}\left(t_j\right)-\widehat{\xi}_{k_2l_2}\left(t_{j+1}\right)\right|\left|\widehat{\xi}_{k_3l_3}\left(t_j\right)\right| \prod_{j=1}^3\sqrt{1+\left(\lvert\mu_{k_j}\rvert^2+\rvert\mu_{l_j}\rvert^2\right)^{\alpha/2}} \right)^2 \\
& \left.\left.+\left(\sum_{\left(k_1, k_2, k_3\right) \in {I}_k^{N_0},\atop \left(l_1, l_2, l_3\right) \in {I}_l^{N_0}}\left|\widehat{\xi}_{k_1l_1}\left(t_{j+1}\right)\right|\left|\widehat{\xi}_{k_2l_2}\left(t_{j+1}\right)\right|\left|\widehat{\xi}_{k_3l_3}\left(t_j\right)-\widehat{\xi}_{k_3l_3}\left(t_{j+1}\right)\right| \prod_{j=1}^3\sqrt{1+\left(\lvert\mu_{k_j}\rvert^2+\rvert\mu_{l_j}\rvert^2\right)^{\alpha/2}} \right)^2\right]\right\} .
\end{aligned}
\end{equation}
Next we give the following auxiliary functions to help us estimate the RHS of (\ref{c33}):
\begin{align}\label{c34}
\theta(x, y)&=\sum_{k,l \in \mathbb{Z}}\sqrt{1+\left(\lvert\mu_{k_j}\rvert^2+\rvert\mu_{l_j}\rvert^2\right)^{\alpha/2}}\left|\widehat{\xi}_{kl}\left(t_n\right)\right| \mathrm{e}^{i \mu_k(x-a)+i \mu_l(y-a)}, \\
\theta_1(x, y)&=\sum_{k,l \in \mathbb{Z}} \sqrt{1+\left(\lvert\mu_{k_j}\rvert^2+\rvert\mu_{l_j}\rvert^2\right)^{\alpha/2}}\left|\widehat{\xi}_{kl}\left(t_j\right)-\widehat{\xi}_{kl}\left(t_{j+1}\right)\right| \mathrm{e}^{i \mu_k(x-a)+i \mu_l(y-a)}, \\
\theta_2(x, y)&=\sum_{k,l \in \mathbb{Z}} \sqrt{1+\left(\lvert\mu_{k_j}\rvert^2+\rvert\mu_{l_j}\rvert^2\right)^{\alpha/2}}\left|\widehat{\xi}_{kl}\left(t_j\right)\right| \mathrm{e}^{i \mu_k(x-a)+i \mu_l(y-a)},\\
\theta_3(x, y)&=\sum_{k,l \in \mathbb{Z}} \sqrt{1+\left(\lvert\mu_{k_j}\rvert^2+\rvert\mu_{l_j}\rvert^2\right)^{\alpha/2}}\left|\widehat{\xi}_{kl}\left(t_{j+1}\right)\right| \mathrm{e}^{i \mu_k(x-a)+i \mu_l(y-a)}.
\end{align}

By the assumption $(\mathrm{A})$, we have $\|\theta\|_s \lesssim\left\|\xi\left(t_n\right)\right\|_{s+\alpha / 2}(s \leq m)$. Then
\begin{equation}\label{c35}
\begin{aligned}
\left\|R_1^n\right\|_{\alpha / 2}^2 \lesssim & \varepsilon^4 \tau^4\left[\left\|\xi\left(t_n\right)\right\|_{m+\alpha / 2}^6+n \sum_{j=0}^{n-1}\left\|\xi\left(t_j\right)-\xi\left(t_{j+1}\right)\right\|_{m+\alpha / 2}^2\right. \\
& \left.\left(\left\|\xi\left(t_j\right)\right\|_{m+\alpha / 2}+\left\|\xi\left(t_{j+1}\right)\right\|_{m+\alpha / 2}\right)^4\right] \\
 \lesssim &\varepsilon^4 \tau^4+n^2 \varepsilon^4 \tau^4\left(\varepsilon^2 \tau\right)^2 \lesssim \varepsilon^4 \tau^4, \quad 0 \leq n \leq p.
\end{aligned}
\end{equation}

Similarly, we  can have the estimates for $R_q^n$ with $q=2,3,4$. Substituting the estimates for $R^n$ into (\ref{c17}), we have
\begin{equation}\label{c36}
\left\|e^{[n+1]}\right\|_{\alpha/2}\lesssim\tau_0^{m+\alpha/2}+\varepsilon^2\tau^2+\varepsilon^2\tau\sum\limits_{j=0}^n\left\|e^{[j]}\right\|_{\alpha/2},\quad0\leq n\leq p.
\end{equation}
Using the Gronwall's inequality \cite{25}, it implies
\begin{equation}\label{c37}
\left\|e^{[n+1]}\right\|_{\alpha/2}\lesssim\tau_0^{m+\alpha/2}+\varepsilon^2\tau^2,\quad0\leq n\leq p.
\end{equation}
Thus the first inequality in (\ref{c7}) holds for $n=p+1$.  Furthermore, we have
\begin{equation}\label{c38}
\left\|\varphi^{[p+1]}\right\|_{\alpha/2}\le \left\|\varphi(t_{p+1})\right\|_{\alpha/2}+\left\|e^{[p+1]}\right\|_{\alpha/2}\le K+1.
\end{equation}
which shows  the second inequality in (\ref{c7}) also holds for $n=p+1$. So far we have completed the induction process. Based on (\ref{b31}) and (\ref{b8}),  the improved error bound (\ref{c5}) is proved.
\end{proof}
\subsection{Improved uniform error bounds for the fully discretization scheme}
%In order to give the improved uniform error bounds up to the time $T_ \varepsilon$ of the the fully discrete scheme (\ref{b9}),  we first have following result for the local truncation error for the fully discrete scheme (\ref{b9}).

\begin{lem}
For $0<\varepsilon \le 1$, the local truncation error of the fully discretization scheme (\ref{b9}) for the relativistic NSFSE (\ref{b3}) can be expressed as
\begin{equation}
{\mathcal{\bar{E}}}^n:=P_NS_{\tau}(P_N\varphi(t_n))-P_N \varphi\left(t_{n+1}\right)=P_N \mathcal{H}\left(P_N \varphi\left(t_n\right)\right)+\mathcal{Y}^n, \quad 0 \leq n \leq T_\varepsilon/\tau-1,
\label{c401}
\end{equation}
%\begin{equation}
%\begin{split}
%{\mathcal{\bar{E}}}^n:&=P_N \left(e^{i\tau\langle\nabla\rangle_\alpha}P_N\varphi(t_n)+\tau e^{i\frac{\tau\langle\nabla\rangle_\alpha}{2}}F\left(e^{i\frac{\tau\langle\nabla\rangle_\alpha}{2}}P_N\varphi(t_n)\right)
%\right)-P_N \varphi\left(t_{n+1}\right)\\&=P_N \mathcal{H}\left(P_N \varphi\left(t_n\right)\right)+\mathcal{Y}^n, \quad 0 \leq n \leq \frac{T / \varepsilon^2}{\tau}-1,
%\end{split}
%\label{c401}
%\end{equation}
then we have the following estimates under the assumption (A),
\begin{equation}\label{c402}
\left\|\mathcal{H}\left(P_N \varphi\left(t_n\right)\right)\right\|_{\alpha/2} \lesssim \varepsilon^2 \tau^3, \quad\left\|\mathcal{Y}^n\right\|_{\alpha/2} \lesssim \varepsilon^4 \tau^3+\varepsilon^2 \tau h^m.
\end{equation}
\end{lem}
\begin{proof}
We can write the local truncation error as
\begin{equation}\label{c403}
\begin{split}
\mathcal{\bar{E}}^n=&P_N\left(\tau  \mathrm{e}^{i \frac{\tau\langle\nabla\rangle_{\alpha}}{2}} F\left( \mathrm{e}^{i \frac{\tau\langle\nabla\rangle_{\alpha}}{2}} P_N\varphi(t_n)\right)-\int_0^\tau \mathrm{e}^{i(\tau-s)\langle\nabla\rangle_\alpha} F\left(\varphi\left(t_n+s\right)\right)ds\right)\\
=&P_N\left(\tau  \mathrm{e}^{i \frac{\tau\langle\nabla\rangle_{\alpha}}{2}} H\left( \mathrm{e}^{i \frac{\tau\langle\nabla\rangle_{\alpha}}{2}} P_N\varphi(t_n)\right)-\int_0^\tau \mathrm{e}^{i(\tau-s)\langle\nabla\rangle_\alpha} H\left(\varphi\left(t_n+s\right)\right)ds\right.\\&\left.+\tau  \mathrm{e}^{i \frac{\tau\langle\nabla\rangle_{\alpha}}{2}}\varepsilon^4 i\langle\nabla\rangle_{\alpha}^{-1}r\left( \mathrm{e}^{i \frac{\tau\langle\nabla\rangle_{\alpha}}{2}} P_N\varphi(t_n)\right)-\int_0^\tau \mathrm{e}^{i(\tau-s)\langle\nabla\rangle_\alpha} \varepsilon^4 i
\langle\nabla\rangle_{\alpha}^{-1} r\left(\varphi\left(t_n+s\right)\right)ds\right)\\
=&P_N\left(\tau  \mathrm{e}^{i \frac{\tau\langle\nabla\rangle_{\alpha}}{2}} H\left( \mathrm{e}^{i \frac{\tau\langle\nabla\rangle_{\alpha}}{2}} P_N\varphi(t_n)\right)-\int_0^\tau \mathrm{e}^{i(\tau-s)\langle\nabla\rangle_\alpha} H\left(\varphi\left(t_n+s\right)\right)ds\right)+\mathcal{\bar{R}}^n_1\\
=&P_N\left(\tau  \mathrm{e}^{i \frac{\tau\langle\nabla\rangle_{\alpha}}{2}} H\left(e^{i \frac{\tau\langle\nabla\rangle_{\alpha}}{2}} P_N\varphi(t_n)\right)
-\int_0^\tau \mathrm{e}^{i(\tau-s)\langle\nabla\rangle_\alpha} H\left(\mathrm{e}^{is\langle\nabla\rangle_\alpha} \varphi(t_n)+\int_0^\tau \mathrm{e}^{i(s-\sigma)\langle\nabla\rangle_\alpha} F\left(\varphi\left(t_n+\sigma\right)\right)d\sigma\right)ds\right)+\mathcal{\bar{R}}^n_1\\
=&P_N\left(\tau  \mathrm{e}^{i \frac{\tau\langle\nabla\rangle_{\alpha}}{2}} H\left(e^{i \frac{\tau\langle\nabla\rangle_{\alpha}}{2}} P_N\varphi(t_n)\right)-\int_0^\tau \mathrm{e}^{i(\tau-s)\langle\nabla\rangle_\alpha} H\left( \mathrm{e}^{is\langle\nabla\rangle_{\alpha}}\varphi\left(t_n\right)\right)ds\right)+\mathcal{\bar{R}}^n_1+\mathcal{\bar{R}}^n_2\\
=&P_N\left(\tau  \mathrm{e}^{i \frac{\tau\langle\nabla\rangle_{\alpha}}{2}} H\left( \mathrm{e}^{i \frac{\tau\langle\nabla\rangle_{\alpha}}{2}} P_N\varphi(t_n)\right)-\int_0^\tau \mathrm{e}^{i(\tau-s)\langle\nabla\rangle_\alpha} H\left( \mathrm{e}^{is\langle\nabla\rangle_{\alpha}}P_N\varphi\left(t_n\right)\right)ds\right)\\
&+P_N\left(\int_0^\tau \mathrm{e}^{i(\tau-s)\langle\nabla\rangle_\alpha} H\left( \mathrm{e}^{is\langle\nabla\rangle_{\alpha}}P_N\varphi\left(t_n\right)\right)ds-\int_0^\tau \mathrm{e}^{i(\tau-s)\langle\nabla\rangle_\alpha} H\left( \mathrm{e}^{is\langle\nabla\rangle_{\alpha}}\varphi\left(t_n\right)\right)ds\right)+\mathcal{\bar{R}}^n\\
=&P_N\mathcal{H}\left(P_N\varphi\left(t_n\right)\right)+\mathcal{Y}^n.
\end{split}
\end{equation}
where $\mathcal{H}^n\left(P_N\varphi\left(t_n\right)\right)$ is defined in (\ref{c3}), $\mathcal{Y}^n$ is $$\mathcal{Y}^n=P_N\left(\int_0^\tau \mathrm{e}^{i(\tau-s)\langle\nabla\rangle_\alpha} H\left(e^{is\langle\nabla\rangle_{\alpha}}P_N\varphi\left(t_n\right)\right)ds-\int_0^\tau \mathrm{e}^{i(\tau-s)\langle\nabla\rangle_\alpha} H\left(e^{is\langle\nabla\rangle_{\alpha}}\varphi\left(t_n\right)\right)ds\right)+\mathcal{\bar{R}}^n,$$ and $\mathcal{\bar{R}}^n=\mathcal{\bar{R}}_1^n+\mathcal{R}_2^n.$

Under the assumption (A) with $m\ge 0$, we can get
\begin{equation}
\left\|\mathcal{H}\left(P_N\varphi\left(t_n\right)\right)\right\|_{\alpha/2} \lesssim \tau^3\left\|\partial_s^2\left( \mathrm{e}^{i (\tau-s)\langle\nabla\rangle_\alpha} H\left( \mathrm{e}^{i s\langle\nabla\rangle_\alpha} P_N\varphi\left(t_n\right)\right)\right)\right\|_{\alpha/2} \lesssim \varepsilon^2 \tau^3\left\|\varphi\left(t_n\right)\right\|_{\alpha/2+1},
\end{equation}
and
\begin{equation}
\begin{split}
&\left\|\int_0^\tau \mathrm{e}^{i(\tau-s)\langle\nabla\rangle_\alpha} H\left(\mathrm{e}^{is\langle\nabla\rangle_{\alpha}}P_N\varphi\left(t_n\right)\right)ds-\int_0^\tau \mathrm{e}^{i(\tau-s)\langle\nabla\rangle_\alpha} H\left(\mathrm{e}^{is\langle\nabla\rangle_{\alpha}}\varphi\left(t_n\right)\right)ds\right\|_{\alpha/2}
\\ \lesssim & \varepsilon^2 \tau \left\| P_N\varphi\left(t_n\right)-\varphi\left(t_n\right)\right\|_{\alpha/2} \\ \lesssim &\varepsilon^2 \tau h^m \left\|\varphi\left(t_n\right)\right\|_{\alpha/2+m},
\end{split}
\end{equation}
which means $\left\|\mathcal{Y}^n\right\|_{\alpha/2} \lesssim \varepsilon^4 \tau^3+\varepsilon^2 \tau h^m$.  The proof of the error bounds (\ref{c402}) is  completed.
\end{proof}

Next the improved uniform error bounds up to the time at $T_ \varepsilon$ of the fully discretization scheme (\ref{b9})-(\ref{b10}) are given as follows.
\begin{thrm}
For $0<\varepsilon \le 1$, let $h_0>0$, $0<\tau_0<1$ be small enough and independent of $\varepsilon$. When $0<h \leq h_0$ and $0<\tau< \gamma \frac{\pi(b-a)^{\alpha / 2} \tau_0^{\alpha / 2}}{2 \sqrt{ \tau_0^\alpha(b-a)^\alpha+2^{3\alpha/2} \pi^\alpha\left(1+\tau_0\right)^\alpha}}$, where $\gamma \in(0,1)$ is a fixed constant, we have the following improved uniform error bound under the assumption (A),
\begin{equation}\label{c39}
\left\|u\left(\cdot, t_n\right)-I_N u^n\right\|_{\alpha / 2}+\left\|\partial_t u\left(\cdot, t_n\right)-I_N v^n\right\| \lesssim h^m+\varepsilon^2 \tau^2+\tau_0^{m+\alpha / 2}, \quad 0 \leq n \leq T_{\varepsilon} / \tau .
\end{equation}
Especially, if the exact solution is sufficiently smooth, then the improved uniform error bound for sufficiently small $\tau$ is
\begin{equation}\label{c40}
\left\|u\left(\cdot, t_n\right)-I_N u^n\right\|_{\alpha / 2}+\left\|\partial_t u\left(\cdot, t_n\right)-I_N v^n\right\| \lesssim h^m+\varepsilon^2 \tau^2, \quad 0 \leq n \leq T_{\varepsilon} / \tau.
\end{equation}
\label{t2}
\end{thrm}

\begin{proof}
Considering that $\varphi\left(t_n\right)-I_N \varphi^n=\varphi\left(t_n\right)-P_N \varphi\left(t_n\right)+P_N \varphi\left(t_n\right)-I_N \varphi^n$, we derive
\begin{equation}\label{c41}
\left\|\varphi\left( t_n\right)-I_N \varphi^n\right\|_{\alpha / 2} \lesssim \left\|P_N \varphi(t_n)-I_N \varphi^n\right\|_{\alpha / 2}+h^m.
\end{equation}
Define the error function $e^n:=e^n(x, y) \in Y_N$ as
\begin{equation}\label{c42}
e^n:=I_N \varphi^n-P_N  \varphi(t_n), \quad 0 \leq n \leq T_{\varepsilon} / \tau,
\end{equation}
we have
\begin{equation}\label{c43}
\begin{split}
e^{n+1}=&I_N \varphi^{n+1}-P_NS_{\tau}(P_N\varphi(t_n))+\mathcal{\bar{{E}}}^n\\
=&\mathrm{e}^{i \tau\langle\nabla\rangle_{\alpha}} e^n+Z^n+{\mathcal{\bar{E}}}^n,
\end{split}
\end{equation}
where
\begin{equation}\label{c44}
Z^n=\tau \mathrm{e}^{i \frac{\tau\langle\nabla\rangle_{\alpha}}{2}}\left(I_N F\left(\mathrm{e}^{i \frac{\tau\langle\nabla\rangle_{\alpha}}{2}}I_N \varphi^n\right)-P_N F\left(\mathrm{e}^{i \frac{\tau\langle\nabla\rangle_{\alpha}}{2}}P_N \varphi\left(t_n\right)\right)\right).
\end{equation}
Then we get
\begin{equation}\label{c441}
  e^{n+1}=\mathrm{e}^{i (n+1)\tau\langle\nabla\rangle_{\alpha}} e^0+\sum_{j=0}^{n}\mathrm{e}^{i(n-j) \tau\langle\nabla\rangle_{\alpha}}\left( Z^j+{\mathcal{\bar{E}}}^j\right).
\end{equation}

Similar to the proof for the semi-discretization, we also adopt the mathematical induction to prove that there exist $h_e>0$, $\tau_e>0$ such that for $0<h<h_e$ and $0<\tau<\tau_e$, we have the following estimates
\begin{equation}\label{c45}
\left\|e^n\right\|_{\alpha/2}  \lesssim h^m+\varepsilon^2 \tau^2+\tau_0^{m+1},\quad \left\|I_N \varphi^n\right\|_{\alpha/2} \leq K+1, \quad 0 \leq n \leq {T_ \varepsilon/ \tau}.
\end{equation}
%where $M>0$ is a constant depending on $T$.
For the case  $n=0$, (\ref{c45}) is ture for sufficiently small $0<h<h_1$ with $h_1>0$ based on the standard Fourier interpolation result, i.e., $\left\|e^0\right\|_{\alpha/2} \lesssim h^m$ and $\left\|I_N \varphi^0\right\|_{\alpha/2} \leq $ $K+1$. Next we suppose (\ref{c45}) is ture for $0<n \leq p \leq {T_\varepsilon}/{\tau}$, and we  prove it is ture for the case $n=p+1$.

By the definition of $Z^n$, we have

\begin{equation}\label{c46}
\begin{split}
Z^n=&\tau \mathrm{e}^{i \tau\langle\nabla\rangle_{\alpha}}\left(I_N F\left(\mathrm{e}^{i \frac{\tau\langle\nabla\rangle_{\alpha}}{2}}I_N \varphi^n\right)-I_N F\left(\mathrm{e}^{i \frac{\tau\langle\nabla\rangle_{\alpha}}{2}}P_N \varphi(t_n)\right)\right.
\\&\left.+I_N F\left(\mathrm{e}^{i \frac{\tau\langle\nabla\rangle_{\alpha}}{2}}P_N \varphi(t_n)\right)-P_N F\left(\mathrm{e}^{i \frac{\tau\langle\nabla\rangle_{\alpha}}{2}}P_N \varphi\left(t_n\right)\right)\right),
\end{split}
\end{equation}
then  we have for $0<h<h_2$ and $0<\tau<\tau_2$,
\begin{equation}\label{c47}
\left\|Z^n\right\|_{\alpha/2} \lesssim \varepsilon^2 \tau\left(h^m+\left\|e^n\right\|_{\alpha/2}\right), \quad 0 \leq n \leq p.
\end{equation}
According to the estimates (\ref{c441}) and (\ref{c47}), it holds
\begin{equation}\label{c48}
\left\|e^{n+1}\right\|_{\alpha/2} \lesssim h^m+\varepsilon^2 \tau^2+\varepsilon^2 \tau \sum_{j=0}^n\left\|e^j\right\|_{\alpha/2}+\left\|\sum_{j=0}^n \mathrm{e}^{i(n-j) \tau\langle\nabla\rangle} P_N\mathcal{H}\left( P_N\varphi\left(t_j\right)\right)\right\|_{\alpha/2}, \quad 0 \leq n \leq p.
\end{equation}
Similar to the proof in Theorem \ref{thrm1} and based on the assumption (A), we replace  $P_N$ by $P_{N_0}$ in (\ref{c47}) and obtain
\begin{equation}\label{c49}
\left\|e^{n+1}\right\|_{\alpha/2} \lesssim h^m+\tau_0^{m+1}+\varepsilon^2 \tau^2+\varepsilon^2 \tau \sum_{j=0}^n\left\|e^j\right\|_{\alpha/2}+\left\|R^n\right\|_{\alpha/2},
\end{equation}
where $R^n$ is defined in (\ref{c18}). Due to the estimate (\ref{c35}) for ${R}^n$, we have
\begin{equation}\label{c50}
\left\|e^{n+1}\right\|_{\alpha/2} \lesssim h^m+\tau_0^{m+1}+\varepsilon^2 \tau^2+\varepsilon^2 \tau \sum_{j=0}^n\left\|e^j\right\|_{\alpha/2}, \quad 0 \leq n \leq p.
\end{equation}
Then the following estimate can be obtained by the Gronwall inequality:
\begin{equation}\label{c501}
\left\|e^{n+1}\right\|_{\alpha/2} \lesssim h^m+\tau_0^{m+1}+\varepsilon^2 \tau^2, \quad 0 \leq n \leq p,
\end{equation}
which implies the first inequality in (\ref{c45}) is ture for $n=p+1$. Further, we have
\begin{equation}\label{c502}
\left\|I_N\varphi^{p+1}\right\|_{\alpha/2} \leq\left\|\varphi\left(t_{p+1}\right)\right\|_{\alpha/2}+\left\|e^{p+1}\right\|_{\alpha/2} \leq K+1,
\end{equation}
thus the second inequality in (\ref{c45}) is ture for $n=p+1$. The proof of (\ref{c45}) is finished and the improved uniform error bound (\ref{c39}) can be obtained by (\ref{b31}) as well as (\ref{b10}).
\end{proof}
\begin{remark}
In\cite{48, 49}, the improved uniform error bounds for the fully discretization are mainly based on the error bounds of the semi-discretization and the bound of $\lVert\varphi^{[n]}\rVert_{H^{m+\alpha/2}}$ is necessary to obtain the space convergence order $h^m$. However, the bound of $\lVert\varphi^{[n]}\rVert_{H^{m+\alpha/2}}$ is not available for the nonlinear equation. Here, we prove the error bounds for the fully discretization directly by the mathematical induction, in which the bound of $\lVert\varphi^{[n]}\rVert_{H^{m+\alpha/2}}$ is not needed.
\end{remark}
\begin{remark}
%For $u, v \in H_{\text {per }}^{s+r}(\Omega), s, r \geq 0$ , we have $$ \left((-\Delta)^{s+r} u, v\right)=\left((-\Delta)^s u,(-\Delta)^r v\right).$$
The  NSFSGE (\ref{b1}) conserves the energy as
\begin{equation}\label{c51}
\begin{aligned}
E(t) & :=\int_{\Omega}\left[\left|\partial_t u(x, y, t)\right|^2+\left|(-\Delta)^{{\alpha}/{4}} u(x, y, t)\right|^2+\frac{2}{\varepsilon^2}(1-\cos (\varepsilon u(x, y, t))\right] d {\Omega} \\
& \equiv \int_{\Omega}\left[\left|u_1(x, y)\right|^2+\left|(-\Delta)^{{\alpha}/{4}} u_0(x, y)\right|^2+\frac{2}{\varepsilon^2}\left(1-\cos \left(\varepsilon u_0(x, y)\right)\right)\right] d{\Omega} \\
& =E(0), \quad t \geq 0.
\end{aligned}
\end{equation}
Similar to the energy conservation of the classical nonlinear sine-Grodon equations, the above equation can be proved with the aid of the relation \cite{26}
 that $ \left((-\Delta)^{s+r} u, v\right)=\left((-\Delta)^s u,(-\Delta)^r v\right)$ for $u, v \in H_{\text {per }}^{s+r}(\Omega), s, r \geq 0.$

We intorduce the  discrete energy at $t=t_n$ with the space mesh size $h$ as

\begin{equation}\label{c52}
E_h^n=h^2 \sum_{p=0}^{N-1} \sum_{q=0}^{N-1}\left[\left|v_{pq}^n\right|^2+\left|\left((-\Delta)^{\alpha / 4} \varphi\right)_{pq}^n\right|^2+\frac{2}{\varepsilon^2}\left(1-\cos \left(\varepsilon \varphi_{pq}^n\right)\right)\right],
\end{equation}
then the following estimate of the discrete energy can be obtained:
\begin{equation}\label{c53}
\left|E_h^n-E_h^0\right| \lesssim h^m+\varepsilon^2 \tau^2+\tau_0^{m+\alpha / 2}, \quad 0 \leq n \leq T_{\varepsilon} / \tau.
\end{equation}
If the exact solution is sufficiently smooth, the estimate of the discrete energy for sufficiently small $\tau$ is
\begin{equation}\label{c54}
\left|E_h^n-E_h^0\right| \lesssim h^m+\varepsilon^2 \tau^2, \quad 0 \leq n \leq T_{\varepsilon} / \tau.
\end{equation}
\end{remark}

\section{Extensions}
\label{sec:4}
In this section, we extend the TSFP method and improved  uniform  error bounds to the complex NSFSGE and the oscillatory complex NSFSGE. %which propagates waves with wavelength at $O\left(\varepsilon^{2 p}\right)$ in time.

\subsection{The complex NSFSGE}
Consider the following complex NSFSGE:
\begin{equation}
\left\{\begin{array}{l}
\partial_{t t} u(x,y, t)+(-\Delta)^{{\alpha}/{2}} u(x,y, t)+\frac{1}{\varepsilon} \sin (\varepsilon u(x,y, t))=0, \quad (x, y) \in \Omega, \quad t>0, \\
%u(a,t)=u(b,t), \partial_{x}u(a,t)=\partial_{x}u(b,t),\\
u(x, y, 0)=u_0(x, y), \quad \partial_t u(x, y, 0)=u_1(x, y), \quad (x,y) \in \Omega,
\end{array}\right.
\label{d1}
\end{equation}
with periodic boundary equations. $u:=u(x, y, t)$ is a complex-valued function, $u_0(x,y)$ and $u_1(x,y)$ are two known complex valued functions independent of $\varepsilon$. Introducing $v(x, y, t)=\partial_t u(x, y, t)$ and
\begin{equation}\label{d2}
\varphi_{ \pm}(x, y, t)=u(x, y, t) \mp i\langle\nabla\rangle_\alpha^{-1} v(x, y, t), \quad (x, y) \in \Omega, \quad t \geq 0,
\end{equation}
 then Eq. (\ref{d1}) can be changed into the coupled relativistic NSFSEs as follows:

\begin{equation}\label{d3}
\left\{\begin{array}{l}
i \partial_t \varphi_{ \pm}(x,y,t) \pm\langle\nabla\rangle_\alpha \varphi_{ \pm} (x,y,t)\pm\langle\nabla\rangle_\alpha^{-1} f\left(\frac{\varphi_{+}+\varphi_{-}}{2}\right)=0, \\
\varphi_{ \pm}(x,y,0)=u_0(x,y) \mp i\langle\nabla\rangle_\alpha^{-1} u_1(x,y).
\end{array}\right.
\end{equation}
%where $f(\varphi)=\frac{1}{\varepsilon} \sin (\varepsilon \varphi)-\varphi$.
If the exact solution $u:=u(x, y, t)$ of the  complex NSFSGE (\ref{d1}) up to the time $T_{\varepsilon}=T / \varepsilon^{2}$ exits and  satisfies the assumption (A),  we give the following improved uniform error bounds.
\begin{thrm}
For $0<\varepsilon \leq 1$, let $h_0>0$, $0<\tau_0<1$  be small enough and independent of $\varepsilon$. When $0<h \leq h_0$ and $0<\tau< \gamma \tau_0$, where $\gamma >0 $ is a fixed constant,  we have the following improved uniform error estimate under the assumption (A),
\begin{equation}\label{d4}
\left\|u\left(\cdot, t_n\right)-I_N u^n\right\|_{\alpha / 2}+\left\|\partial_t u\left(\cdot, t_n\right)-I_N v^n\right\| \lesssim h^m+\varepsilon^2 \tau^2+\tau_0^{m+\alpha / 2}, \quad 0 \leq n \leq T_{\varepsilon} / \tau .
\end{equation}
Especially, if the exact solution is sufficiently smooth, then the improved uniform error bound for sufficiently small $\tau$ is
\begin{equation}\label{d5}
\left\|u\left(\cdot, t_n\right)-I_N u^n\right\|_{\alpha / 2}+\left\|\partial_t u\left(\cdot, t_n\right)-I_N v^n\right\| \lesssim h^m+\varepsilon^2 \tau^2, \quad 0 \leq n \leq T_{\varepsilon} / \tau.
\end{equation}
\end{thrm}
\subsection{The oscillatory complex NSFSGE}
Re-scale in time
$$
t=\frac{s}{\varepsilon^{2 p}} \Longleftrightarrow s=\varepsilon^{2 p} t, \quad \omega(x, y, s)=u(x, y, t),
$$
then  Eq. (\ref{d1}) can be written as the following oscillatory complex NSFSGE:
\begin{equation}\label{d6}
\left\{\begin{array}{l}
\varepsilon^{4 p} \partial_{ss} \omega(x, y, s)+(-\Delta)^{{\alpha}/{2}} \omega(x, y, s)+\frac{1}{\varepsilon} \sin (\varepsilon \omega(x,y, s))=0, \quad (x, y) \in \Omega, \quad s>0,\\
\omega(x, y, 0)=u_0(x,y), \quad \partial_s \omega(x,y, 0)=\frac{1}{\varepsilon^{2 p}} u_1(x,y), \quad (x,y) \in \Omega,
\end{array}\right.
\end{equation}
%The solution of (\ref{d6}) propagates waves with amplitude at $O(1)$, wavelength at $O(1)$ and $O(\varepsilon^{2p})$ in space and time, as well as wave velocity at $O(\varepsilon^{-2p})$. Nonlinear equation
Taking the time step $\lambda=\varepsilon^{2 p} \tau$, if the exact solution $\omega:=\omega(x, y, s)$ of Eq. (\ref{d6}) exits and satisfies the following assumption
\begin{equation}\label{d601}
\text{(B)} \quad \|\omega\|_{L^{\infty}\left([0, T] ; H^{m+\alpha / 2}\right)} \lesssim 1, \quad\left\|\partial_s \omega\right\|_{L^{\infty}\left([0, T] ; H^m\right)} \lesssim \frac{1}{\varepsilon^{2 p}}, \quad m \geq 0,
\end{equation}
the improved error bounds for the oscillatory  complex NSFSGE (\ref{d6}) up to the fixed time $T$ can be obtained.
 \begin{thrm}
For $0<\varepsilon \leq 1$, let $h_0>0$, $0<\lambda_0<1$  be small enough and independent of $\varepsilon$. When $0<h \leq h_0$ and $0<\lambda< \gamma \lambda_0\varepsilon^{2p}$, where $\gamma >0$ is a fixed constant, we have the following improved uniform error estimate Under the assumption (B),
\begin{equation}\label{d7}
\left\|\omega\left(\cdot, s_n\right)-I_N \omega^n\right\|_{\alpha / 2}+\varepsilon^{2p}\left\|\partial_s \omega\left(\cdot, t_n\right)-I_N v^n\right\| \lesssim h^m+\frac{\lambda^2}{\varepsilon^{2p} }+\tau_0^{m+\alpha / 2}, \quad 0 \leq n \leq T / \lambda .
\end{equation}
Especially, if the exact solution is sufficiently smooth, then the improved uniform error bound for sufficiently small $\tau$ is
\begin{equation}\label{d8}
\left\|\omega\left(\cdot, t_n\right)-I_N \omega^n\right\|_{\alpha / 2}+\varepsilon^{2p}\left\|\partial_s \omega\left(\cdot, t_n\right)-I_N v^n\right\| \lesssim h^m+\frac{\lambda^2}{\varepsilon^{2p} }, \quad 0 \leq n \leq T / \lambda.
\end{equation}
\end{thrm}
%\begin{rmrk}
%Although our numerical scheme and theoretical analysis are for the two-dimensional NSFSGE, the TSFP method and the error estimation in our paper can be fully generalized to long-time dynamics of the three-dimensional models and even higher dimensional equations. For the three-dimensional case, we will illustrate by numerical examples. %For simplicity, we don't go into detail for the higher dimensional case.
%\end{rmrk}
\section{Numerical results}
\label{sec:5}
In this section, we provide some numerical examples to verify our uniform error bounds on the TSFP method for the long-time dynamics of the NSFSGE  in 2D and 3D. Moreover, we also give some  applications  to demonstrate  the difference of the dynamic behavior between the fractional sine-Gordon equation and the classical  sine-Gordon euqation.
%\subsection{Long-time dynamics }
\subsection{ The long-time dynamics in 2D }
We consider the NSFSGE $(\ref{a1})$ in 2D with the domain $(x, y) \in \Omega=(0,1) \times(0,2 \pi)$ and the initial conditions are
\begin{equation}\label{f1}
u_0(x, y)=\frac{2}{2+\cos ^2(2 \pi x+y)}, \quad u_1(x, y)=\frac{2}{2+2 \cos ^2(2 \pi x+y)},
\end{equation}
As no exact solutions can be determined for Eq. $(\ref{a1})$, we use the numerical solutions with spatial nodes $N=128$ and time step $\tau=10^{-3}$   as the `exact' solutions for comparison. The error in $H^{\alpha/2}-$norm is defined as
\begin{equation}\label{f1}
e\left(t_n\right)=\left\|u\left(x,y, t_n\right)-I_N u^n\right\|_{\alpha / 2}.
\end{equation}

Figure \ref{fig1} gives the long-time spatial errors at $t=1/\varepsilon^2$ for different $N$ and $\varepsilon$ when $\alpha$ is taken as 2, 1.5 or 1.2. It shows  the  spatial errors decay exponentially with $N$  for different fractional order $\alpha$ in Figure \ref{fig1} $(a_1)-(a_3)$, so the  spectral accuracy can be obtained in the space. From Figure \ref{fig1} $(b_1)-(b_3)$, we observe the spatial errors change very little as the  parameter $\varepsilon$ changes, which indicates  $\varepsilon$ have no influence on the spatial errors. The long-time temporal errors  for different $\tau$ and $\varepsilon$ at $t=1/\varepsilon^2$ with $\alpha$ taken as 2, 1.5 or 1.2 are present in Figure \ref{fig2}. From Figure  \ref{fig2}, we can see that temporal errors change like $O(\varepsilon^2\tau^2)$ up to the time at $O(1/\varepsilon^2)$ and are independent of the fractional order $\alpha$, which confirms the improved uniform error bound (\ref{c39}) are sharp. In addition, Figure \ref{fig3} displays the long-time temporal errors of the discrete energy for different $\tau$ and $\varepsilon$ at $t=1/\varepsilon^2$ and it indicates that the uniform error bound $O(\varepsilon^2\tau^2)$ for the discrete energy are obtained.
\begin{figure*}[htbp]
\centering
\subfigure{
\begin{minipage}[t]{0.3\textwidth}
\centering
\includegraphics[width=5cm]{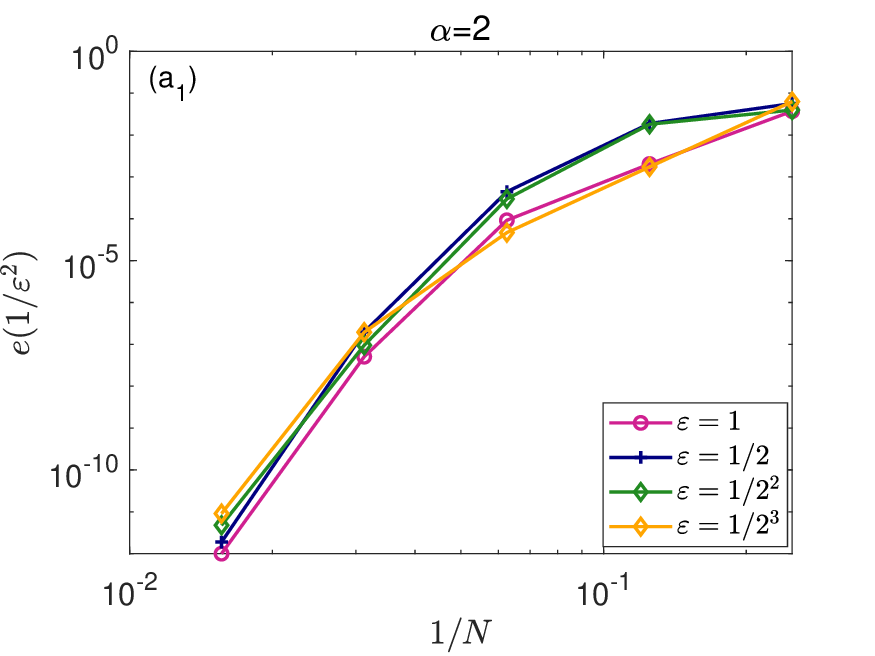}
\end{minipage}
}
\subfigure{
\begin{minipage}[t]{0.3\textwidth}
\centering
\includegraphics[width=5cm]{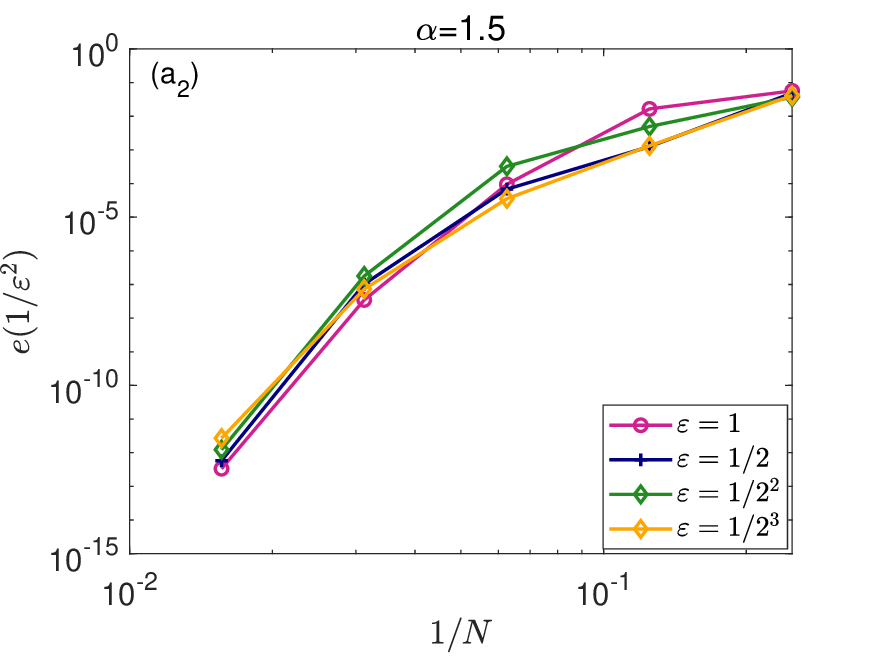}
\end{minipage}
}
\subfigure{
\begin{minipage}[t]{0.3\textwidth}
\centering
\includegraphics[width=5cm]{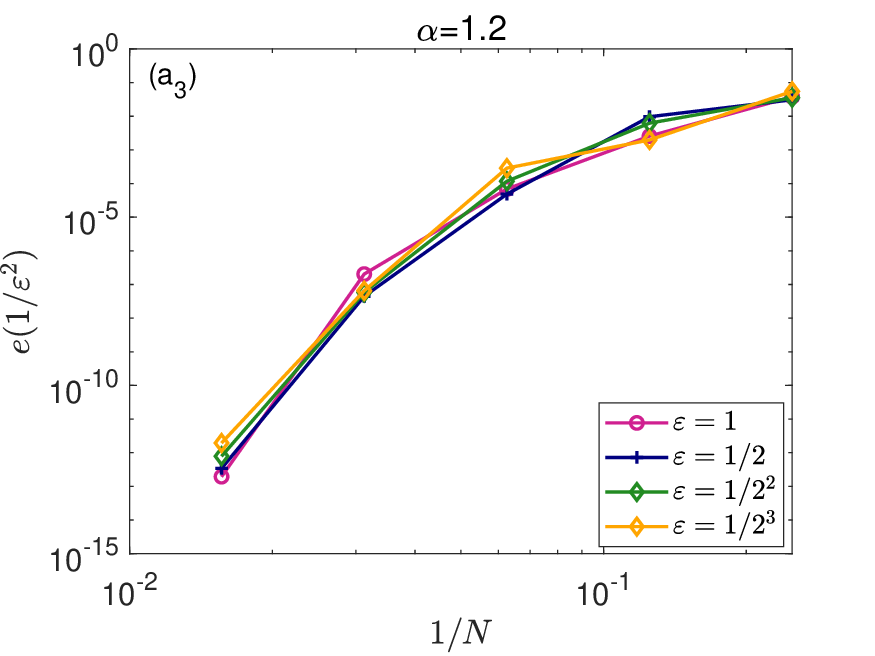}
\end{minipage}
}
\\
\subfigure{
\begin{minipage}[t]{0.3\textwidth}
\centering
\includegraphics[width=5cm]{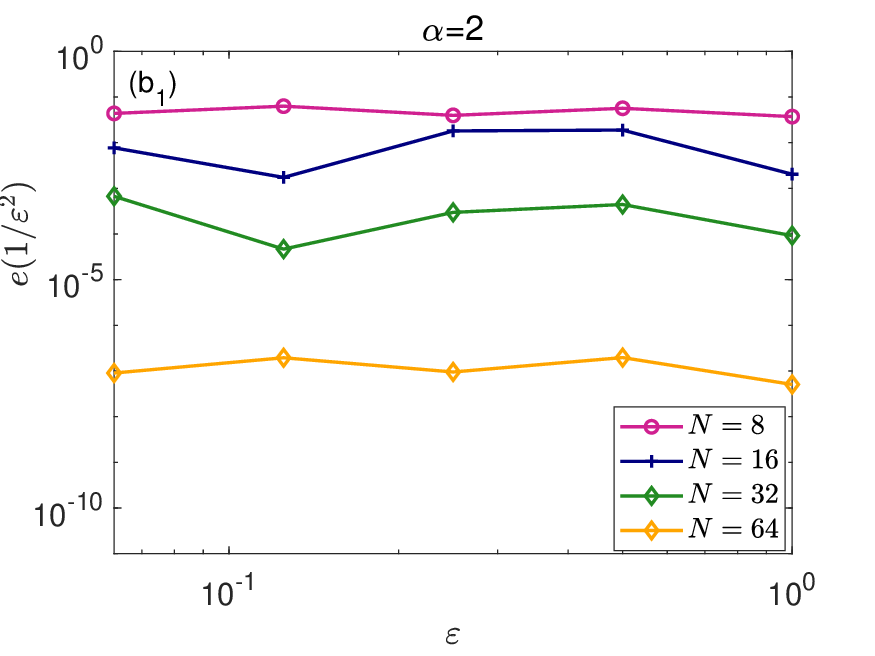}
\end{minipage}
}
\subfigure{
\begin{minipage}[t]{0.3\textwidth}
\centering
\includegraphics[width=5cm]{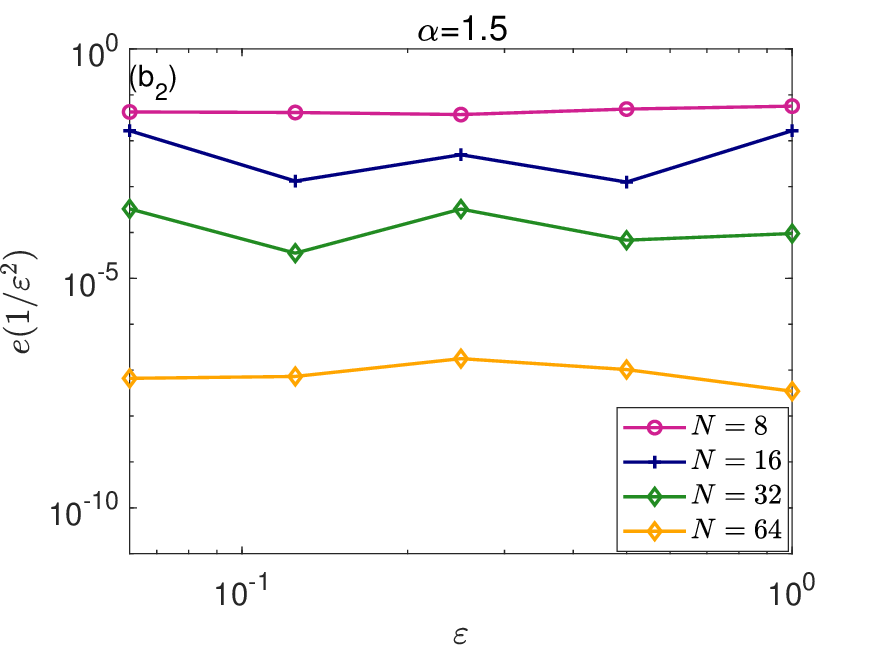}
\end{minipage}
}
\subfigure{
\begin{minipage}[t]{0.3\textwidth}
\centering
\includegraphics[width=5cm]{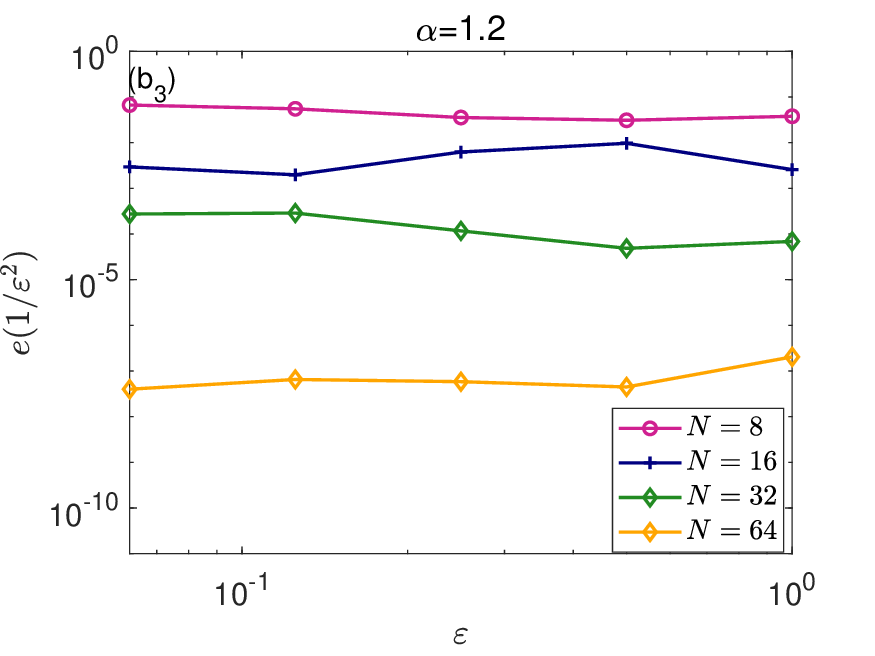}
\end{minipage}
}
\caption{ Long-time spatial errors for the NSFSGE in 2D  at $t=1/\varepsilon^2$ with different $\alpha$. }
\label{fig1}
\end{figure*}

\begin{figure*}[htbp]
\centering
\subfigure{
\begin{minipage}[t]{0.3\textwidth}
\centering
\includegraphics[width=5cm]{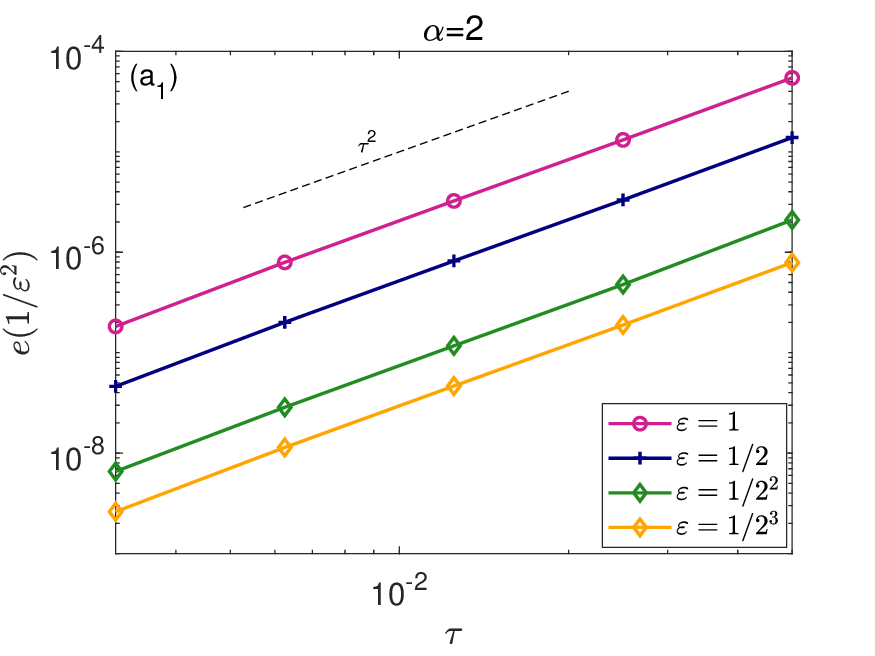}
\end{minipage}
}
\subfigure{
\begin{minipage}[t]{0.3\textwidth}
\centering
\includegraphics[width=5cm]{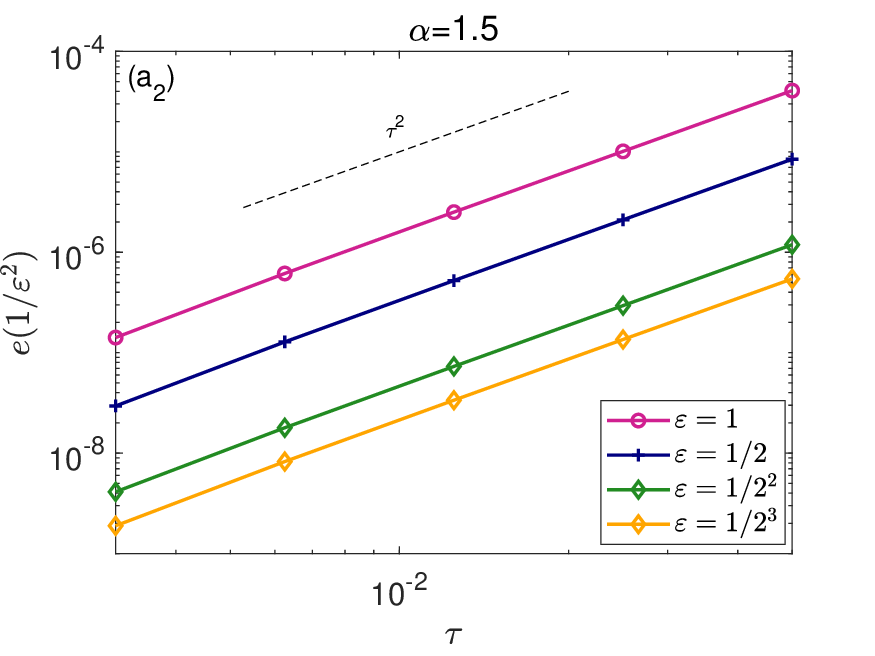}
\end{minipage}
}
\subfigure{
\begin{minipage}[t]{0.3\textwidth}
\centering
\includegraphics[width=5cm]{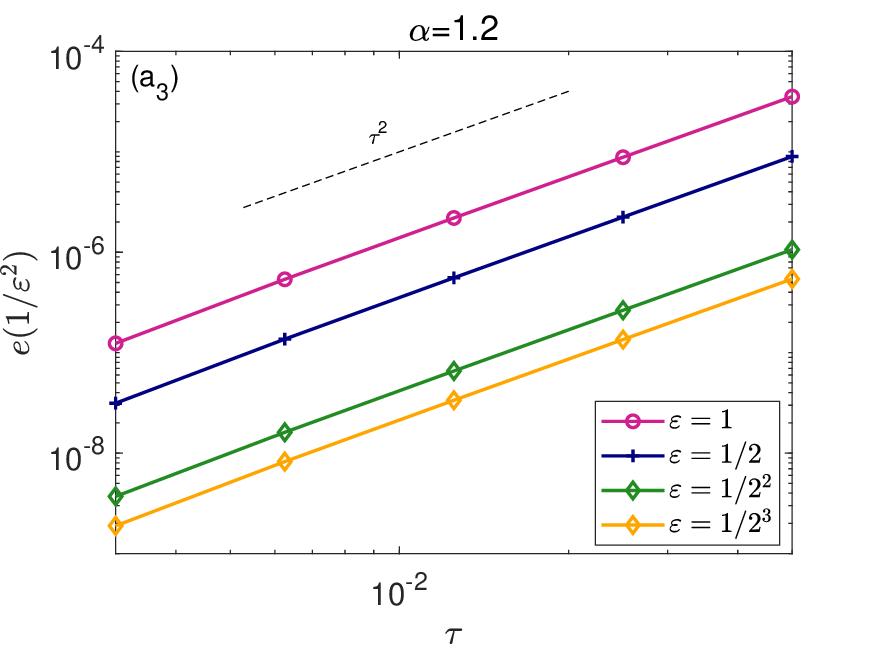}
\end{minipage}
}
\\
\subfigure{
\begin{minipage}[t]{0.3\textwidth}
\centering
\includegraphics[width=5cm]{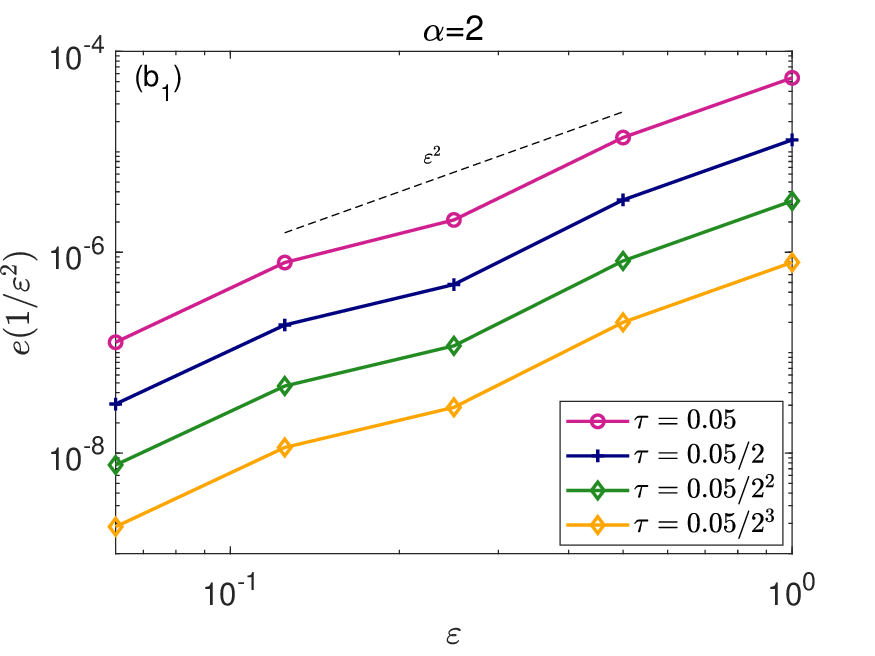}
\end{minipage}
}
\subfigure{
\begin{minipage}[t]{0.3\textwidth}
\centering
\includegraphics[width=5cm]{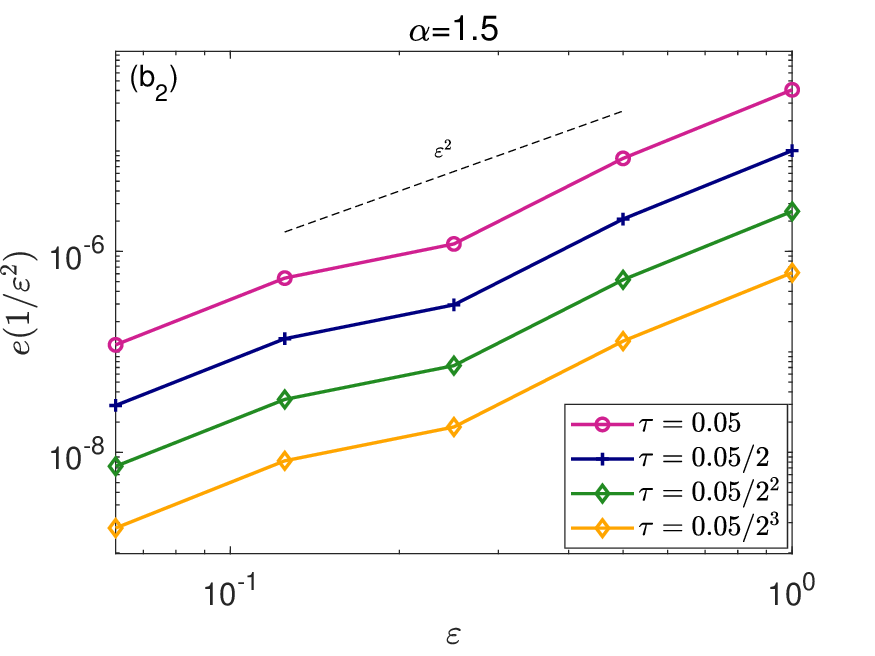}
\end{minipage}
}
\subfigure{
\begin{minipage}[t]{0.3\textwidth}
\centering
\includegraphics[width=5cm]{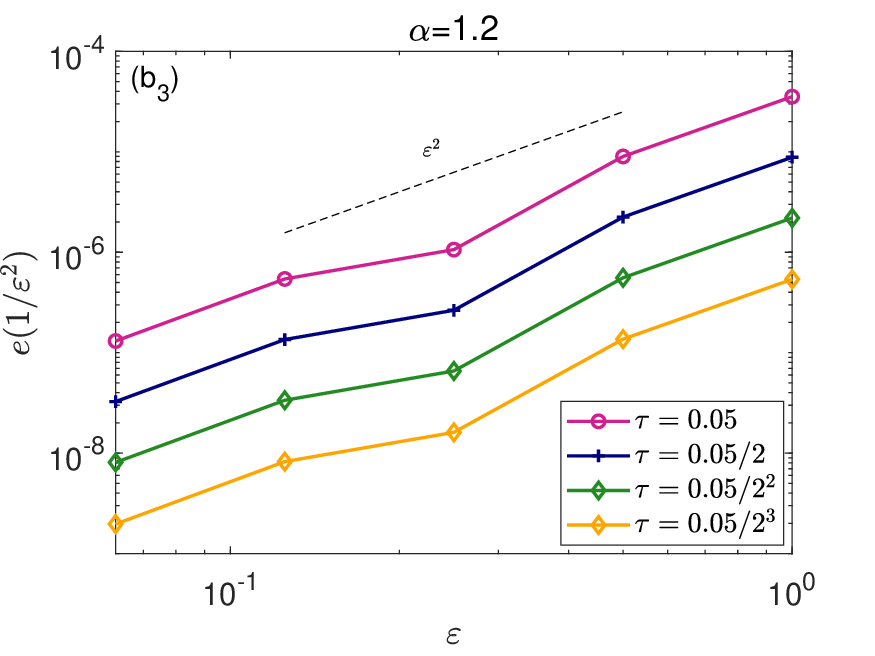}
\end{minipage}
}
\caption{ Long-time temporal errors for the NSFSGE in 2D at $t=1/\varepsilon^2$ with different $\alpha$. }
\label{fig2}
\end{figure*}

\begin{figure*}[htbp]
\centering
\subfigure{
\begin{minipage}[t]{0.3\textwidth}
\centering
\includegraphics[width=5cm]{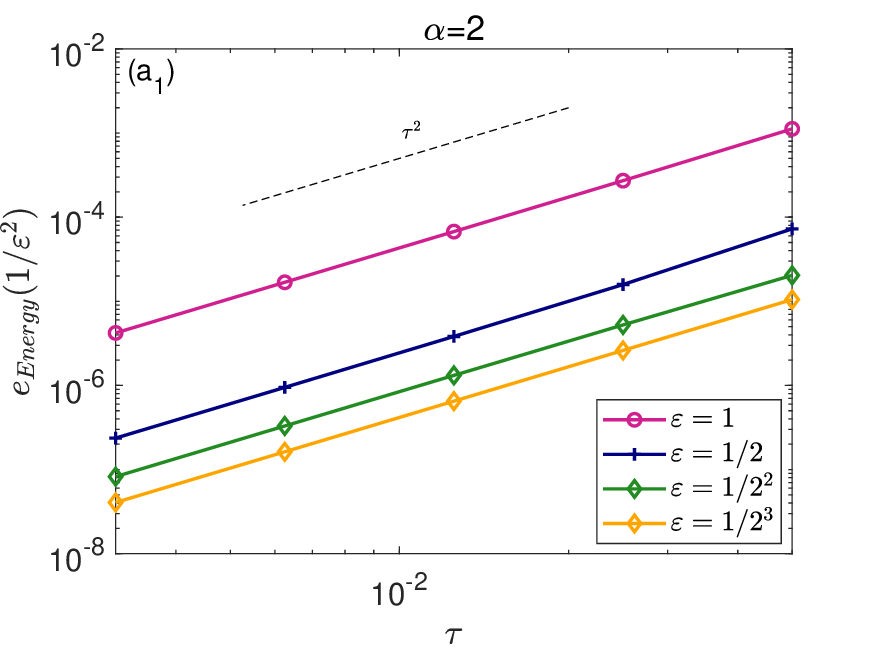}
\end{minipage}
}
\subfigure{
\begin{minipage}[t]{0.3\textwidth}
\centering
\includegraphics[width=5cm]{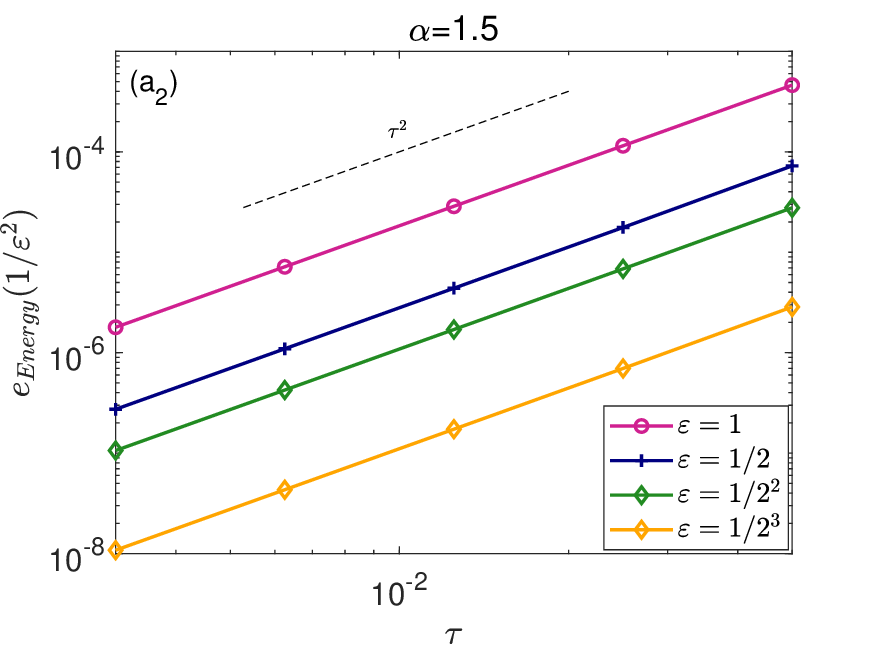}
\end{minipage}
}
\subfigure{
\begin{minipage}[t]{0.3\textwidth}
\centering
\includegraphics[width=5cm]{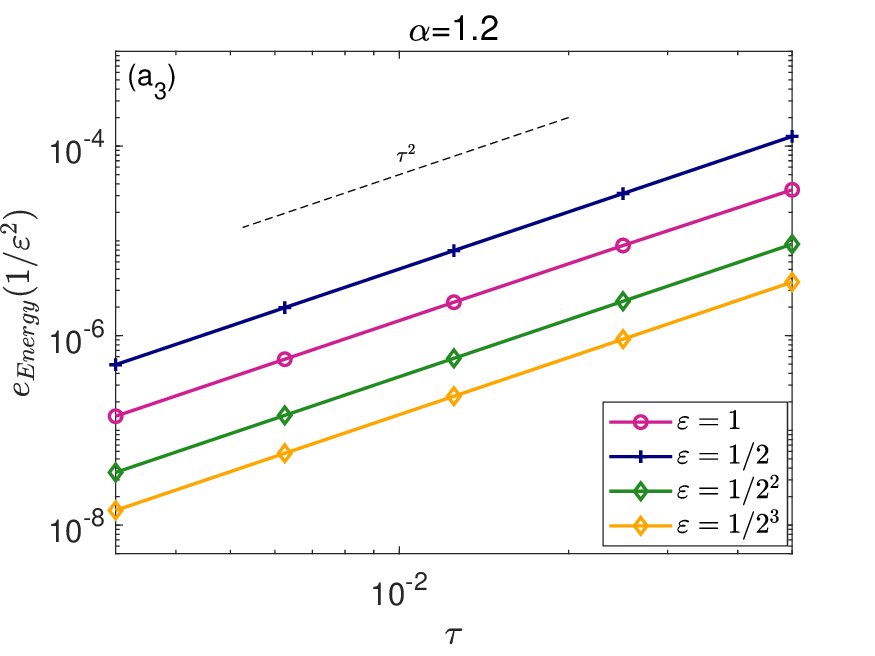}
\end{minipage}
}
\\
\subfigure{
\begin{minipage}[t]{0.3\textwidth}
\centering
\includegraphics[width=5cm]{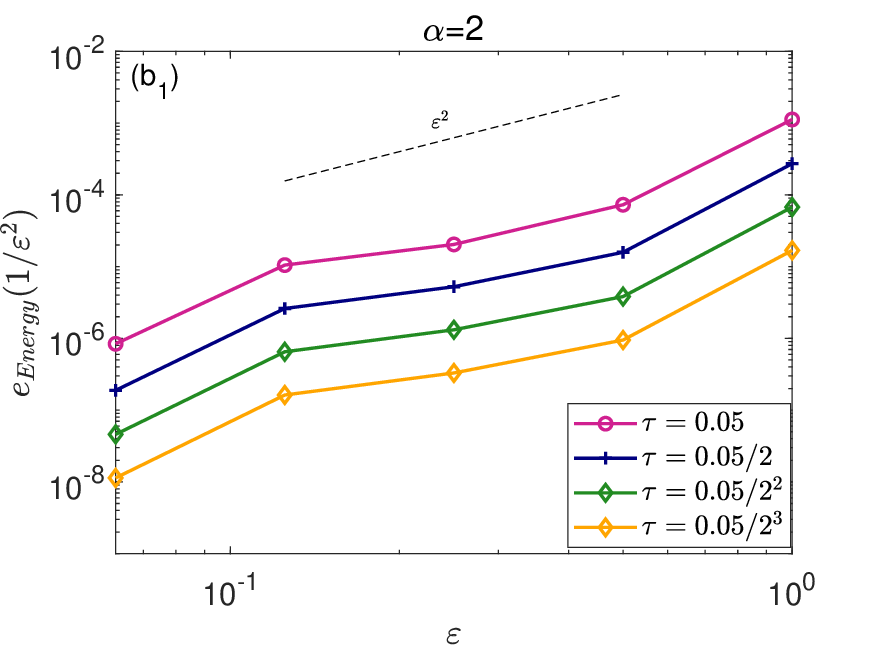}
\end{minipage}
}
\subfigure{
\begin{minipage}[t]{0.3\textwidth}
\centering
\includegraphics[width=5cm]{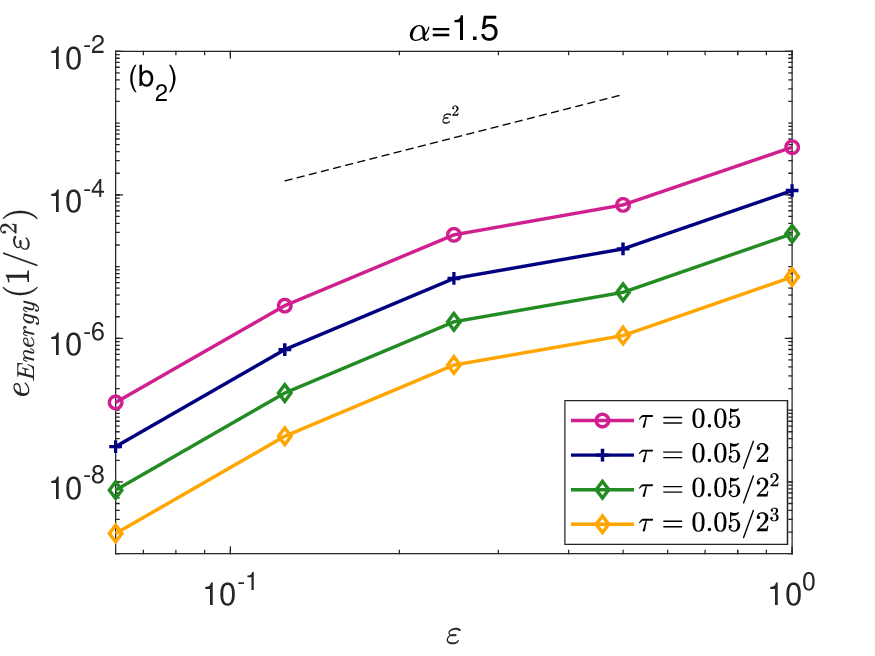}
\end{minipage}
}
\subfigure{
\begin{minipage}[t]{0.3\textwidth}
\centering
\includegraphics[width=5cm]{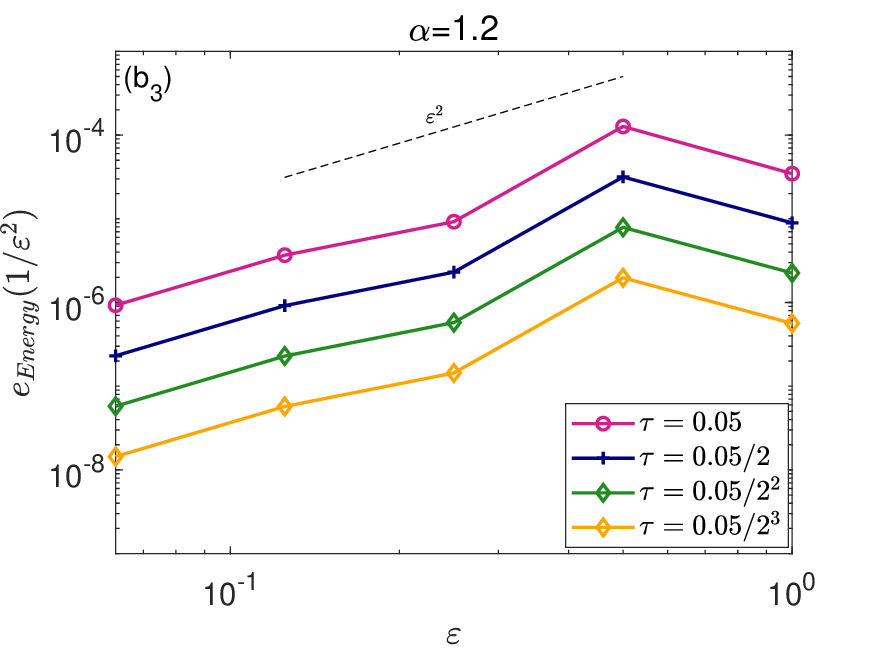}
\end{minipage}
}
\caption{ Long-time temporal errors for the discrete energy of the NSFSGE  in 2D at $t=1/\varepsilon^2$ with different $\alpha$. }
\label{fig3}
\end{figure*}

\subsection{The long-time dynamics in 3D}
We consider the NSFSGE (\ref{a1}) in 3D  with the domain $(x, y, z) \in \Omega=(0,2) \times(0,2 \pi)\times(0,2 \pi)$ and the initial conditions are taken as
\begin{equation}\label{3f1}
u_0(x, y, z)=\frac{1}{1+\sin ^2(2 \pi x+y+z)}, \quad u_1(x, y, z)=\frac{2}{1+ \sin ^2(2 \pi x+y+z)}.
\end{equation}

\begin{figure*}[htbp]
\centering
\subfigure{
\begin{minipage}[t]{0.3\textwidth}
\centering
\includegraphics[width=5cm]{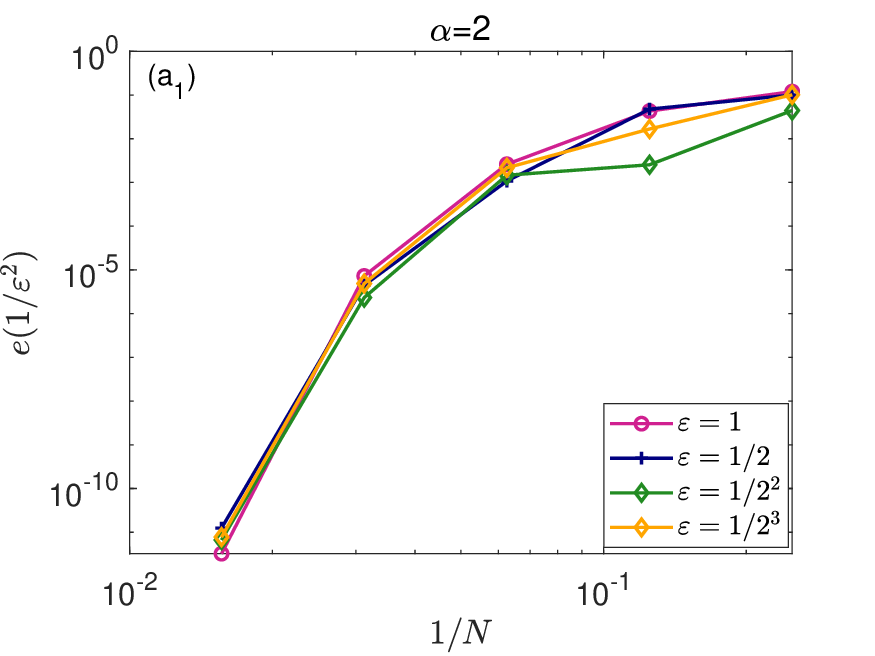}
\end{minipage}
}
\subfigure{
\begin{minipage}[t]{0.3\textwidth}
\centering
\includegraphics[width=5cm]{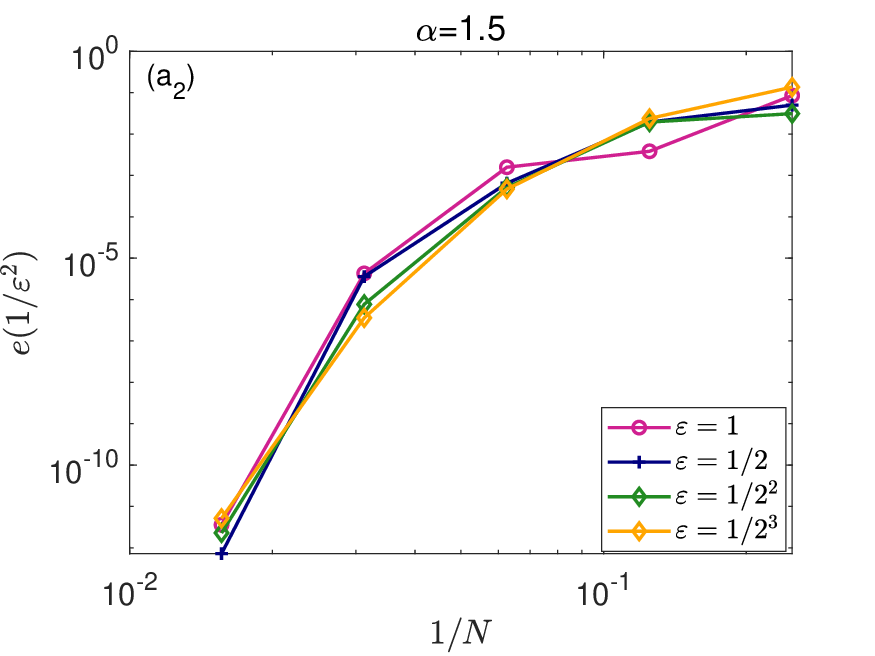}
\end{minipage}
}
\subfigure{
\begin{minipage}[t]{0.3\textwidth}
\centering
\includegraphics[width=5cm]{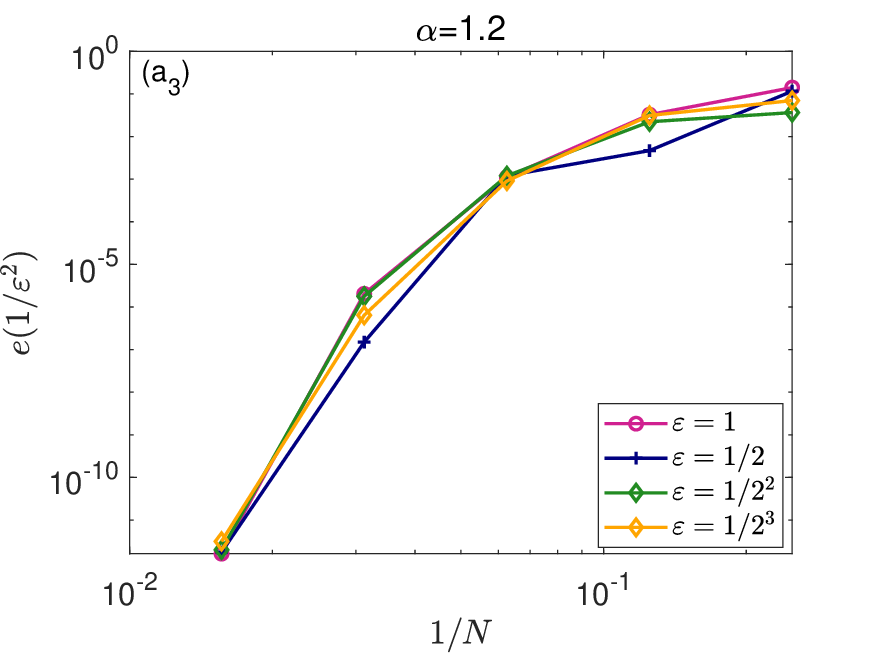}
\end{minipage}
}
\\
\subfigure{
\begin{minipage}[t]{0.3\textwidth}
\centering
\includegraphics[width=5cm]{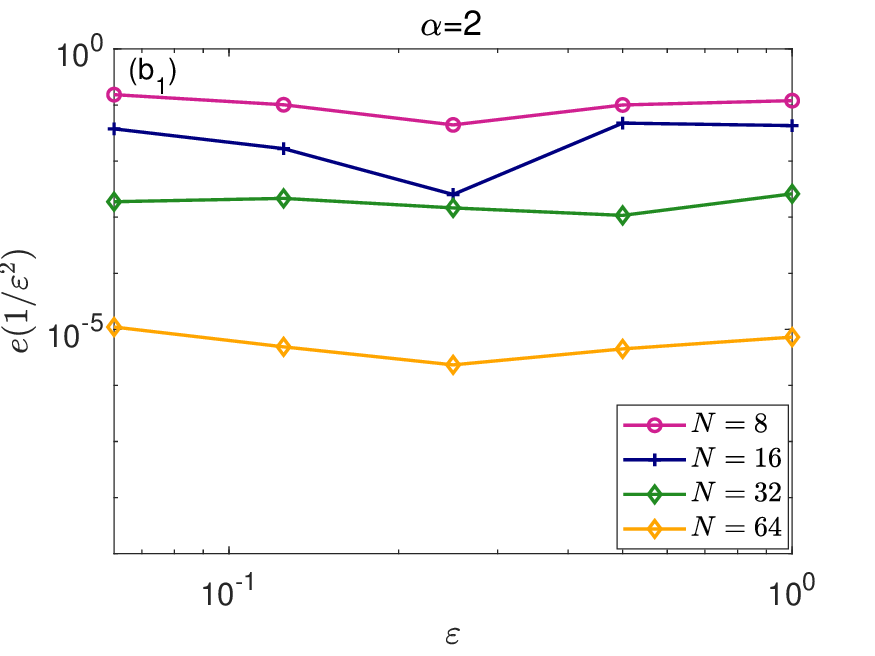}
\end{minipage}
}
\subfigure{
\begin{minipage}[t]{0.3\textwidth}
\centering
\includegraphics[width=5cm]{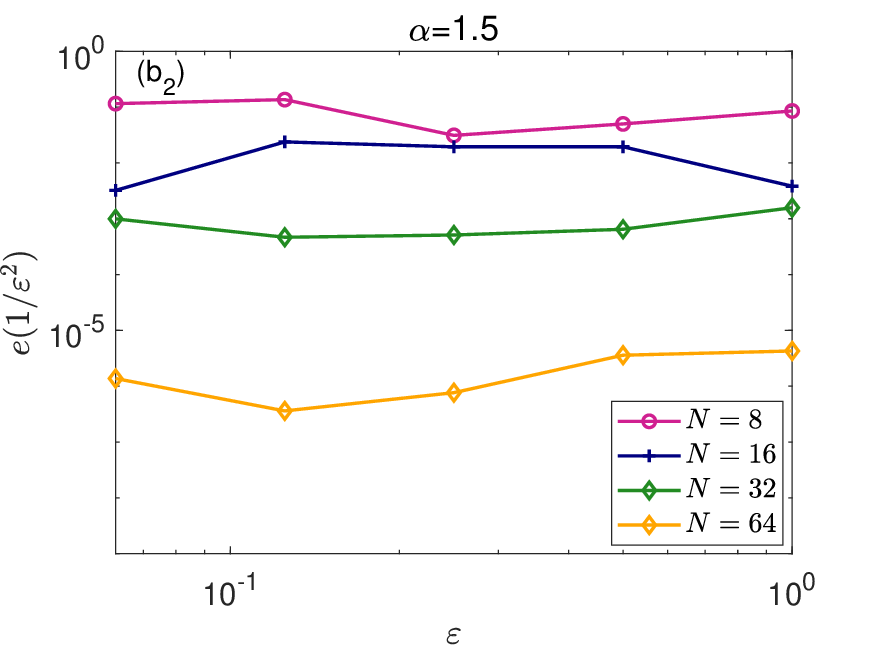}
\end{minipage}
}
\subfigure{
\begin{minipage}[t]{0.3\textwidth}
\centering
\includegraphics[width=5cm]{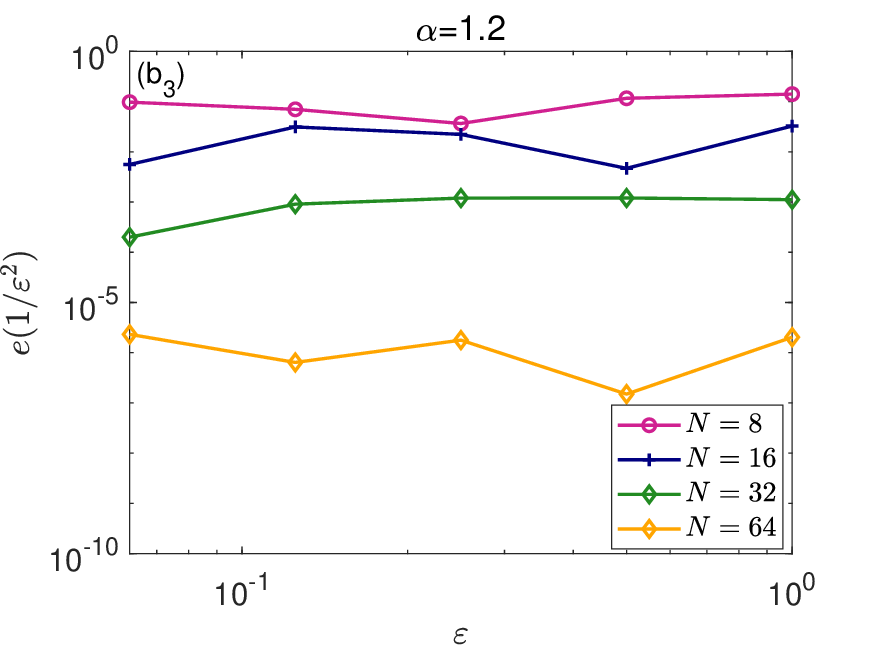}
\end{minipage}
}
\caption{ Long-time spatial errors for the NSFSGE  in 3D at $t=1/\varepsilon^2$ with different $\alpha$. }
\label{fig4}
\end{figure*}

\begin{figure*}[htbp]
\centering
\subfigure{
\begin{minipage}[t]{0.3\textwidth}
\centering
\includegraphics[width=5cm]{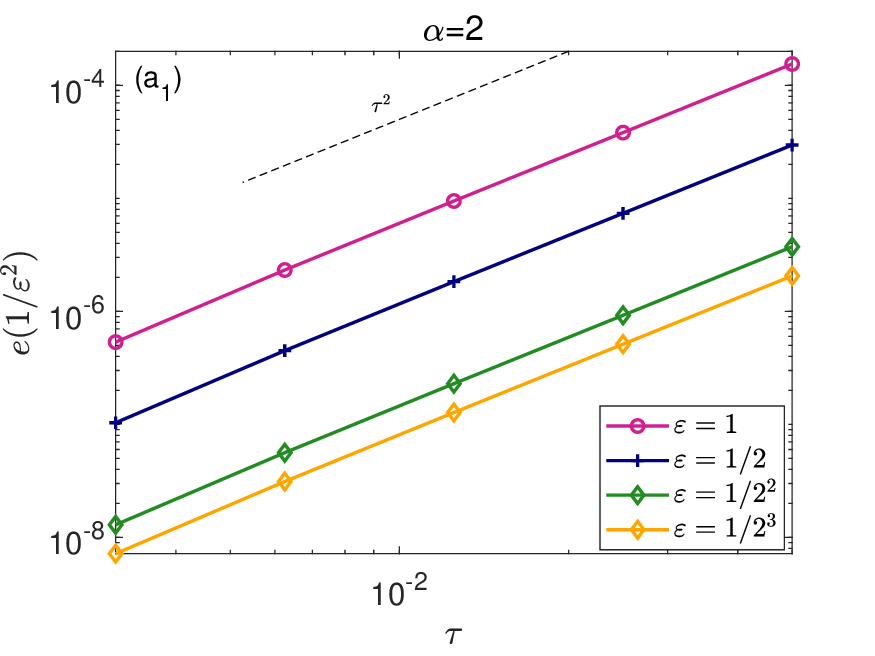}
\end{minipage}
}
\subfigure{
\begin{minipage}[t]{0.3\textwidth}
\centering
\includegraphics[width=5cm]{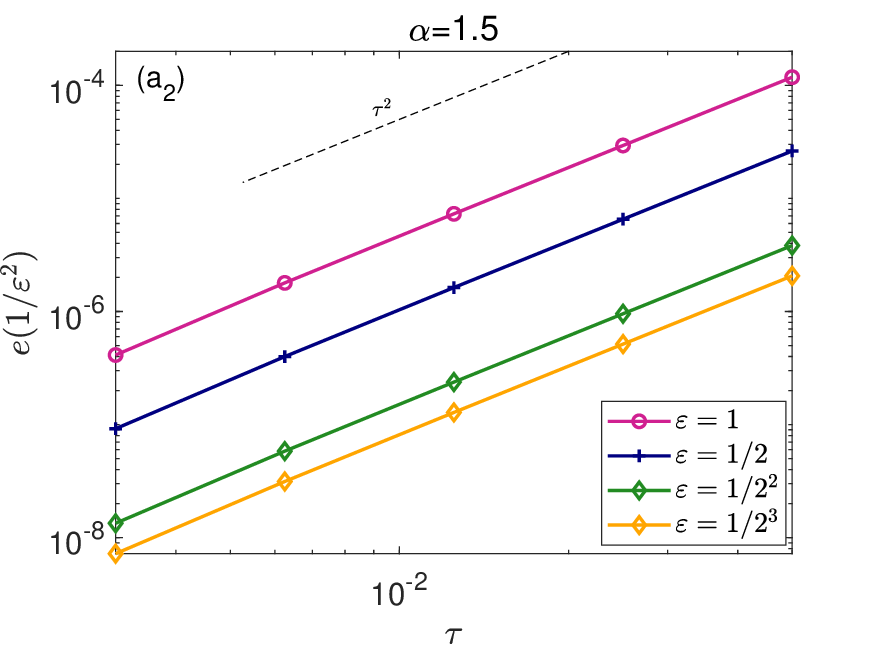}
\end{minipage}
}
\subfigure{
\begin{minipage}[t]{0.3\textwidth}
\centering
\includegraphics[width=5cm]{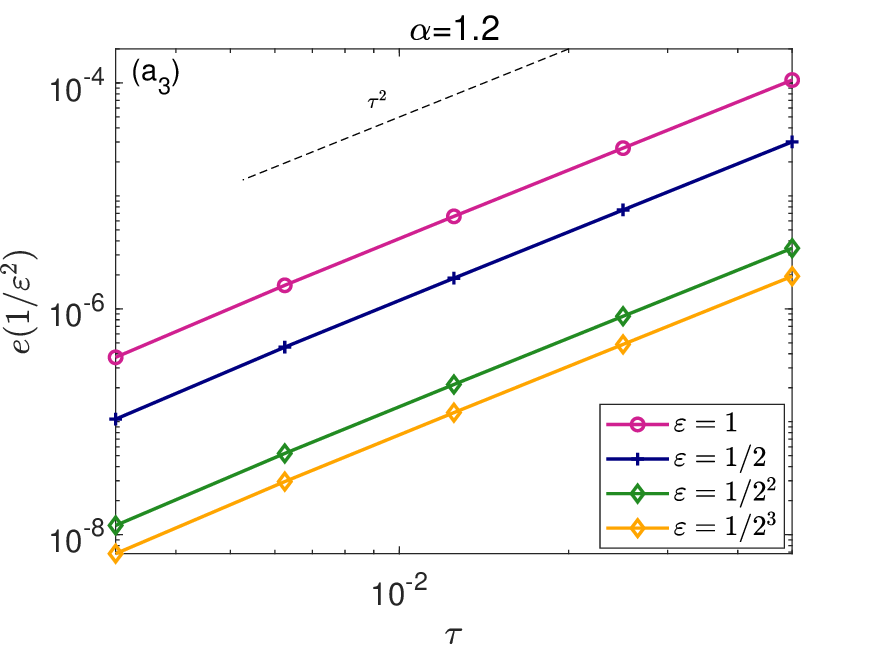}
\end{minipage}
}
\\
\subfigure{
\begin{minipage}[t]{0.3\textwidth}
\centering
\includegraphics[width=5cm]{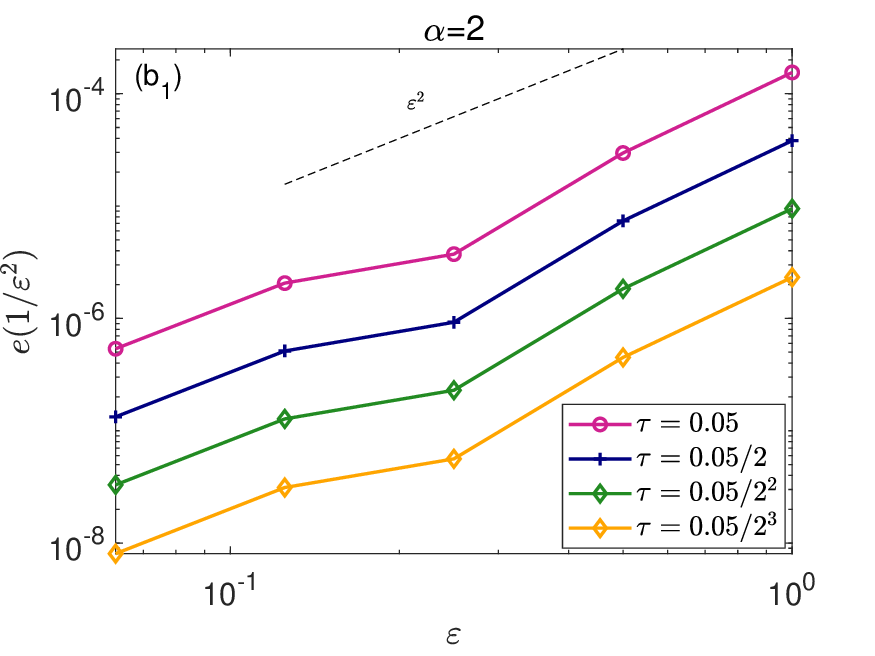}
\end{minipage}
}
\subfigure{
\begin{minipage}[t]{0.3\textwidth}
\centering
\includegraphics[width=5cm]{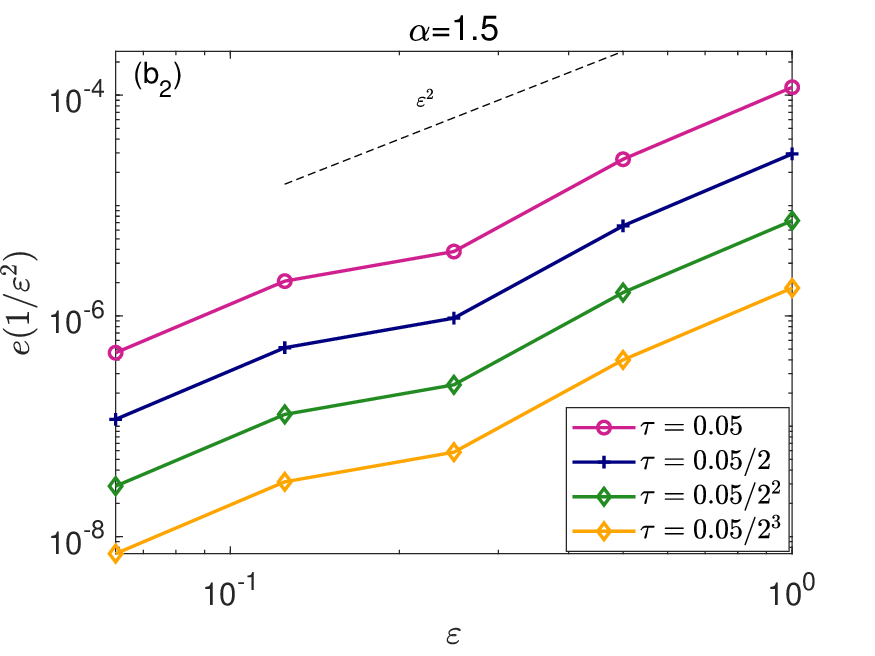}
\end{minipage}
}
\subfigure{
\begin{minipage}[t]{0.3\textwidth}
\centering
\includegraphics[width=5cm]{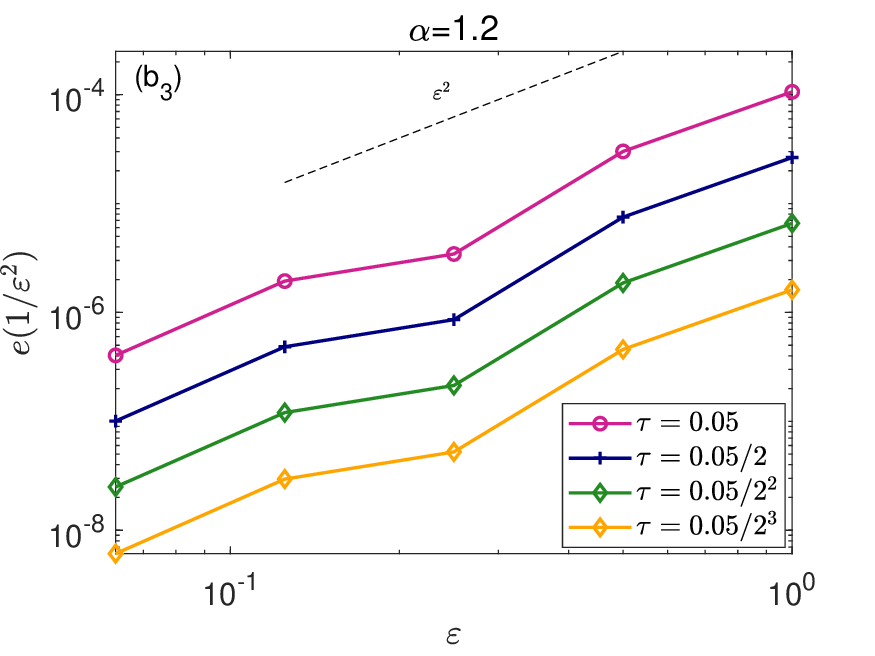}
\end{minipage}
}
\caption{ Long-time temporal errors for the NSFSGE  in 3D at $t=1/\varepsilon^2$ with different $\alpha$. }
\label{fig5}
\end{figure*}

\begin{figure*}[htbp]
\centering
\subfigure{
\begin{minipage}[t]{0.3\textwidth}
\centering
\includegraphics[width=5cm]{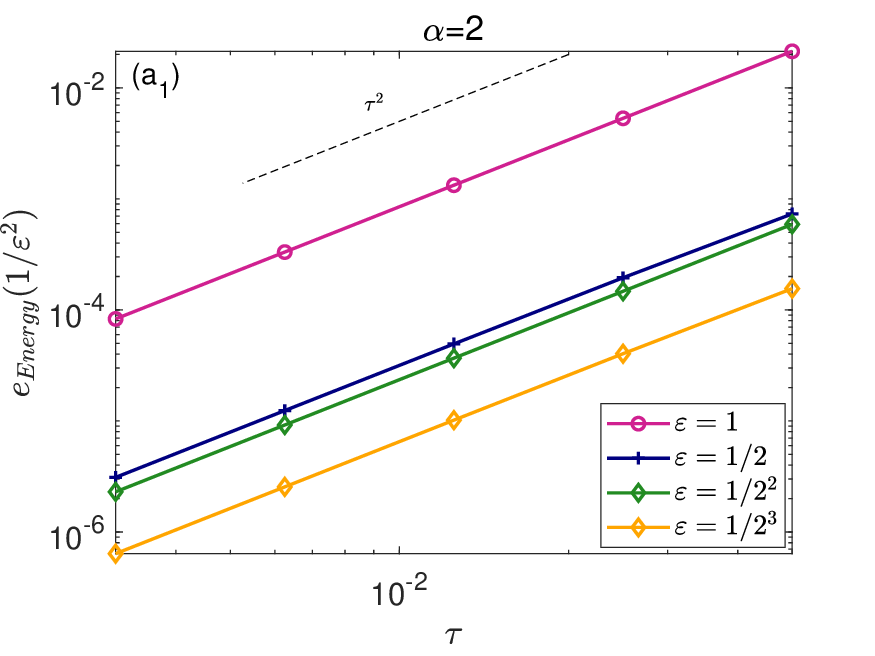}
\end{minipage}
}
\subfigure{
\begin{minipage}[t]{0.3\textwidth}
\centering
\includegraphics[width=5cm]{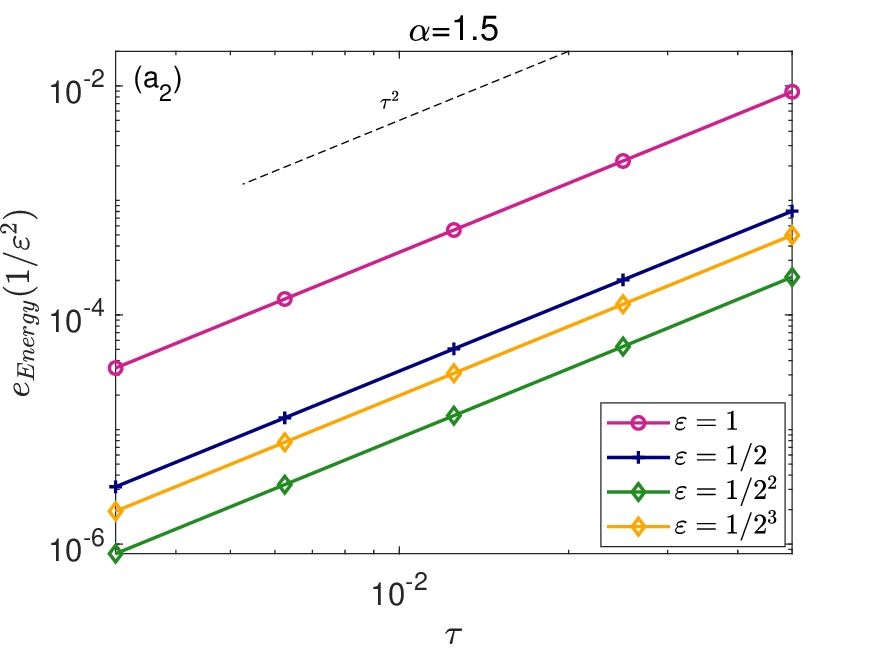}
\end{minipage}
}
\subfigure{
\begin{minipage}[t]{0.3\textwidth}
\centering
\includegraphics[width=5cm]{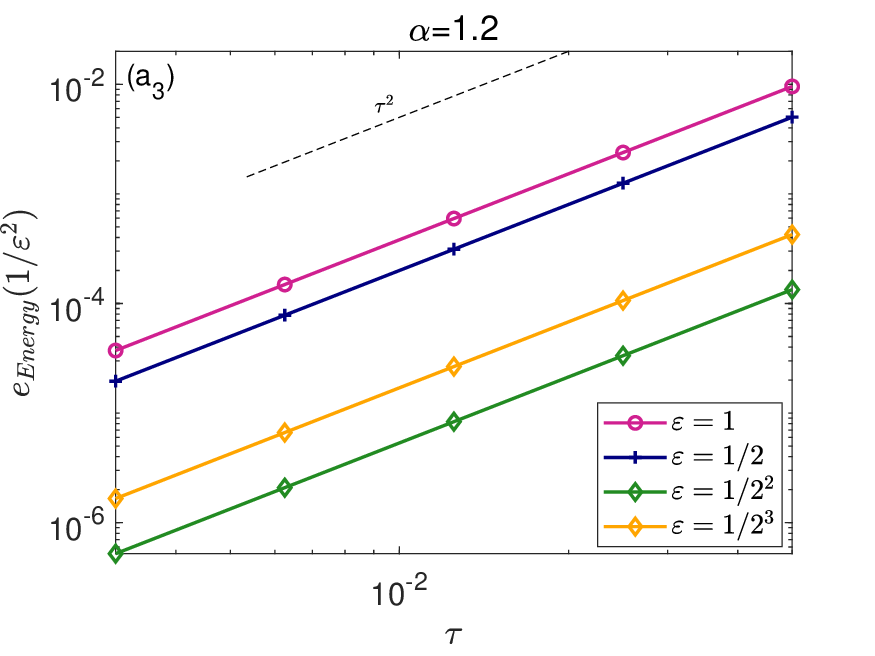}
\end{minipage}
}
\\
\subfigure{
\begin{minipage}[t]{0.3\textwidth}
\centering
\includegraphics[width=5cm]{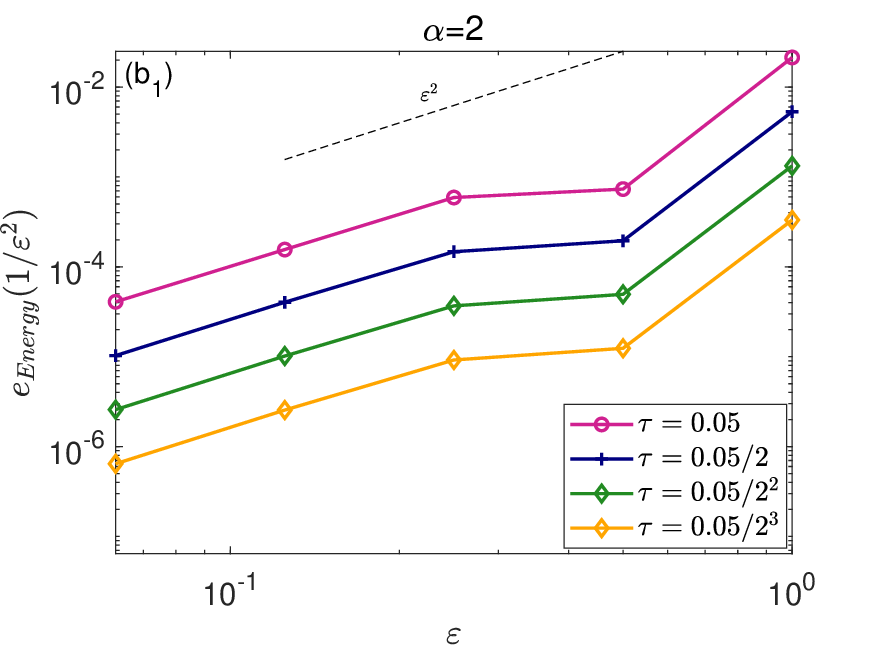}
\end{minipage}
}
\subfigure{
\begin{minipage}[t]{0.3\textwidth}
\centering
\includegraphics[width=5cm]{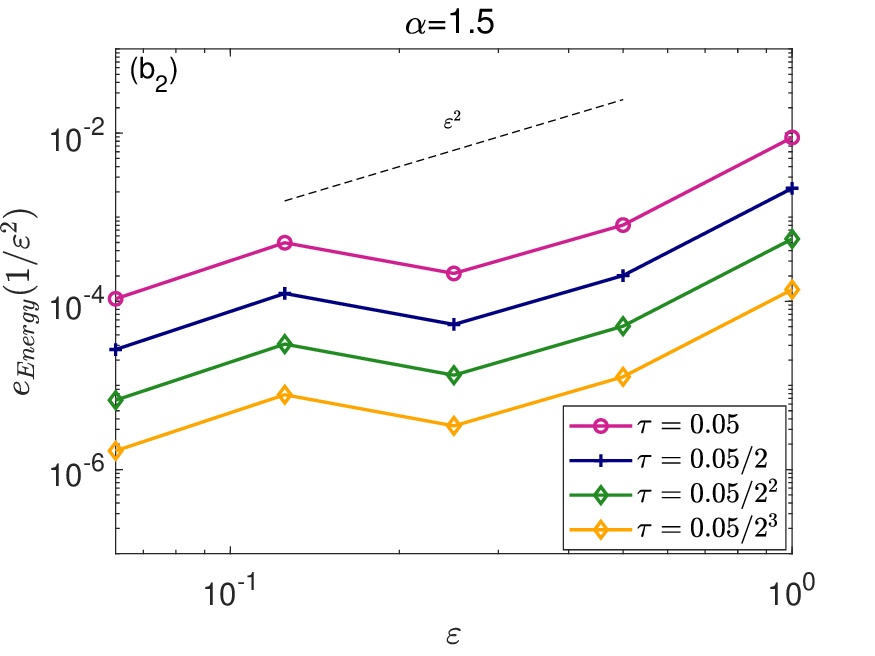}
\end{minipage}
}
\subfigure{
\begin{minipage}[t]{0.3\textwidth}
\centering
\includegraphics[width=5cm]{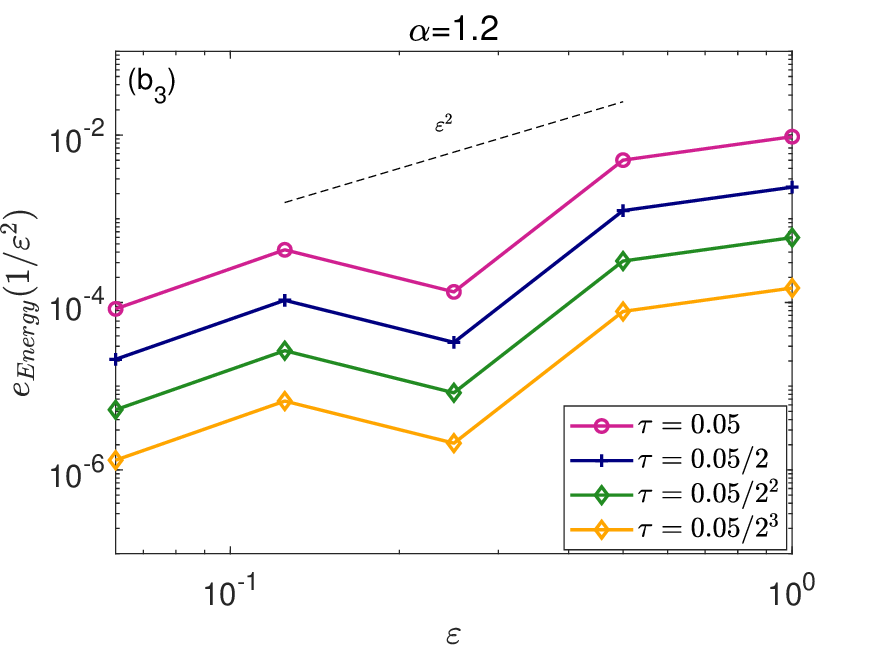}
\end{minipage}
}
\caption{ Long-time temporal errors of the discrete energy for the NSFSGE  in 3D at $t=1/\varepsilon^2$ with different $\alpha$. }
\label{fig6}
\end{figure*}

The numerical solutions with spatial nodes $N=128$ and time step $\tau=10^{-3}$  are used as the `exact' solutions for comparison. Figure \ref{fig4} shows the long-time spatial errors for different $N$ and $\varepsilon$ at $t=1/\varepsilon^2$ with $\alpha$ taken as 2, 1.5, 1.2.  Similar to the 2D case, it is easy to see the spatial errors decay exponentially with $N$  for different fractional order $\alpha$  and vary  little with the decreasing $\varepsilon$.  Figure \ref{fig5} shows the long-time temporal errors  for different $\tau$ and $\varepsilon$  at $t=1/\varepsilon^2$ when $\alpha$ is taken as 2, 1.5, 1.2. From Figure  \ref{fig5}, we can  find that the temporal errors vary like $O(\varepsilon^2\tau^2)$ up to the time at $O(1/\varepsilon^2)$ and  the fractional order $\alpha$ has no effect on the temporal errors,  which shows the  improved uniform error bound (\ref{c39}) holds in 3D.  In addition, the improved uniform error bound $O(\varepsilon^2\tau^2)$ of the discrete energy for the NSFSGE in 3D  is further displayed in Figure \ref{fig6}.

\subsection{The long-time dynamics for the oscillatory complex NSFSGE }
We give some numerical results for the  oscillatory complex NSFSGE (\ref{d6}) in 2D  to demonstrate the improved error bound (\ref{d7}).  We take $p=1$ and the initial values as
%\begin{equation}\label{o1}
%z_0(x, y)=x^2(x-1)^2+y(y-1)+3, \quad z_1(x, y)=x(x-1)(2 x-1)+y^2(y-1)^2 +i\cos(2\pi x)\cos(2\pi y), \quad(x, y) \in \Omega=(0,1)^2.
%\end{equation}
\begin{equation}\label{o1}
u_0(x, y)=x^2(x-1)^2+y(y-1)+6i, \quad u_1(x, y)=x(x-1)(2 x-1)+y^2(y-1)^2 +i\cos(2\pi x+2\pi y), \quad(x, y) \in \Omega=(0,1)^2.
\end{equation}

The numerical solutions with spatial nodes $N=128$ and time step $\lambda=10^{-5}$  are chosen as the `exact' solutions for comparison. The temporal errors are given for  different fractional order $\alpha$ at $t=1$ in Tables \ref{tab1}-\ref{tab3}. We can see the second-order convergence can be obtained for different $\alpha$ only when $\lambda\le \varepsilon^2$, i.e., the upper triangle above the diagonal in Tables \ref{tab1}-\ref{tab3}, which  is in agreement with the theoretical analysis.  This example indicates the TSFP method and the  improved error bound (\ref{d7}) are effective for the oscillatory complex NSFSGE.
\begin{table}[htbp]\small
\caption{\label{tab:test} Long-time temporal errors for the oscillatory complex NSFSGE in 2D when $\alpha=2$.}
\centering
\begin{tabular}{ccccccccccccccccccccccccccccc}
\hline
 &       $e(t=1)$                 && $\lambda_0=0.05$     && $\lambda_0/4$         && $\lambda_0/4^2$     &&  $\lambda_0/4^3$   &&  $\lambda_0/4^4$                \\
\hline
 &      $\varepsilon=1$           && \textbf{4.5038e-01} &&  2.5566e-02  && 1.5881e-03 &&  9.9206e-05 &&  6.1850e-06                          \\
 &     order                      && \textbf{$-$}  &&     2.0694  &&  2.0045 &&   2.0003  &&  2.0018 \\
 \hline
 &      $\varepsilon/2$           &&  1.3076e-01 &&  \textbf{6.3988e-03} &&  3.9449e-04  && 2.4628e-05   && 1.5354e-06                      \\
 &     order                      && $-$ &&\textbf{2.1765} &&   2.0099 &&   2.0008  &&  2.0018 \\
 \hline
 &      $\varepsilon/2^2$         &&   7.8055e-01 &&  3.9508e-02 && \textbf{ 2.4379e-03}   && 1.5197e-04  && 9.4706e-06                            \\
  &     order                     && $-$ && 2.1521   && \textbf{2.0092}  &&  2.0019  &&  2.0021\\
 \hline
 &      $\varepsilon/2^3$         &&  2.9970 &&  9.7845e-02 &&  4.7132e-03  && \textbf{3.0666e-04}  && 1.7924e-05                    \\
 &     order                      && $-$ &&2.4684  &&  2.1879  &&  \textbf{1.9710 } &&  2.0483\\
 \hline
 &      $\varepsilon/2^4$         &&  9.5838e-01 &&  1.1325 &&  4.0318e-02 &&  2.2330e-03 &&  \textbf{1.4494e-04}                          \\
  &     order                     && $-$ &&-0.1204  &&  2.4060 &&   2.0872 &&   \textbf{1.9728} \\
 \hline
\end{tabular}
\label{tab1}
\end{table}

\begin{table}[htbp]\small
\caption{\label{tab:test} Long-time temporal errors for the oscillatory complex NSFSGE in 2D when $\alpha=1.5$.}
\centering
\begin{tabular}{ccccccccccccccccccccccccccccc}
\hline
 &       $e(t=1)$                 && $\lambda_0=0.05$     && $\lambda_0/4$         && $\lambda_0/4^2$     &&  $\lambda_0/4^3$   &&  $\lambda_0/4^4$                \\
\hline
 &      $\varepsilon=1$           && \textbf{   4.4357e-01 }  &&         2.5201e-02 &&   1.5669e-03  &&  9.7883e-05  &&  6.1025e-06                          \\
 &     order                      && \textbf{$-$}  && 2.0688  &&  2.0038  &&  2.0003  &&  2.0018 \\
 \hline
 &      $\varepsilon/2$           &&    1.2110e-01 &&   \textbf{6.3782e-03}  &&  3.9454e-04 &&   2.4639e-05 &&   1.5361e-06                 \\
 &     order                      && $-$ &&    \textbf{2.1234} &&   2.0074  &&  2.0006  &&  2.0018 \\
 \hline
 &      $\varepsilon/2^2$         &&   7.6497e-01 &&   3.9335e-02 &&   \textbf{2.4323e-03} &&   1.5189e-04 &&   9.4695e-06                           \\
  &     order                     && $-$ && 2.1408  &&  \textbf{2.0077}  &&  2.0006  &&  2.0018\\
 \hline
 &      $\varepsilon/2^3$        &&3.0088 && 9.2594e-02  && 4.6984e-03 &&   \textbf{2.9035e-04}  && 1.8089e-05    \\
 &     order                      && $-$ && 2.5111   && 2.1503 &&   \textbf{2.0082}   && 2.0023\\
 \hline
 &      $\varepsilon/2^4$         &&   9.6359e-01  &&  1.1328&&   4.0355e-02 &&   2.2183e-03 &&    \textbf{1.3740e-04}                      \\
  &     order                     && $-$  && -0.1167   && 2.4055 &&   2.0926  &&  \textbf{2.0065} \\
 \hline
\end{tabular}
\label{tab2}
\end{table}

\begin{table}[htbp]\small
\caption{\label{tab:test} Long-time temporal errors for the oscillatory complex NSFSGE in 2D when $\alpha=1.2$.}
\centering
\begin{tabular}{ccccccccccccccccccccccccccccc}
\hline
 &       $e(t=1)$                 && $\lambda_0=0.05$     && $\lambda_0/4$         && $\lambda_0/4^2$     &&  $\lambda_0/4^3$   &&  $\lambda_0/4^4$                \\
\hline
 &      $\varepsilon=1$           && \textbf{ 4.4516e-01}  &&       2.5164e-02   &&   1.5646e-03  &&    9.7738e-05  &&    6.0935e-06                          \\
 &     order                      && \textbf{$-$}  &&  2.0724  &&  2.0038   && 2.0003  &&  2.0018 \\
 \hline
 &      $\varepsilon/2$           &&     1.1653e-01  && \textbf{5.9772e-03}  &&   3.6915e-04  &&   2.3051e-05  &&   1.4371e-06             \\
 &     order                      && $-$ &&  \textbf{2.1426} &&  2.0086   && 2.0006 &&   2.0018 \\
 \hline
 &      $\varepsilon/2^2$         &&     1.6066 &&   3.9512e-02 &&   \textbf{2.4438e-03}&&   1.5261e-04  &&  9.5146e-06                          \\
  &     order                     && $-$ &&    2.6728 &&   \textbf{2.0076}   && 2.0006  &&  2.0018\\
 \hline
 &      $\varepsilon/2^3$         &&      2.9835    &&   2.2866e-01     &&  4.8089e-03    &&  \textbf{ 2.9648e-04}    &&   1.8469e-05  \\
 &     order                      && $-$ &&  1.8528    && 2.7857    && \textbf{2.0098}    && 2.0024\\
 \hline
 &      $\varepsilon/2^4$         &&   9.7440e-01 &&  1.1455 && 4.6627e-02  && 2.2205e-03 &&  \textbf{1.3759e-04 }                       \\
  &     order                     && $-$ && -0.1167   &&  2.3094   &&  2.1961   &&  \textbf{2.0062} \\
 \hline
\end{tabular}
\label{tab3}
\end{table}

\subsection{Applications}
In this subsection,  we give some  applications  to demonstrate  the differences in the dynamic behaviors between the fractional sine-Gordon equation and the classical  sine-Gordon equation. Numerical results of the NSFSGE with different initial conditions are obtained by the TSFP method.

\subsubsection{Elliptical ring soliton in 2D}
We consider the elliptical ring solitons with the initial conditions as \cite{40}
\begin{equation}\label{fel1}
\begin{aligned}
& u_0(x, y)=4 \tan ^{-1}\left(\exp \left(3-\sqrt{\frac{(x-y)^2}{3}+\frac{(x+y)^2}{2}}\right)\right),\quad
& u_1(x, y)=0, \quad  -7\le x, y \le 7.
\end{aligned}
\end{equation}
Here, we choose $ N=256, \tau=10^{-3}$ to discrete the space and time domain, respectively. The numerical solutions of the NSFSGE in terms of $\sin(u/2)$ with different $\alpha$ at times $t=0, 3, 5, 9, 12, 15$  are given in Figures \ref{afig1}-\ref{afig3}.  We can observe a clearly elliptical ring soliton at $t=0$. Then there is a small disturbance  in the soliton  and the soliton starts  to shrink as the time changes, which can be seen at $t=3$ and  $t=5$ clearly.  From $t=9$, the soliton enters the expansion phase.  This expansion is continued until $t=12$, where the soliton is nearly  reformed. Finally, it seems to be in a shrinking phase again at $t = 15$. Note that when $\alpha = 2$, the nonlinear space fractional sine-Gordon equation reduces to the classical nonlinear sine-Gordon equation. The evolution trend of solitons in Figure \ref{afig1} is consistent with the results in References \cite{40, 41, 42, 43}. Furthermore,  we can find the shape of the soliton changes with fractional order $\alpha$ varying, and the shape of the soliton changes more dramatically with  $\alpha$ smaller.

\begin{figure*}[htbp]
\centering
\subfigure{
\begin{minipage}[t]{0.3\textwidth}
\centering
\includegraphics[width=5cm]{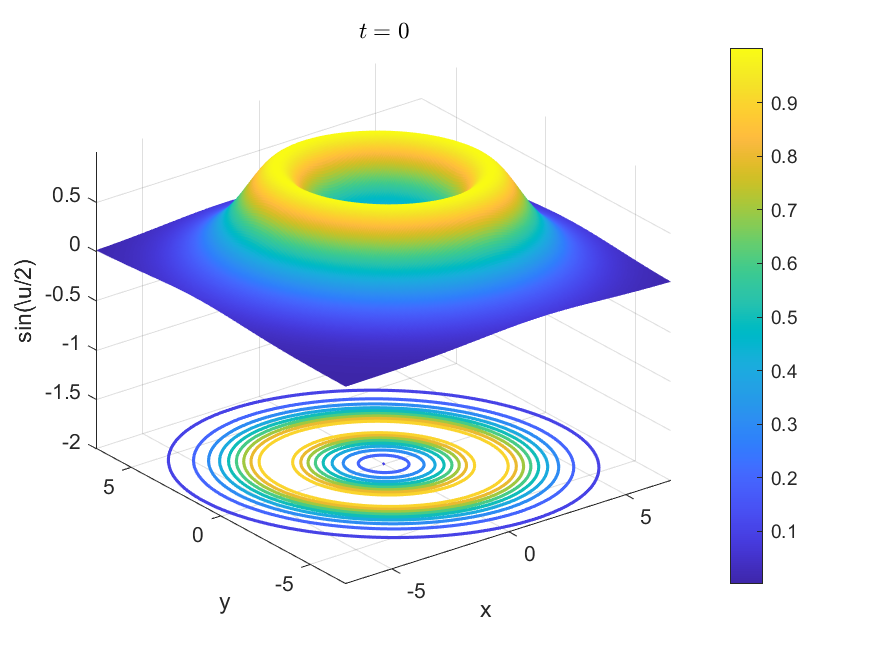}
\end{minipage}
}
\subfigure{
\begin{minipage}[t]{0.3\textwidth}
\centering
\includegraphics[width=5cm]{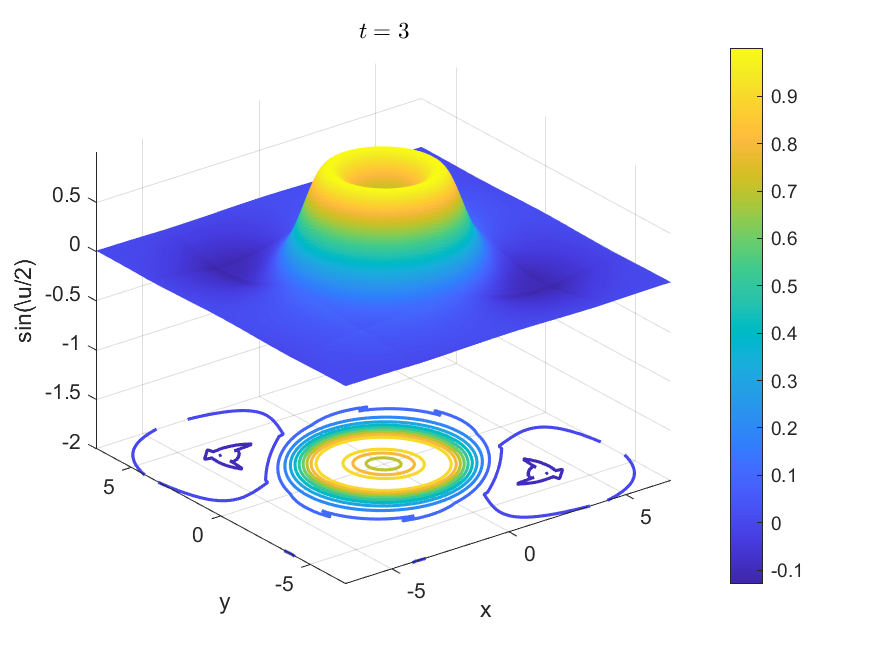}
\end{minipage}
}
\subfigure{
\begin{minipage}[t]{0.3\textwidth}
\centering
\includegraphics[width=5cm]{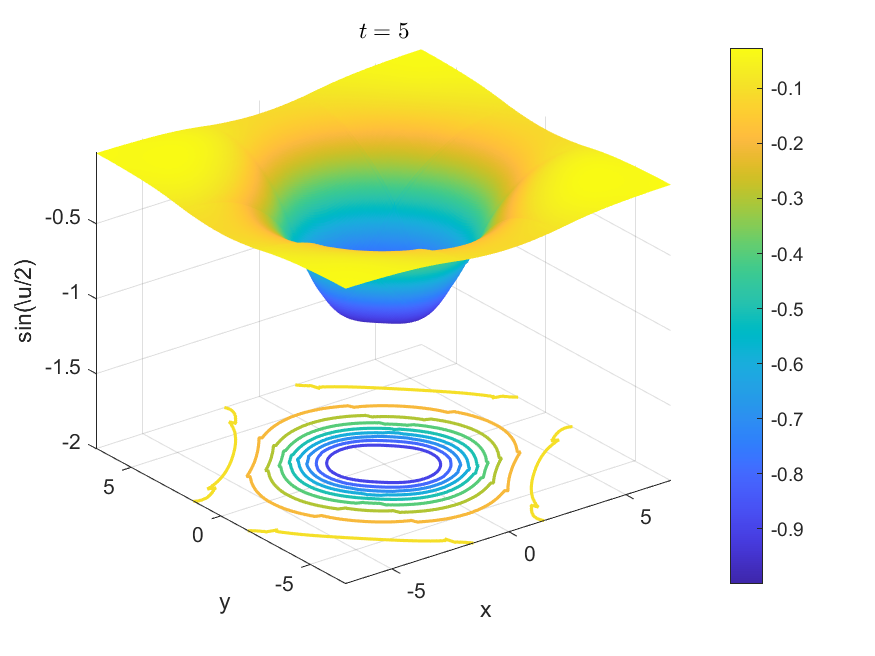}
\end{minipage}
}
\\
\subfigure{
\begin{minipage}[t]{0.3\textwidth}
\centering
\includegraphics[width=5cm]{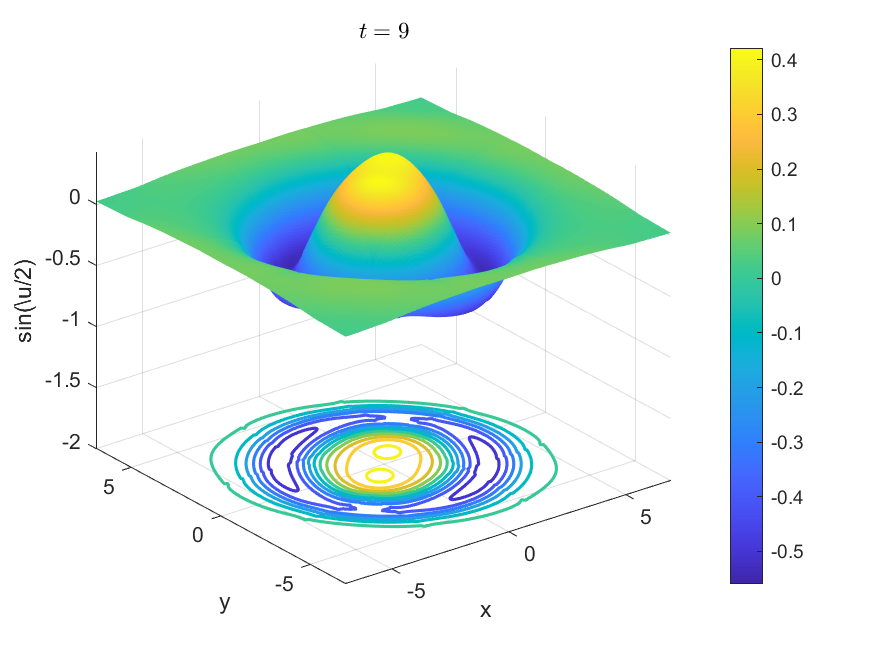}
\end{minipage}
}
\subfigure{
\begin{minipage}[t]{0.3\textwidth}
\centering
\includegraphics[width=5cm]{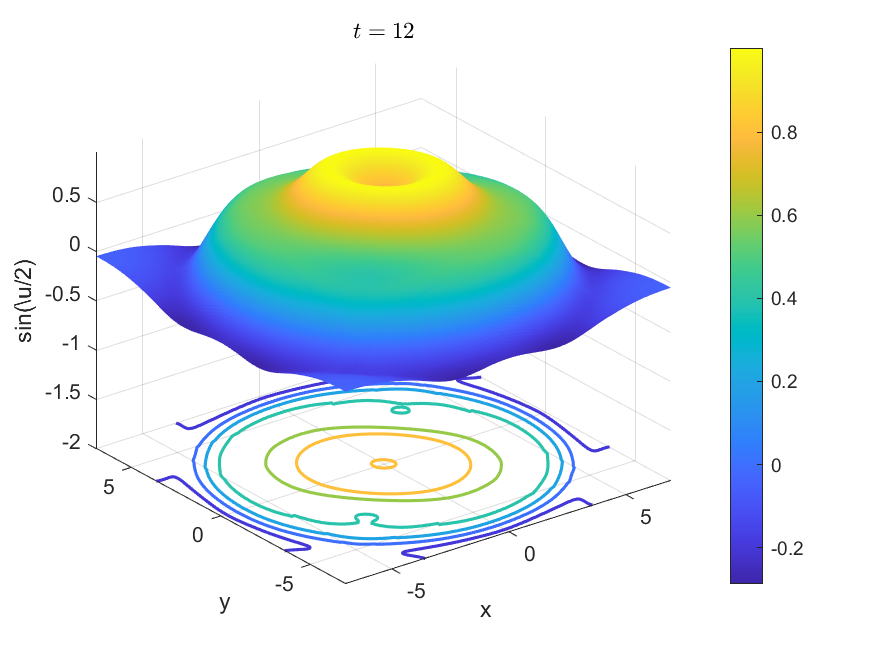}
\end{minipage}
}
\subfigure{
\begin{minipage}[t]{0.3\textwidth}
\centering
\includegraphics[width=5cm]{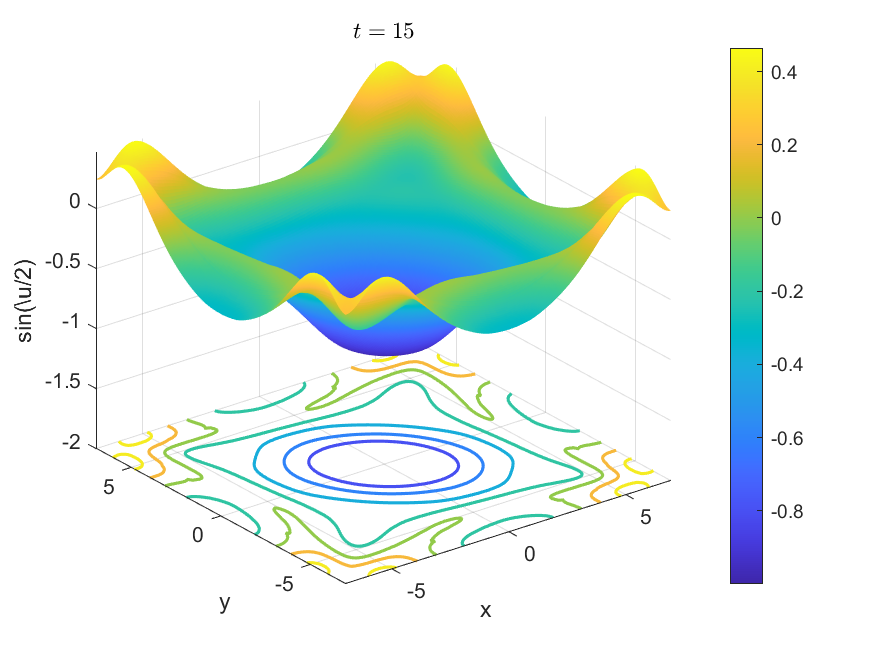}
\end{minipage}
}
\caption{Surfaces and contour plots of  the elliptical ring soliton when $\alpha=2$. }
\label{afig1}
\end{figure*}

\begin{figure*}[htbp]
\centering
\subfigure{
\begin{minipage}[t]{0.3\textwidth}
\centering
\includegraphics[width=5cm]{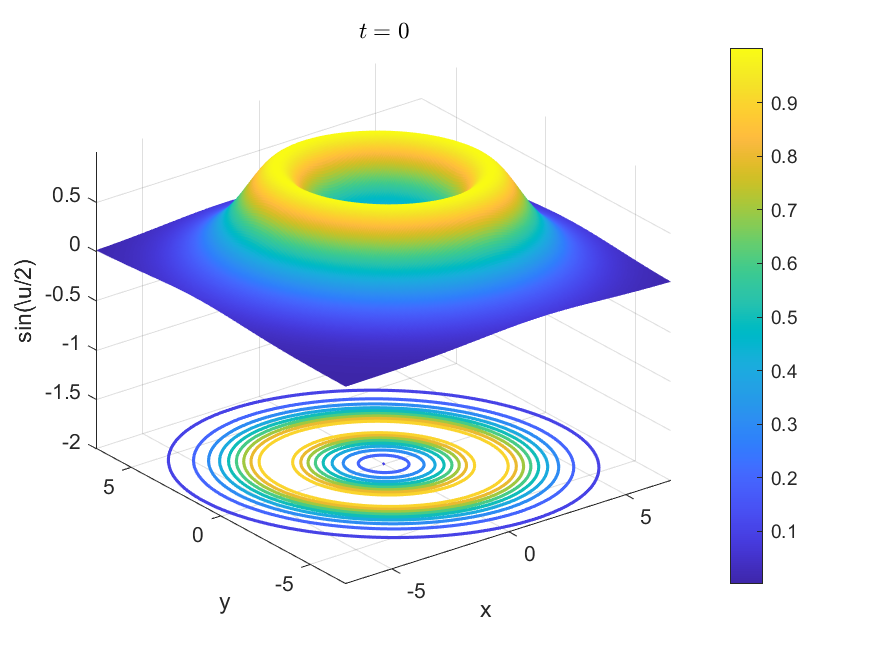}
\end{minipage}
}
\subfigure{
\begin{minipage}[t]{0.3\textwidth}
\centering
\includegraphics[width=5cm]{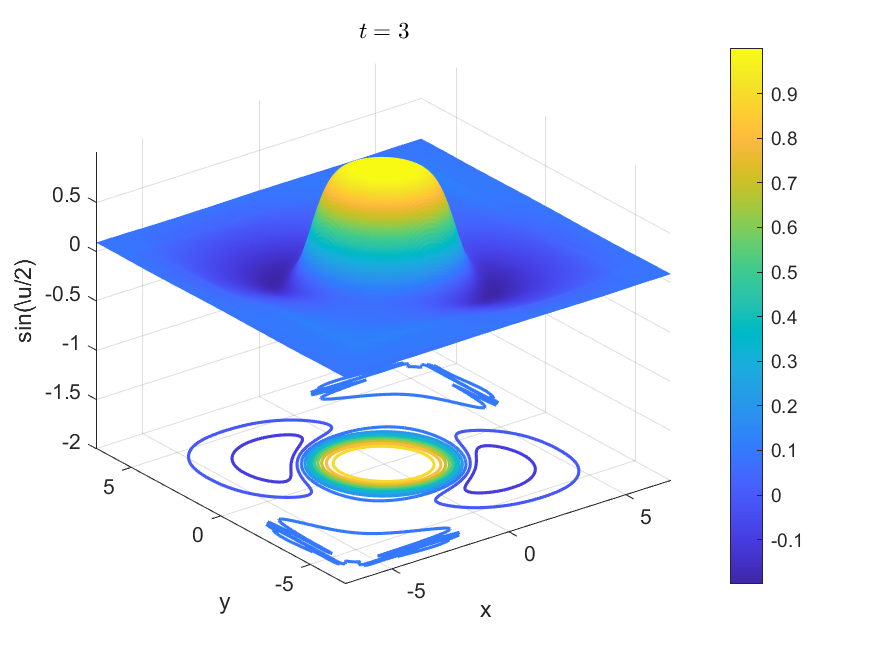}
\end{minipage}
}
\subfigure{
\begin{minipage}[t]{0.3\textwidth}
\centering
\includegraphics[width=5cm]{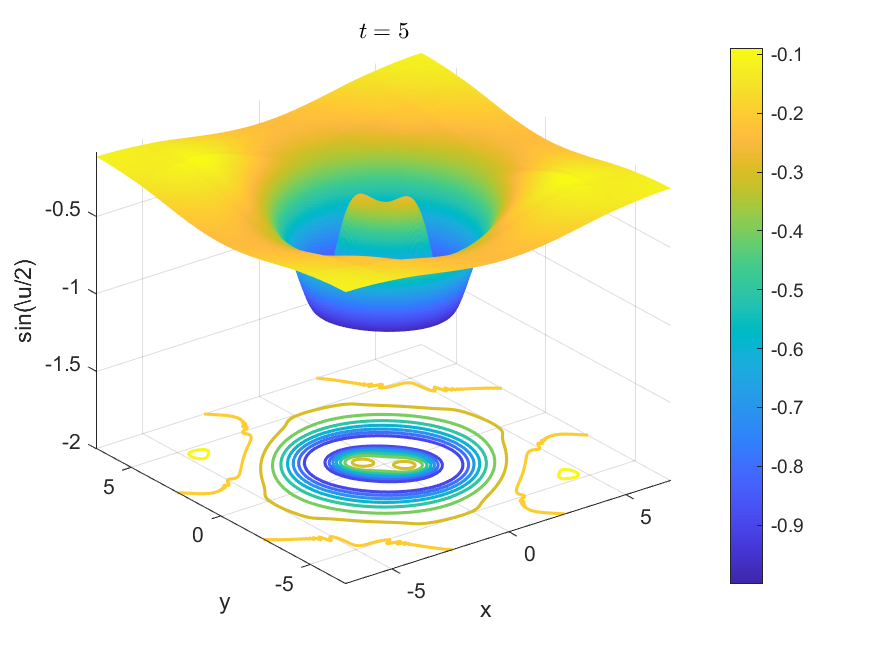}
\end{minipage}
}
\\
\subfigure{
\begin{minipage}[t]{0.3\textwidth}
\centering
\includegraphics[width=5cm]{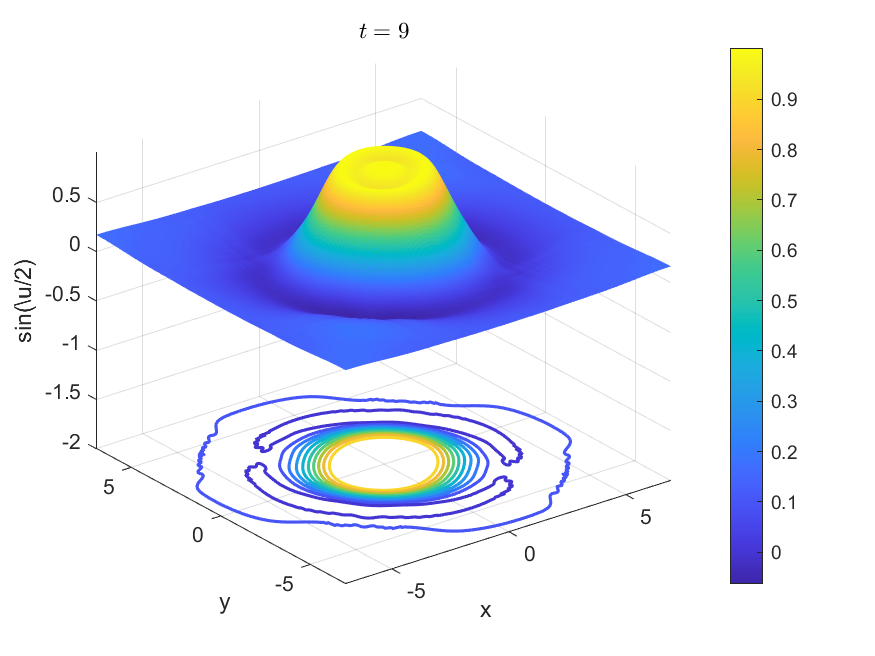}
\end{minipage}
}
\subfigure{
\begin{minipage}[t]{0.3\textwidth}
\centering
\includegraphics[width=5cm]{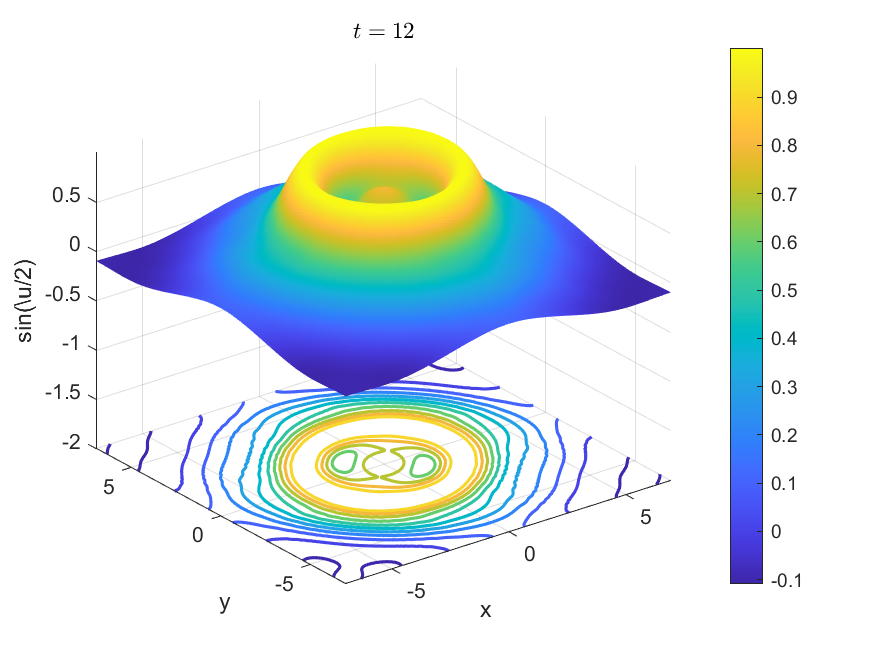}
\end{minipage}
}
\subfigure{
\begin{minipage}[t]{0.3\textwidth}
\centering
\includegraphics[width=5cm]{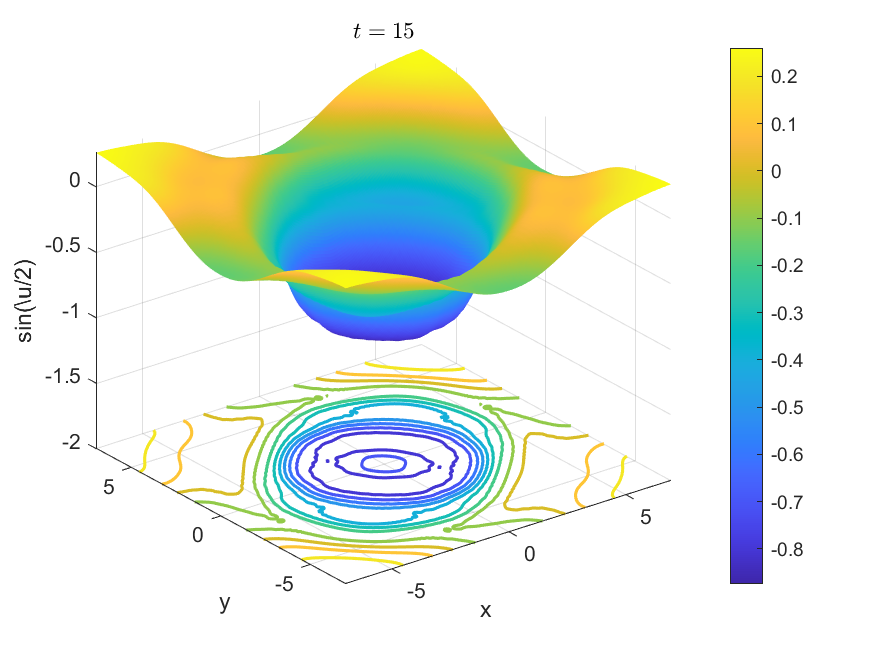}
\end{minipage}
}
\caption{ Surfaces and contour plots of  the elliptical ring soliton in 2D when $\alpha=1.5$. }
\label{afig2}
\end{figure*}

\begin{figure*}[htbp]
\centering
\subfigure{
\begin{minipage}[t]{0.3\textwidth}
\centering
\includegraphics[width=5cm]{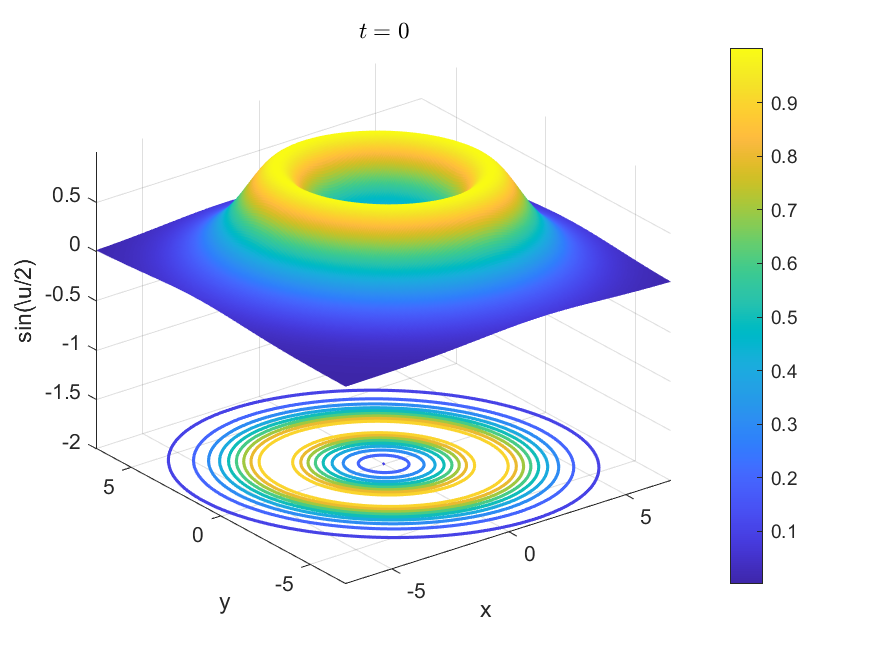}
\end{minipage}
}
\subfigure{
\begin{minipage}[t]{0.3\textwidth}
\centering
\includegraphics[width=5cm]{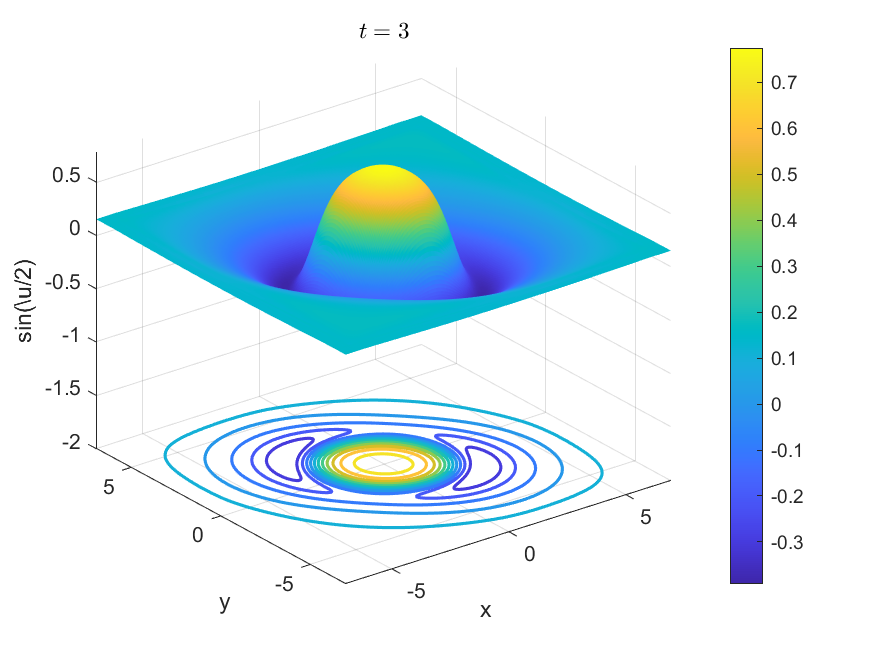}
\end{minipage}
}
\subfigure{
\begin{minipage}[t]{0.3\textwidth}
\centering
\includegraphics[width=5cm]{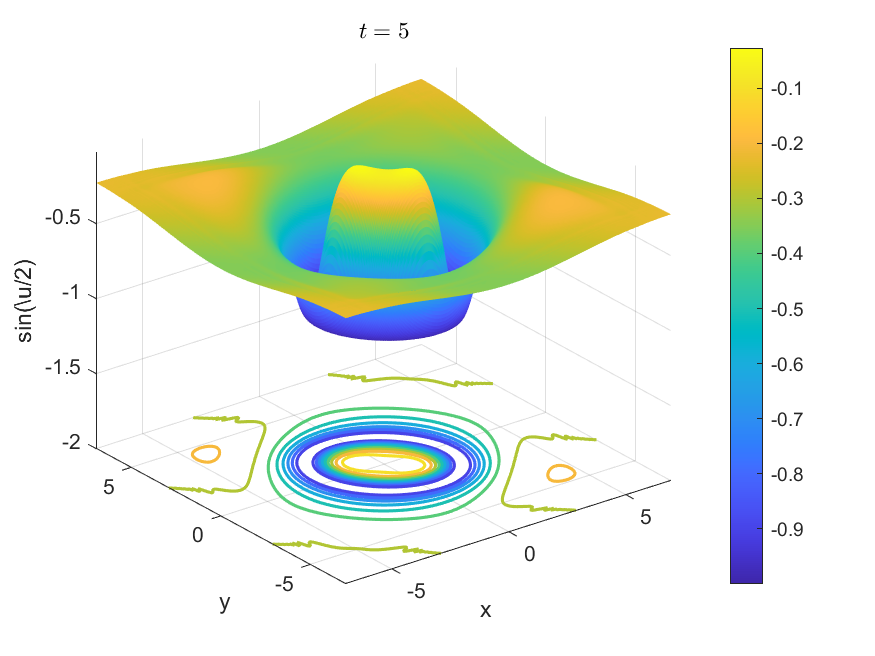}
\end{minipage}
}
\\
\subfigure{
\begin{minipage}[t]{0.3\textwidth}
\centering
\includegraphics[width=5cm]{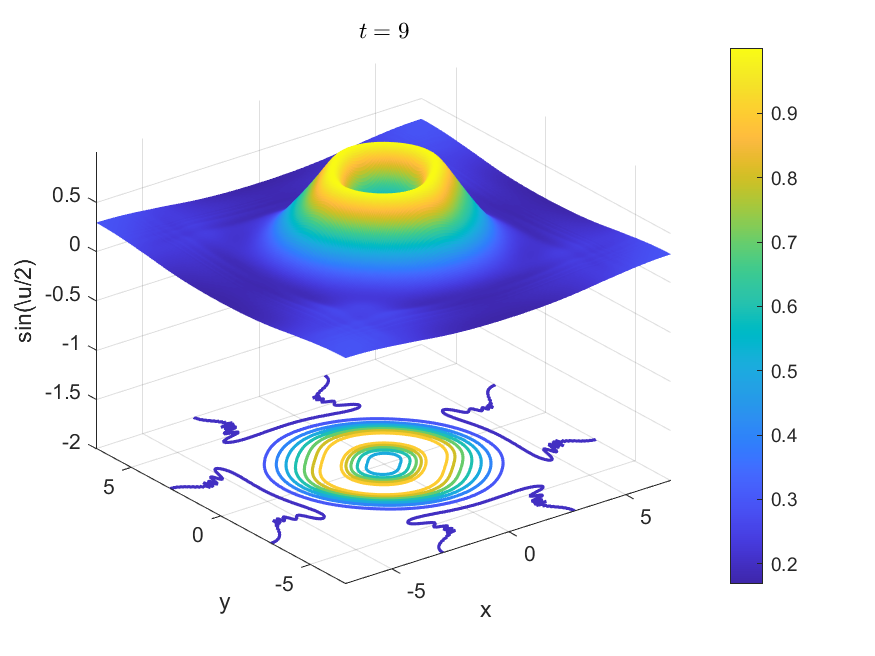}
\end{minipage}
}
\subfigure{
\begin{minipage}[t]{0.3\textwidth}
\centering
\includegraphics[width=5cm]{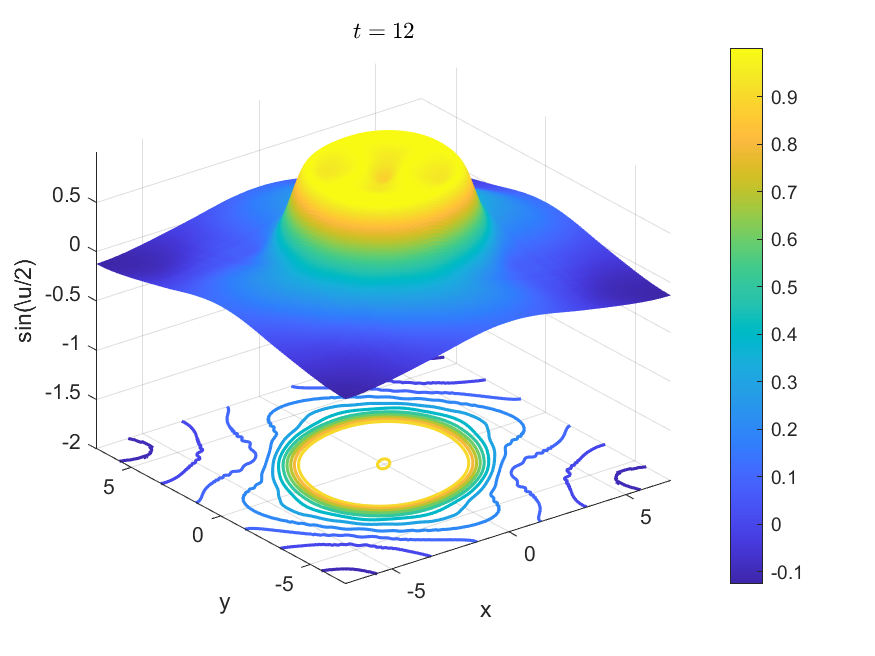}
\end{minipage}
}
\subfigure{
\begin{minipage}[t]{0.3\textwidth}
\centering
\includegraphics[width=5cm]{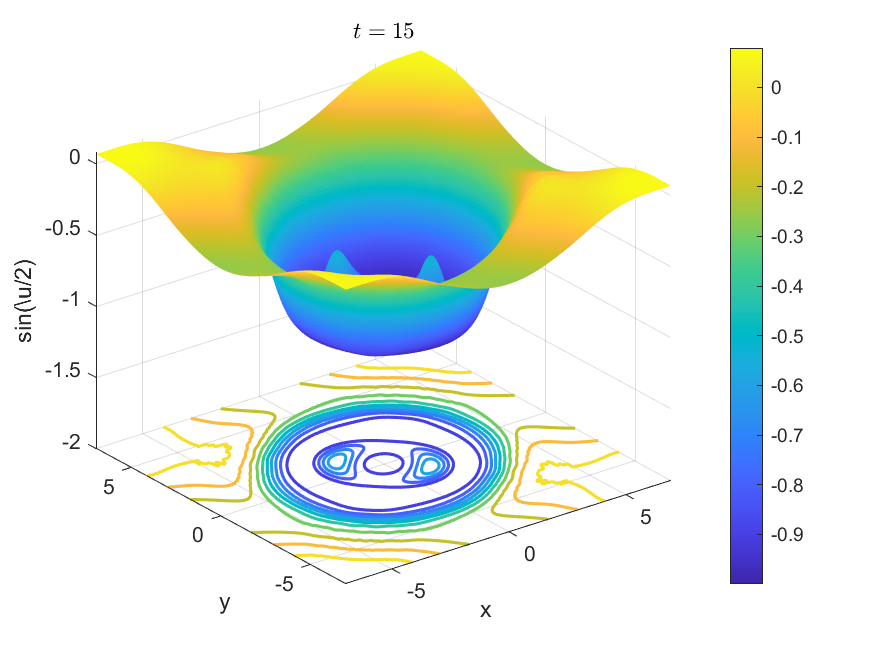}
\end{minipage}
}
\caption{ Surfaces and contour plots of elliptical ring soliton in 2D when $\alpha=1.2$. }
\label{afig3}
\end{figure*}
\subsubsection{Collisions of two circular solutions in 2D}
We consider the circular ring solitons with the initial conditions as \cite{44}
\begin{equation}
\begin{split}
u_0(x, y)&=4 \tan ^{-1}\left[\exp \left(\frac{4-\sqrt{(x+3)^2+(y+7)^2}}{0.436}\right)\right],\quad -30 \leq x \leq 10,\quad -21 \leq y \leq 7, \\
u_1(x, y)&=4.13 \operatorname{sech}\left(\frac{4-\sqrt{(x+3)^2+(y+7)^2}}{0.436}\right),\quad -30 \leq x \leq 10,\quad -21 \leq y \leq 7.
\end{split}
\end{equation}
%\begin{equation}\label{fel2}
%\begin{split}
%z_0(x, y) & =4 \tan ^{-1}\left[\exp \left(\frac{4-\sqrt{(x+3)^2+(y+7)^2}}{0.436}\right)\right],-30 \leq x, y \leq 10 \\
%z_1(x, y) & =\frac{4.13}{\cosh \left(\frac{4-\sqrt{(x+3)^2+(y+7)^2}}{0.436}\right)},-30 \leq x, y \leq 10 .
%\end{split}
%\end{equation}
We choose $ N=256, \tau=10^{-3}$ to discrete the space and time domain. According to the symmetry properties of the problem, the solutions including the extension across $x=-10$ and $y=-7$ are provided.  The surfaces and contour plots of two expanding circular ring solitons for different $\alpha$ at time $t=0, 2, 4, 6, 8, 10$ are shown in Figures \ref{afig4}-\ref{afig6}. From the contour plots, we can clearly observe the collision between two expanding circular ring solitons, where two smaller ring solitons emerge into a large ring soliton for different $\alpha$.  Moreover, the smaller the fractional order $\alpha$, the stronger the shape change of the two solitons. The evolution trend of two circular ring solitons at $\alpha=2$ is  consistent with the References \cite{ 44, 45, 46, 47}, which verifies the validity of our numerical method.
\begin{figure*}[htbp]
\centering
\subfigure{
\begin{minipage}[t]{0.3\textwidth}
\centering
\includegraphics[width=5cm]{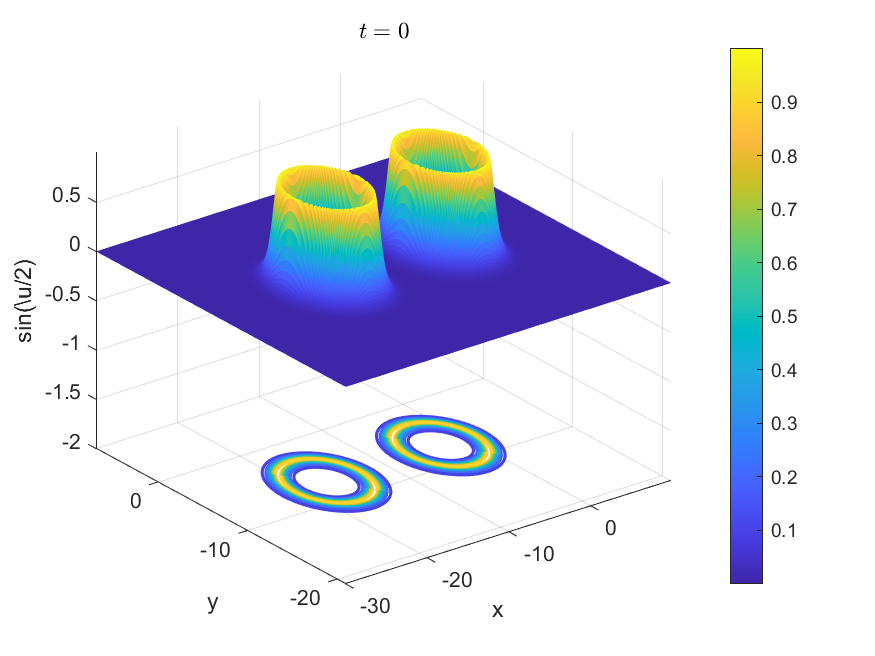}
\end{minipage}
}
\subfigure{
\begin{minipage}[t]{0.3\textwidth}
\centering
\includegraphics[width=5cm]{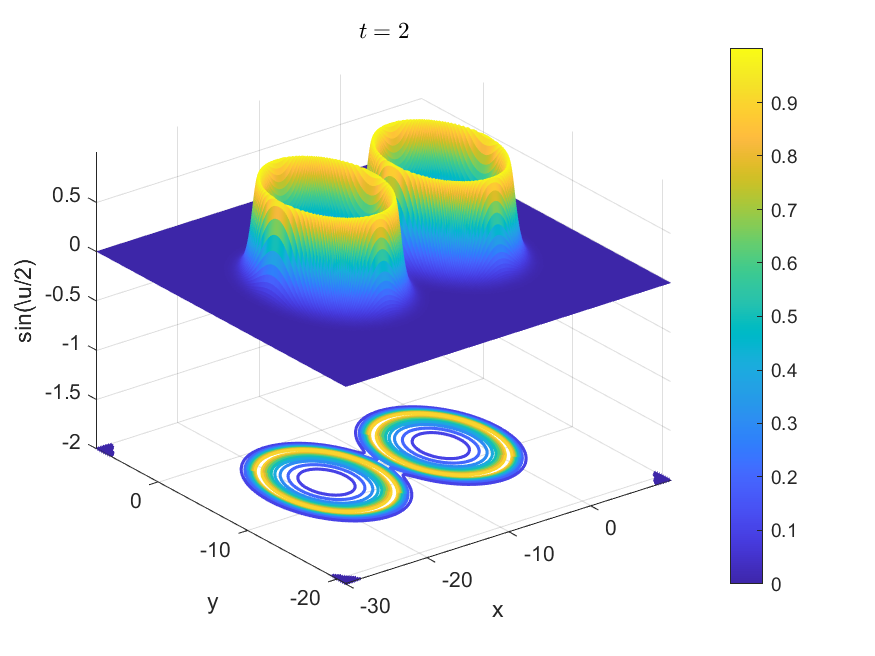}
\end{minipage}
}
\subfigure{
\begin{minipage}[t]{0.3\textwidth}
\centering
\includegraphics[width=5cm]{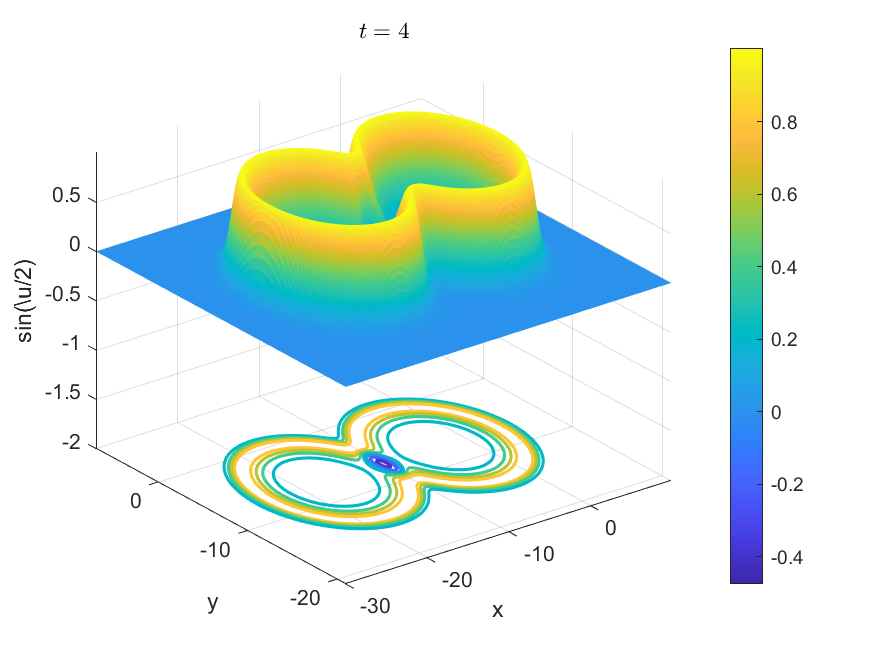}
\end{minipage}
}
\\
\subfigure{
\begin{minipage}[t]{0.3\textwidth}
\centering
\includegraphics[width=5cm]{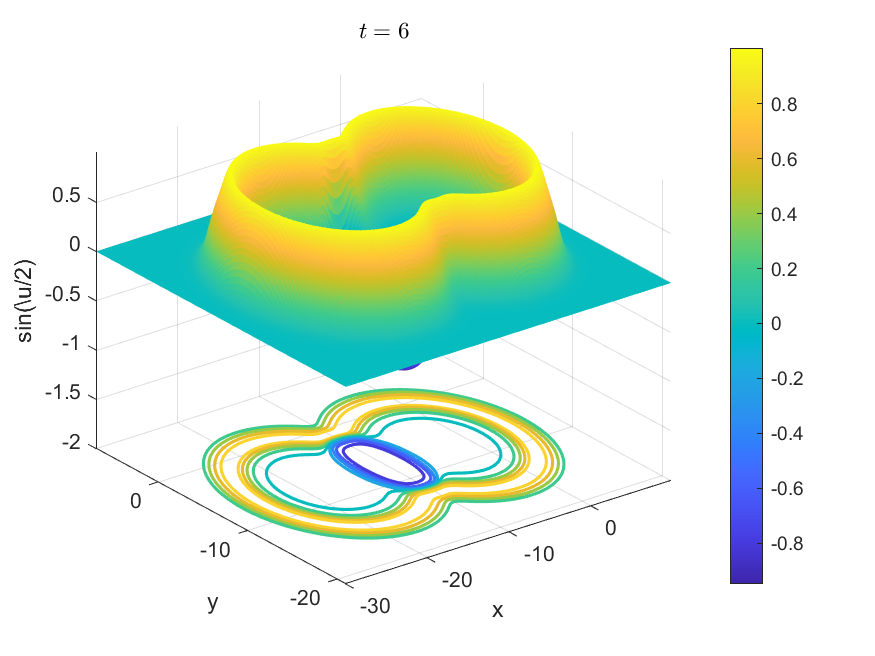}
\end{minipage}
}
\subfigure{
\begin{minipage}[t]{0.3\textwidth}
\centering
\includegraphics[width=5cm]{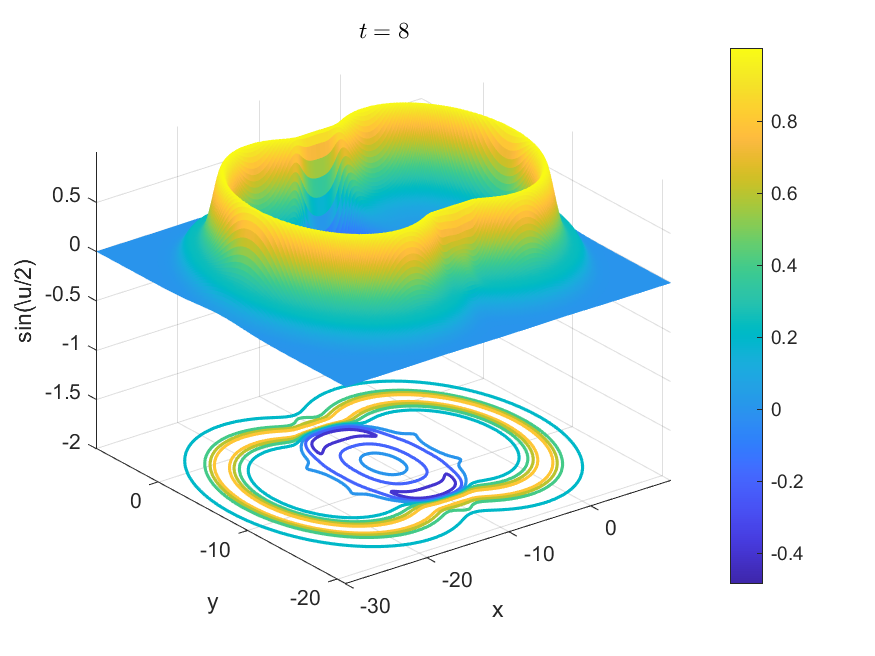}
\end{minipage}
}
\subfigure{
\begin{minipage}[t]{0.3\textwidth}
\centering
\includegraphics[width=5cm]{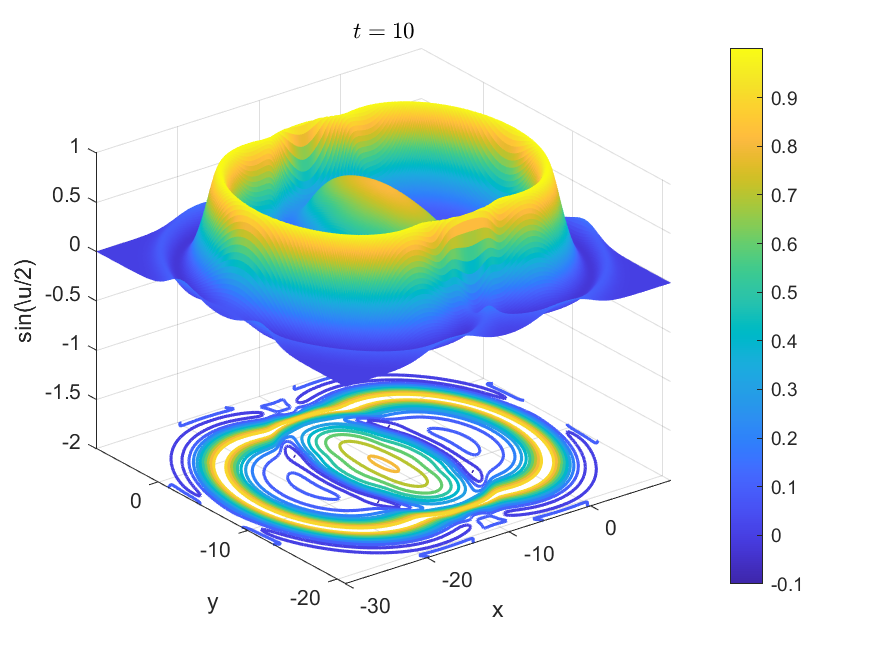}
\end{minipage}
}
\caption{ Surfaces and contour plots of the collisions of two circular ring soliton when $\alpha=2$. }
\label{afig4}
\end{figure*}

\begin{figure*}[htbp]
\centering
\subfigure{
\begin{minipage}[t]{0.3\textwidth}
\centering
\includegraphics[width=5cm]{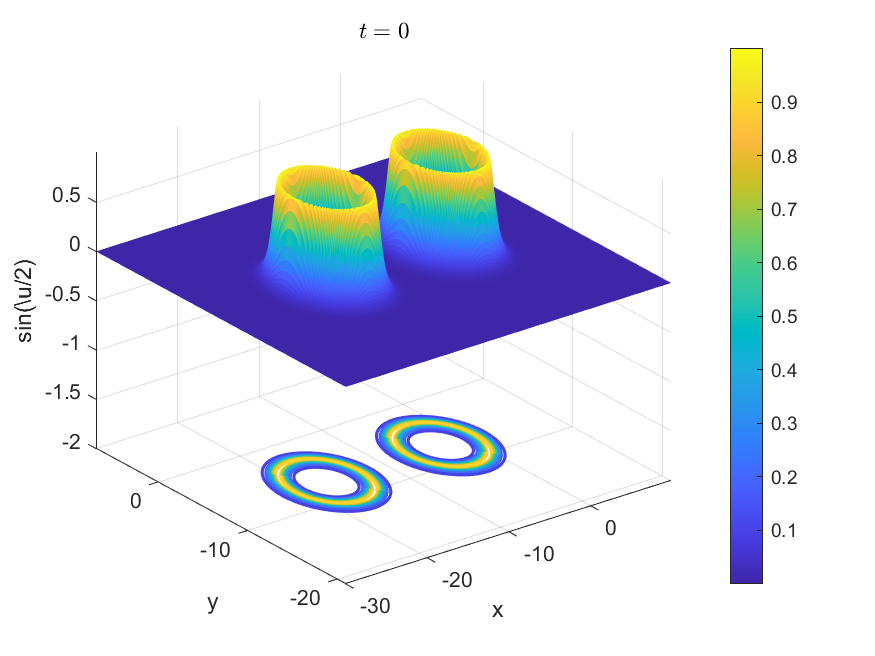}
\end{minipage}
}
\subfigure{
\begin{minipage}[t]{0.3\textwidth}
\centering
\includegraphics[width=5cm]{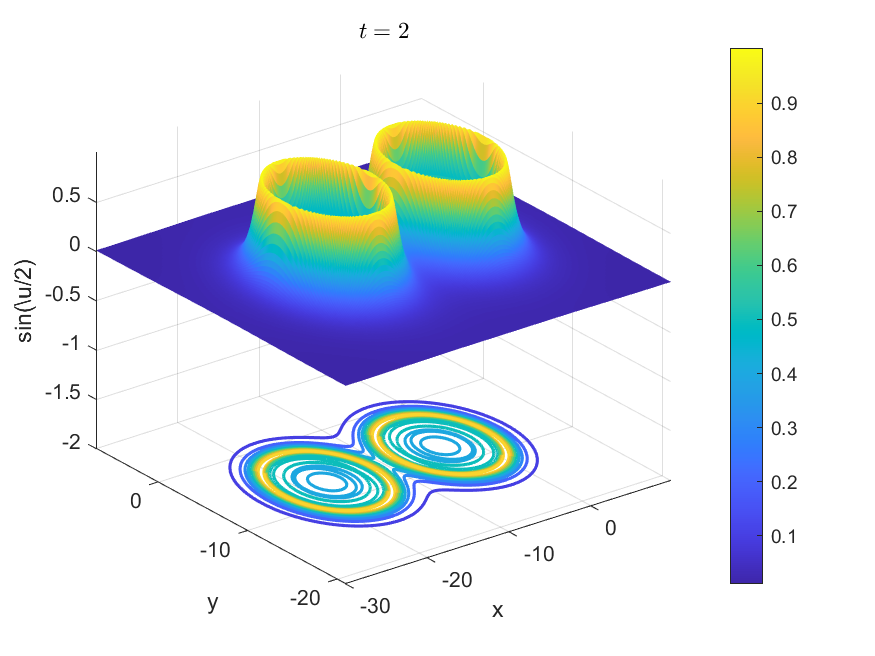}
\end{minipage}
}
\subfigure{
\begin{minipage}[t]{0.3\textwidth}
\centering
\includegraphics[width=5cm]{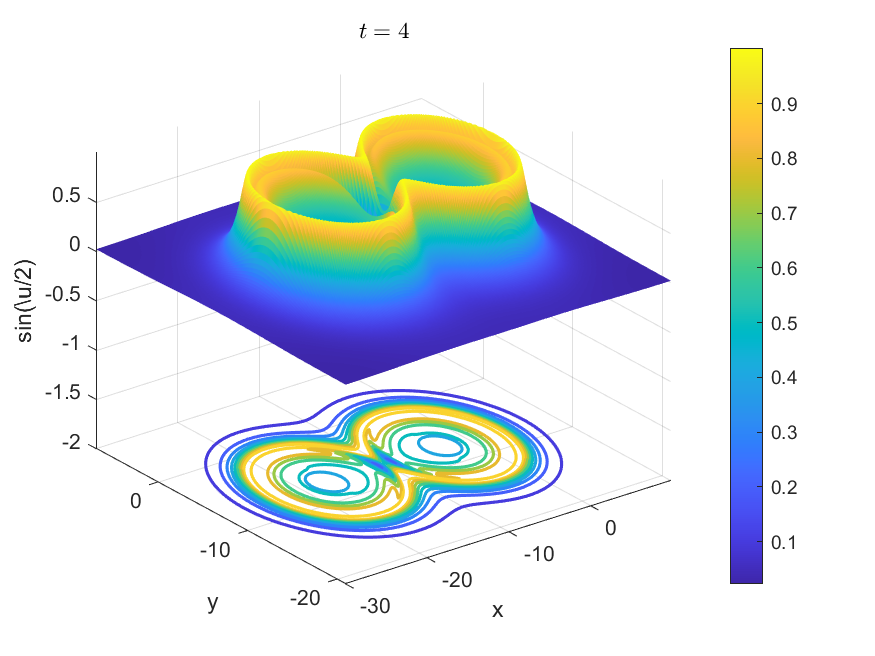}
\end{minipage}
}
\\
\subfigure{
\begin{minipage}[t]{0.3\textwidth}
\centering
\includegraphics[width=5cm]{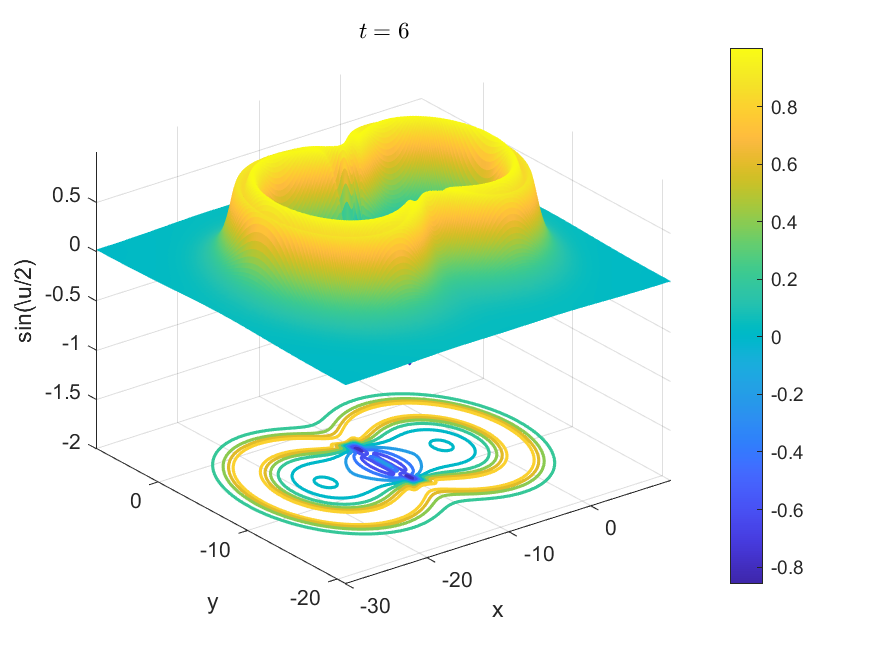}
\end{minipage}
}
\subfigure{
\begin{minipage}[t]{0.3\textwidth}
\centering
\includegraphics[width=5cm]{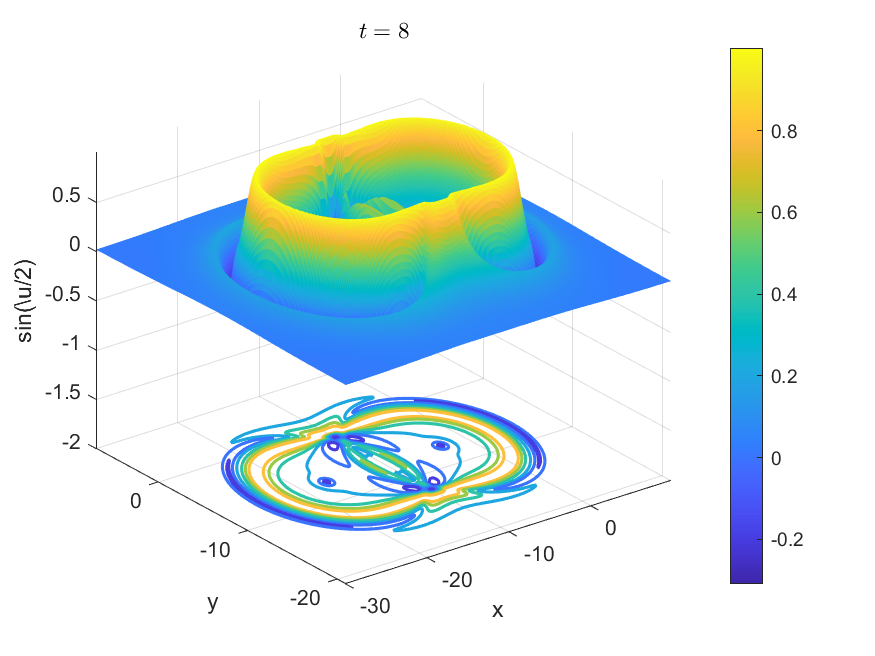}
\end{minipage}
}
\subfigure{
\begin{minipage}[t]{0.3\textwidth}
\centering
\includegraphics[width=5cm]{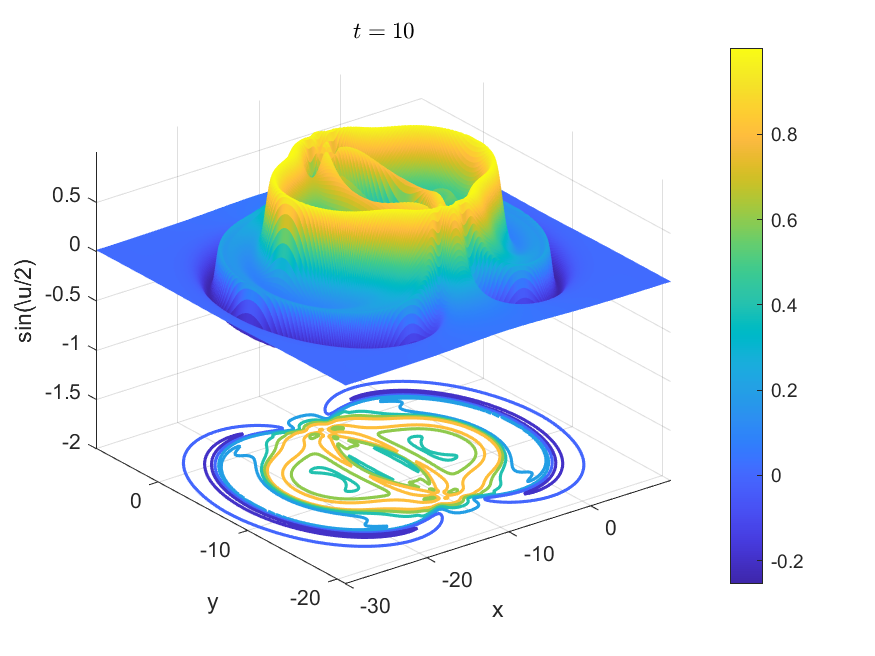}
\end{minipage}
}
\caption{ Surfaces and contour plots of the collisions of two circular ring soliton  when $\alpha=1.5$. }
\label{afig5}
\end{figure*}

\begin{figure*}[htbp]
\centering
\subfigure{
\begin{minipage}[t]{0.3\textwidth}
\centering
\includegraphics[width=5cm]{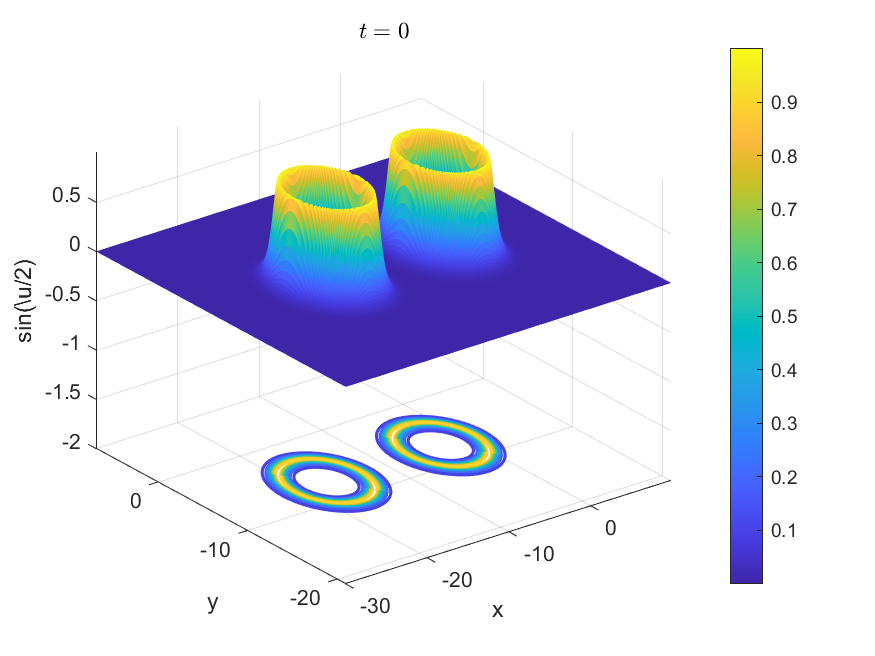}
\end{minipage}
}
\subfigure{
\begin{minipage}[t]{0.3\textwidth}
\centering
\includegraphics[width=5cm]{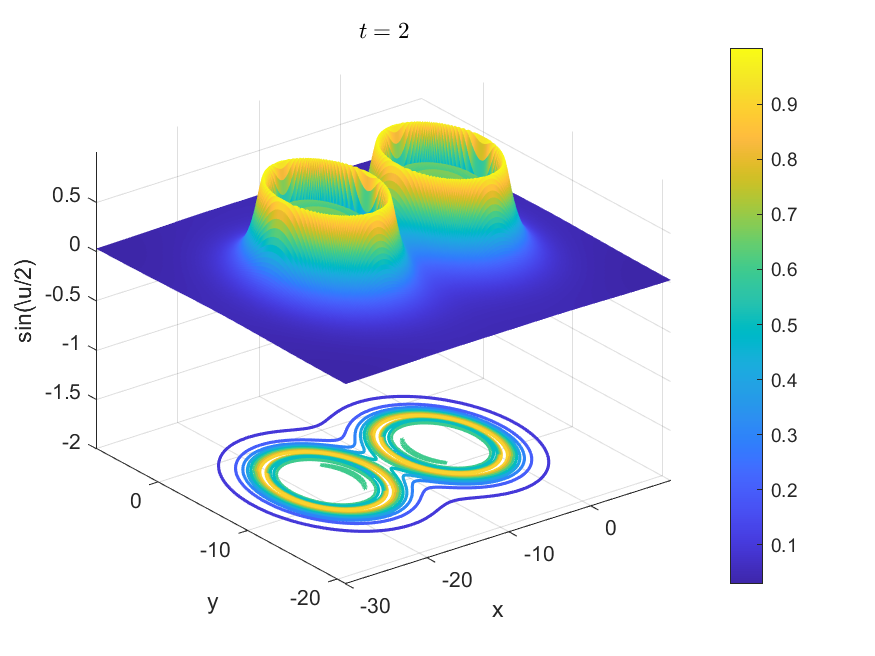}
\end{minipage}
}
\subfigure{
\begin{minipage}[t]{0.3\textwidth}
\centering
\includegraphics[width=5cm]{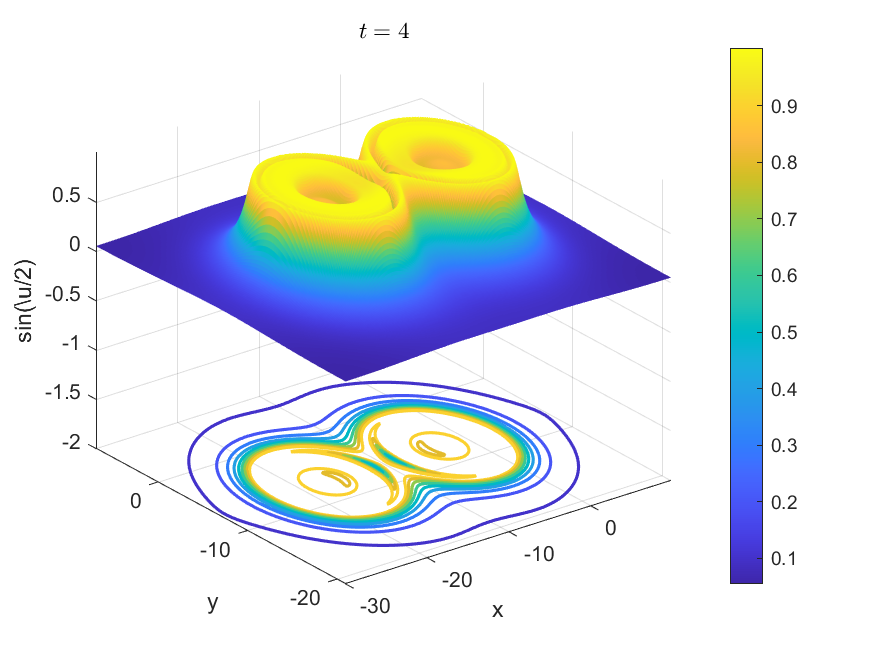}
\end{minipage}
}
\\
\subfigure{
\begin{minipage}[t]{0.3\textwidth}
\centering
\includegraphics[width=5cm]{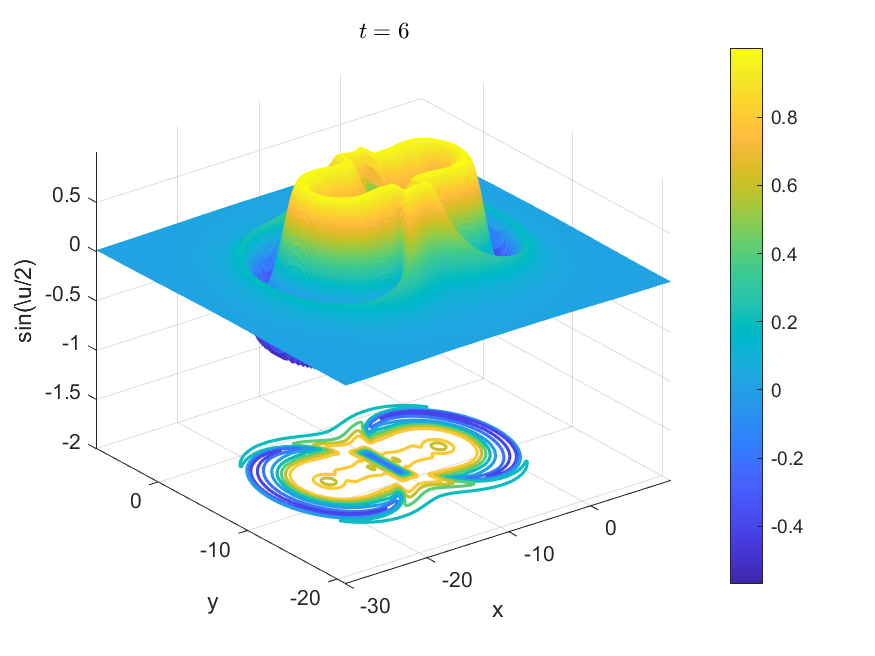}
\end{minipage}
}
\subfigure{
\begin{minipage}[t]{0.3\textwidth}
\centering
\includegraphics[width=5cm]{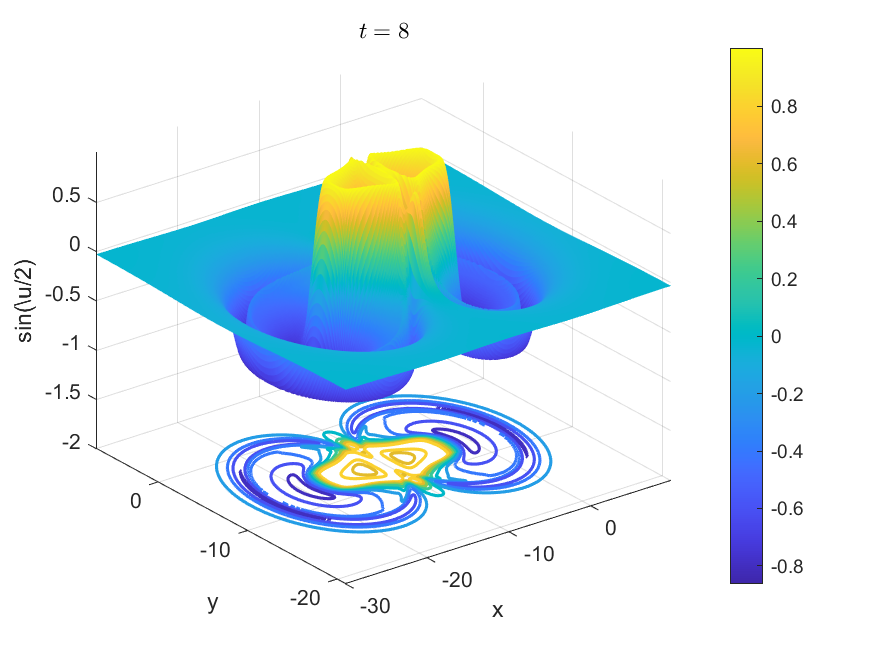}
\end{minipage}
}
\subfigure{
\begin{minipage}[t]{0.3\textwidth}
\centering
\includegraphics[width=5cm]{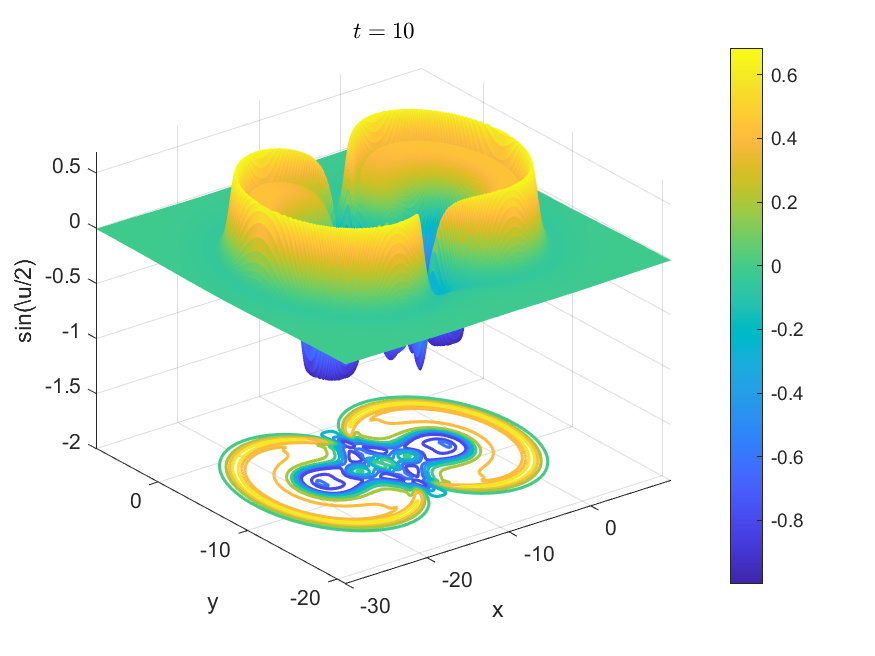}
\end{minipage}
}
\caption{ Surfaces and contour plots of the collisions of two circular ring soliton  when $\alpha=1.2$. }
\label{afig6}
\end{figure*}
%\subsection{Numerical results for 3D problem}
\subsubsection{Collisions of two circular solutions in 3D}
We consider the circular solitons in 3D with the initial conditions as
\begin{equation}
\begin{split}
u_0(x, y, z)&=4 \tan ^{-1}\left[\exp \left(\frac{4-\sqrt{(x+3)^2+(y+7)^2+(z+7)^2}}{0.436}\right)\right],\quad -30 \leq x \leq 10,\quad -21 \leq y, z \leq 7, \\
u_1(x, y, z)&=4.13 \operatorname{sech}\left(\frac{4-\sqrt{(x+3)^2+(y+7)^2+(z+7)^2}}{0.436}\right),\quad -30 \leq x \leq 10,\quad -21 \leq y, z  \leq 7.
\end{split}
\end{equation}
The step sizes $N=64, \tau=10^{-3}$ are taken to discrete the space and time domain. The extension over $x=-10$, $y=-7$ and $z=-7$  is included in the solution by the symmetry property of the problem.  The isosurfaces of two expanding circular ring solitons in 3D for different $\alpha$ at time $t=0, 2, 3, 4, 5, 6$ are exhibited in Figures \ref{afig7}-\ref{afig9}. It is clearly  shown the collision between two expanding circular ring solitons, where two smaller ring solitons emerge into a large ring soliton and then it again separates into two small circular solitons.  We also can find the shapes of the two solitons change rapidly as the value of the fractional order $\alpha$ decreases. Thus the entire process  for the  collision between two expanding circular ring solitons is consistent with the 2D case.
\begin{figure*}[htbp]
\centering
\subfigure{
\begin{minipage}[t]{0.3\textwidth}
\centering
\includegraphics[width=5cm]{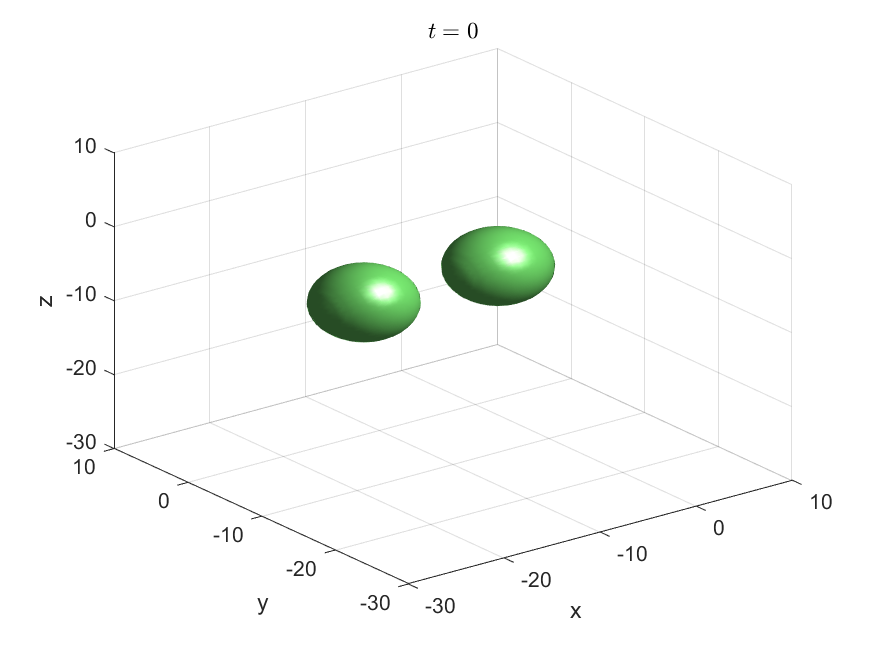}
\end{minipage}
}
\subfigure{
\begin{minipage}[t]{0.3\textwidth}
\centering
\includegraphics[width=5cm]{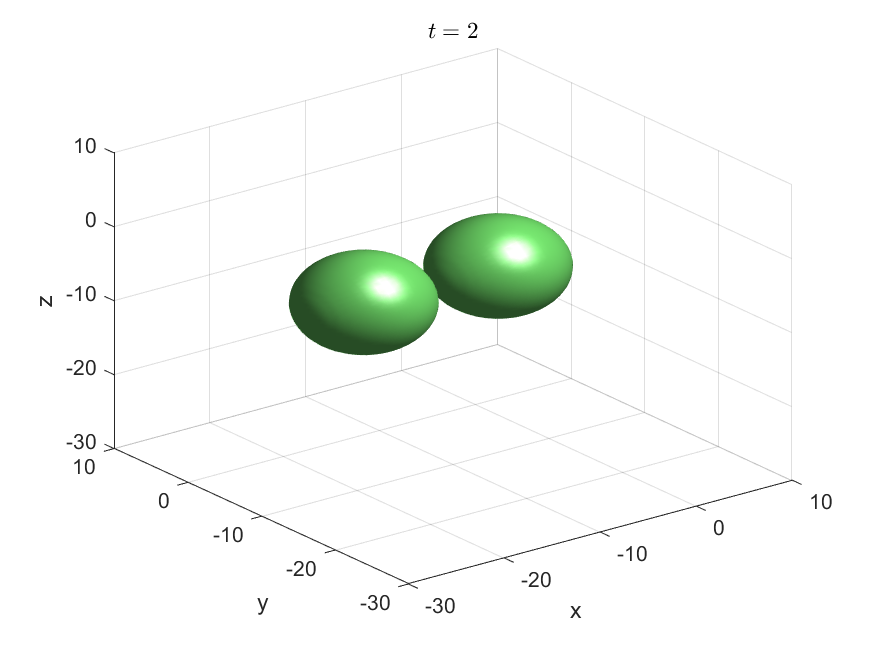}
\end{minipage}
}
\subfigure{
\begin{minipage}[t]{0.3\textwidth}
\centering
\includegraphics[width=5cm]{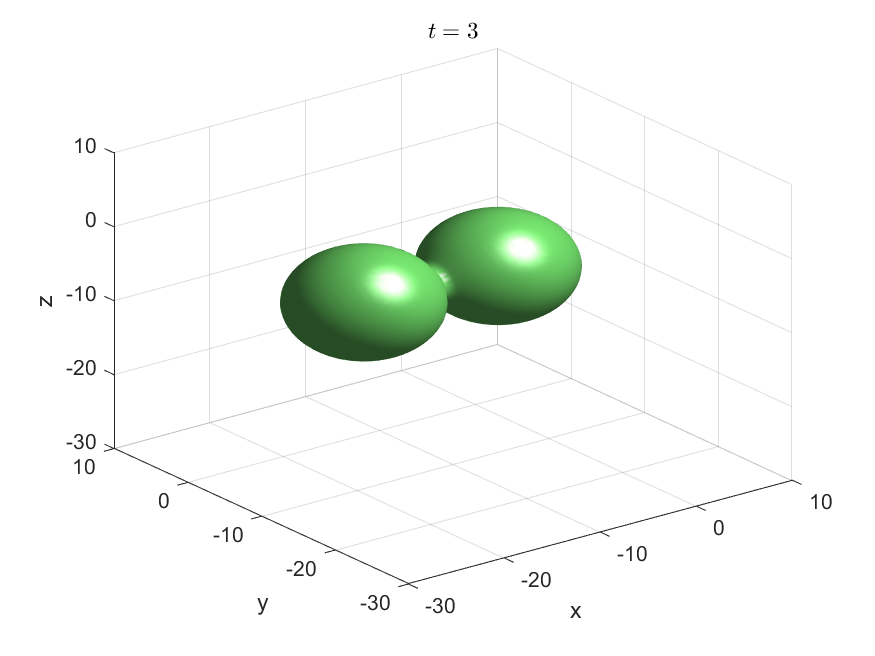}
\end{minipage}
}
\\
\subfigure{
\begin{minipage}[t]{0.3\textwidth}
\centering
\includegraphics[width=5cm]{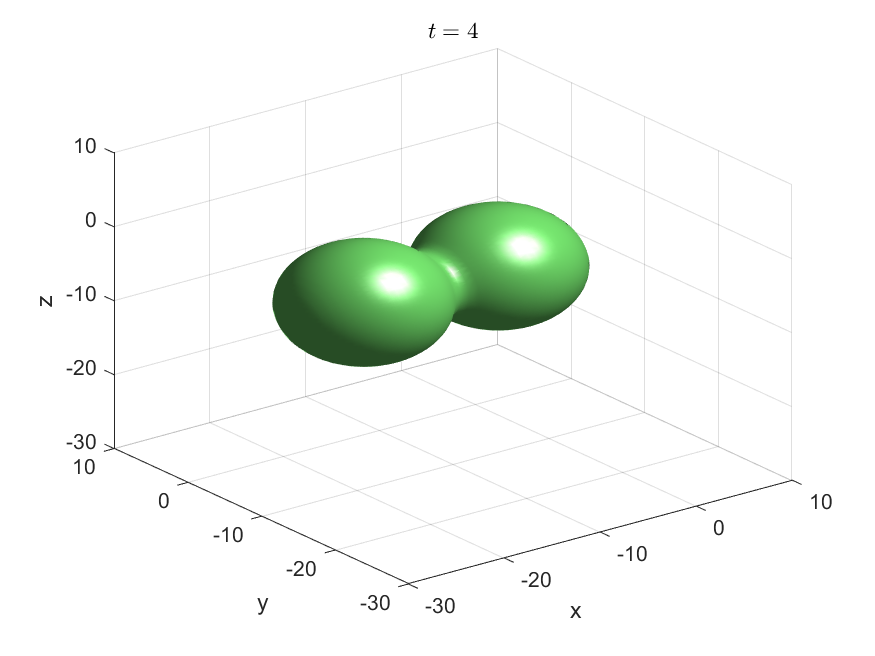}
\end{minipage}
}
\subfigure{
\begin{minipage}[t]{0.3\textwidth}
\centering
\includegraphics[width=5cm]{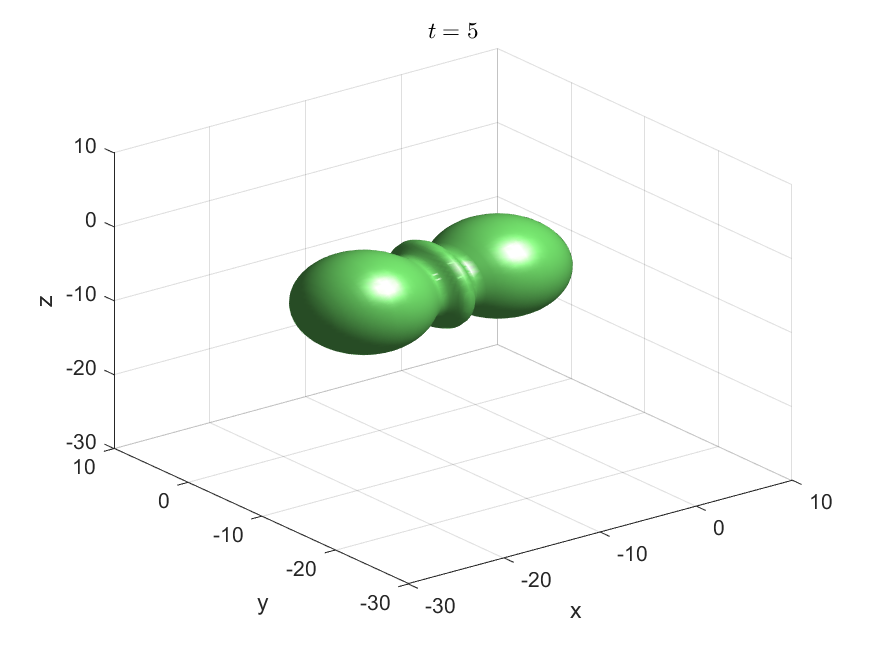}
\end{minipage}
}
\subfigure{
\begin{minipage}[t]{0.3\textwidth}
\centering
\includegraphics[width=5cm]{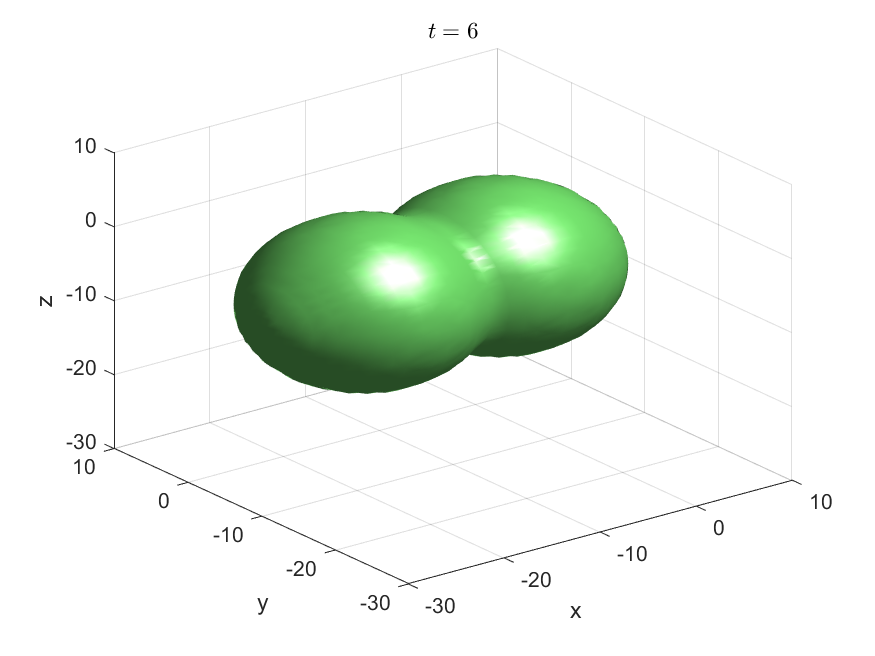}
\end{minipage}
}
\caption{Isosurfaces for the collisions of two circular solitons in 3D when $\alpha=2$. }
\label{afig7}
\end{figure*}
\begin{figure*}[htbp]
\centering
\subfigure{
\begin{minipage}[t]{0.3\textwidth}
\centering
\includegraphics[width=5cm]{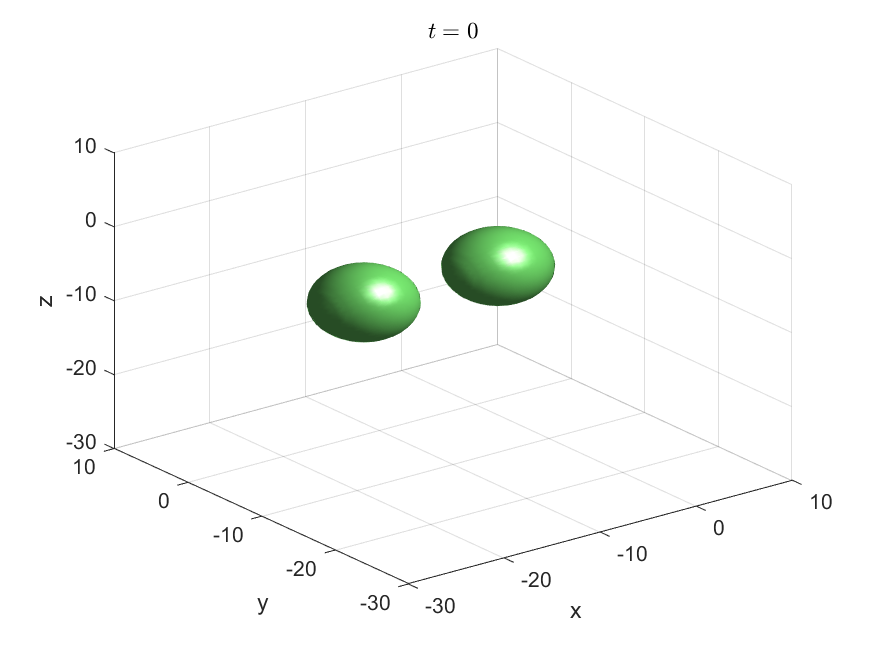}
\end{minipage}
}
\subfigure{
\begin{minipage}[t]{0.3\textwidth}
\centering
\includegraphics[width=5cm]{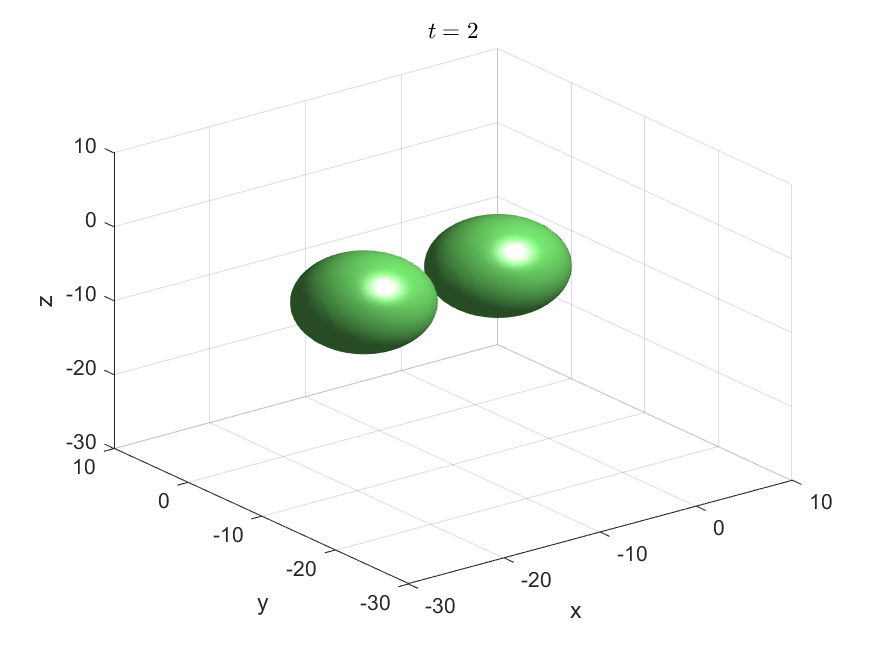}
\end{minipage}
}
\subfigure{
\begin{minipage}[t]{0.3\textwidth}
\centering
\includegraphics[width=5cm]{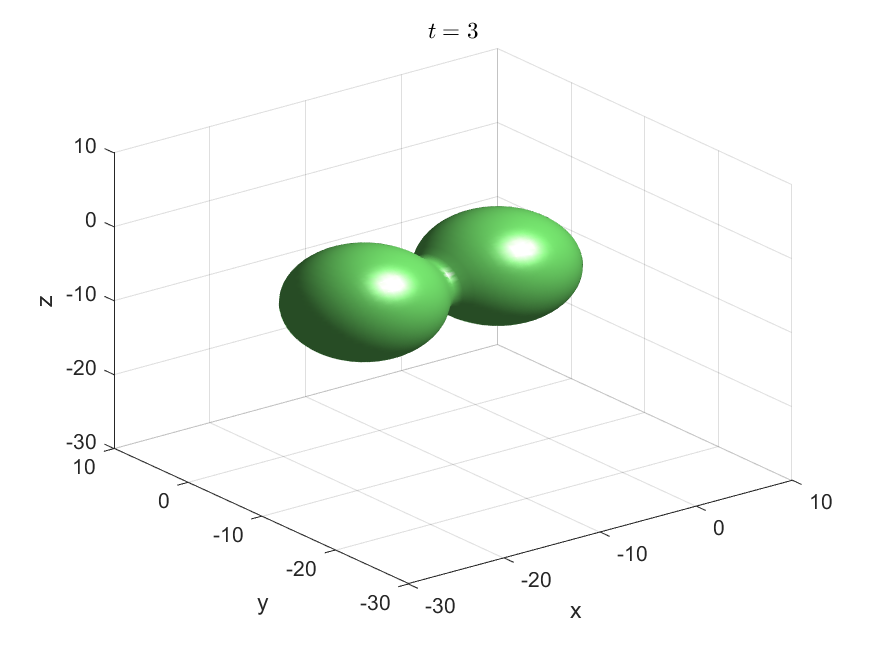}
\end{minipage}
}
\\
\subfigure{
\begin{minipage}[t]{0.3\textwidth}
\centering
\includegraphics[width=5cm]{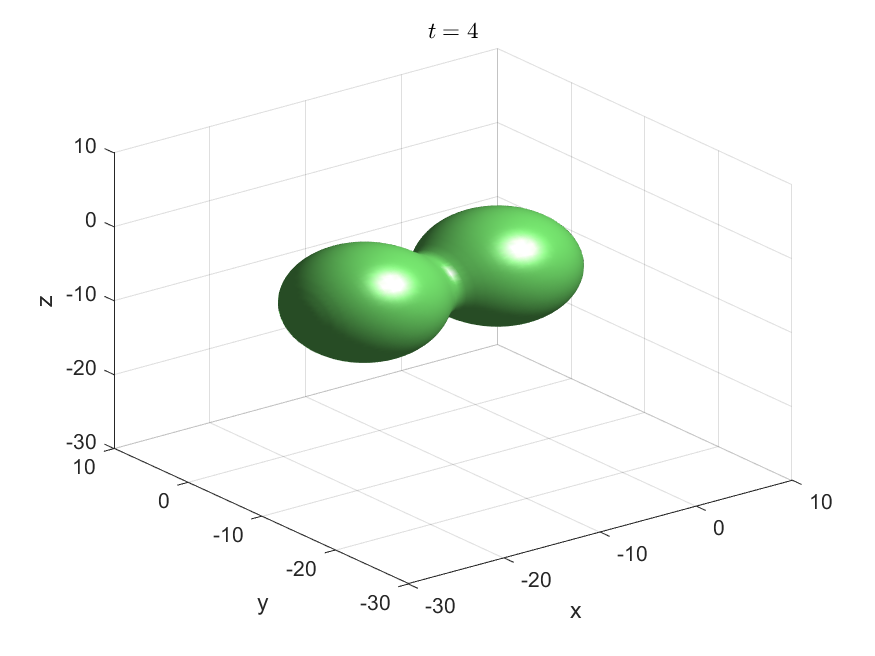}
\end{minipage}
}
\subfigure{
\begin{minipage}[t]{0.3\textwidth}
\centering
\includegraphics[width=5cm]{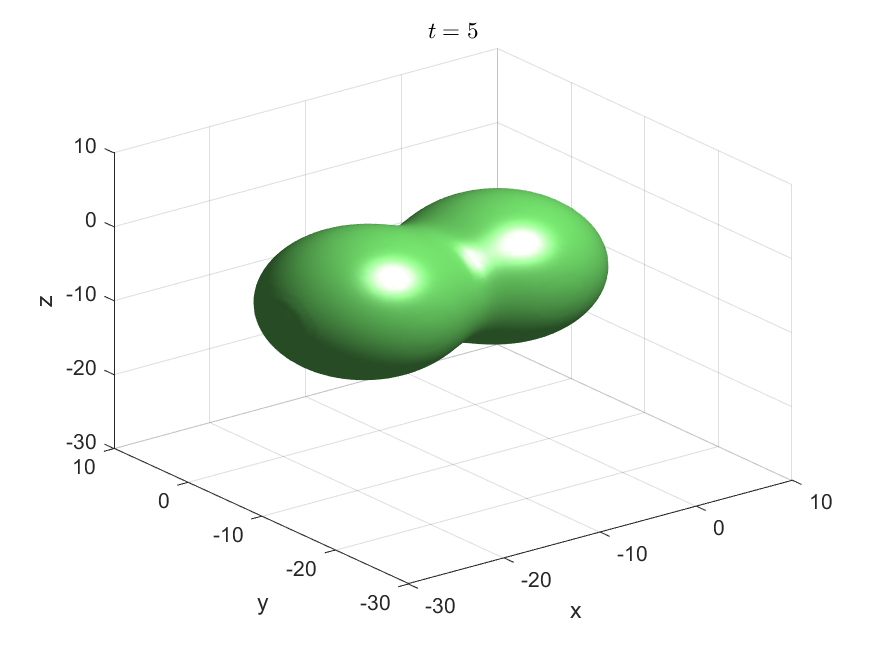}
\end{minipage}
}
\subfigure{
\begin{minipage}[t]{0.3\textwidth}
\centering
\includegraphics[width=5cm]{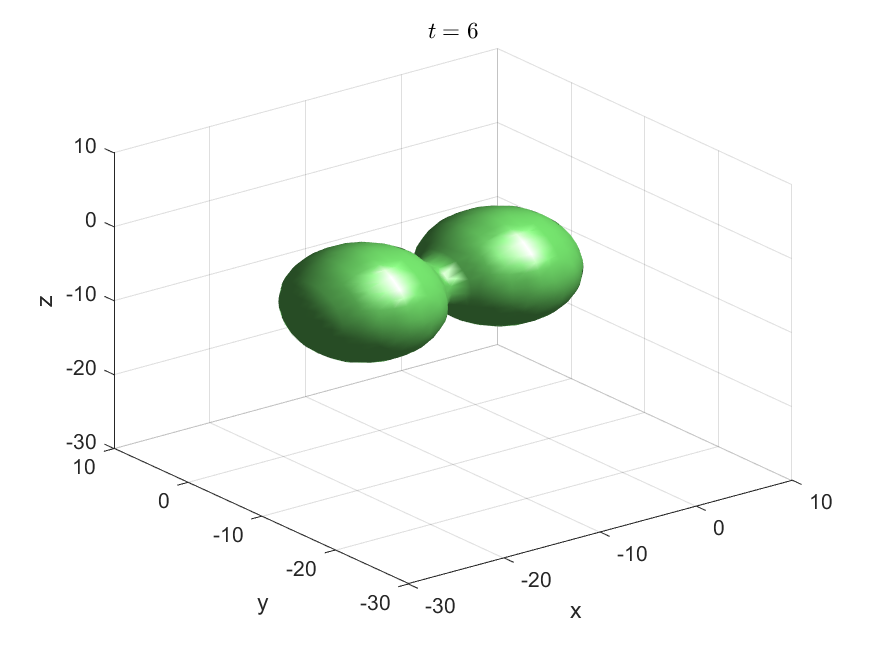}
\end{minipage}
}
\caption{ Isosurfaces for the collisions of two circular solitons in 3D when $\alpha=1.5$. }
\label{afig8}
\end{figure*}
\begin{figure*}[htbp]
\centering
\subfigure{
\begin{minipage}[t]{0.3\textwidth}
\centering
\includegraphics[width=5cm]{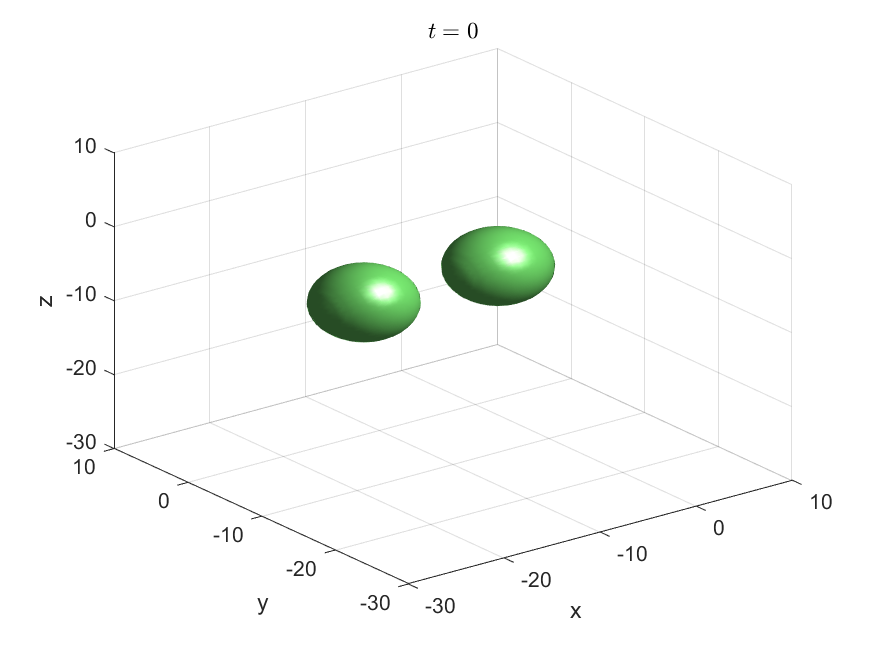}
\end{minipage}
}
\subfigure{
\begin{minipage}[t]{0.3\textwidth}
\centering
\includegraphics[width=5cm]{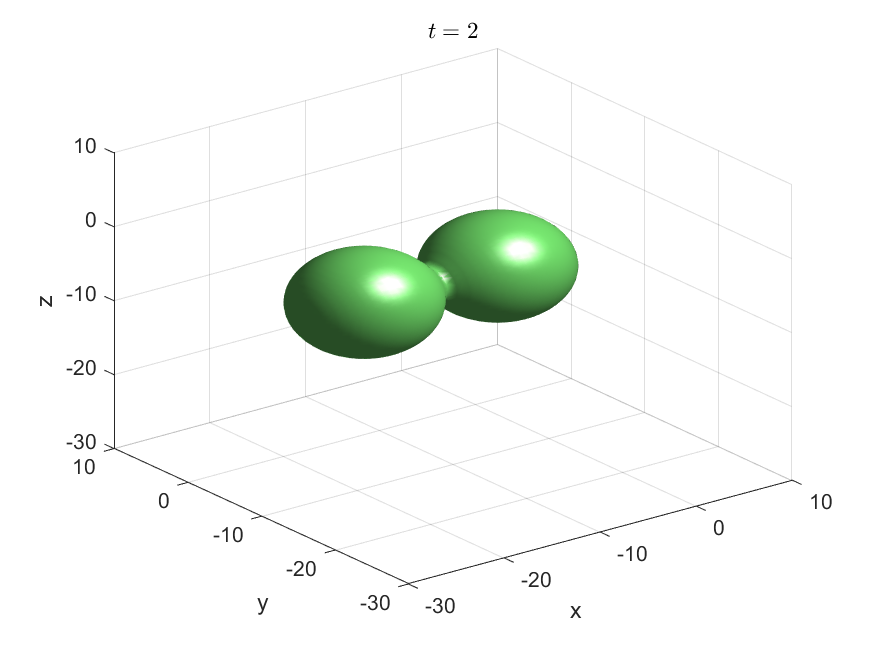}
\end{minipage}
}
\subfigure{
\begin{minipage}[t]{0.3\textwidth}
\centering
\includegraphics[width=5cm]{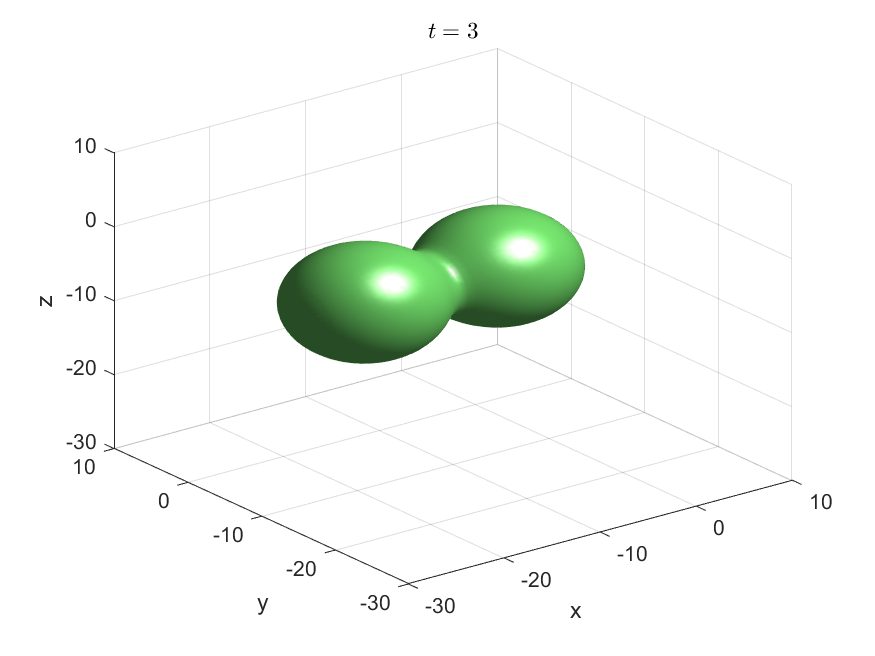}
\end{minipage}
}
\\
\subfigure{
\begin{minipage}[t]{0.3\textwidth}
\centering
\includegraphics[width=5cm]{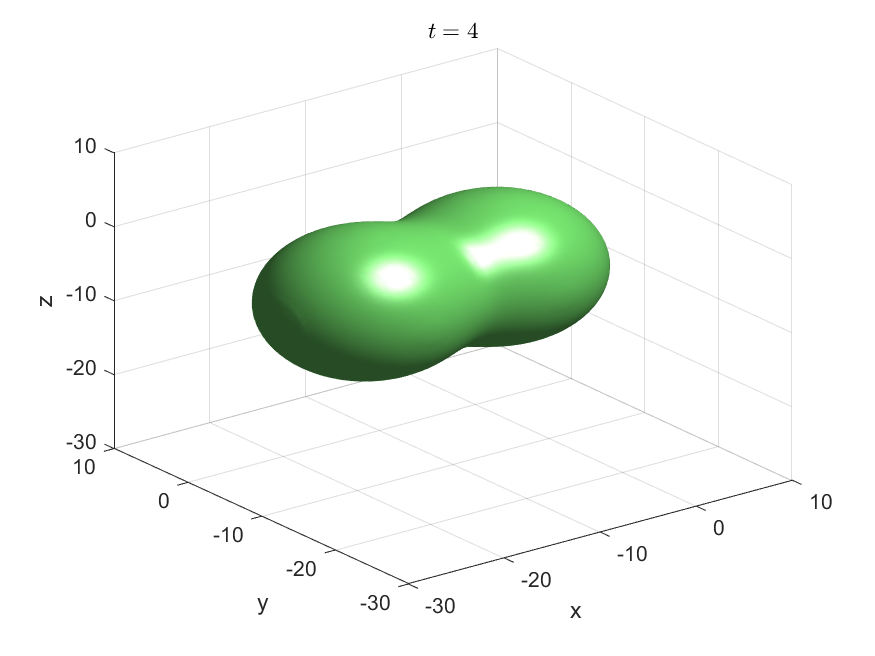}
\end{minipage}
}
\subfigure{
\begin{minipage}[t]{0.3\textwidth}
\centering
\includegraphics[width=5cm]{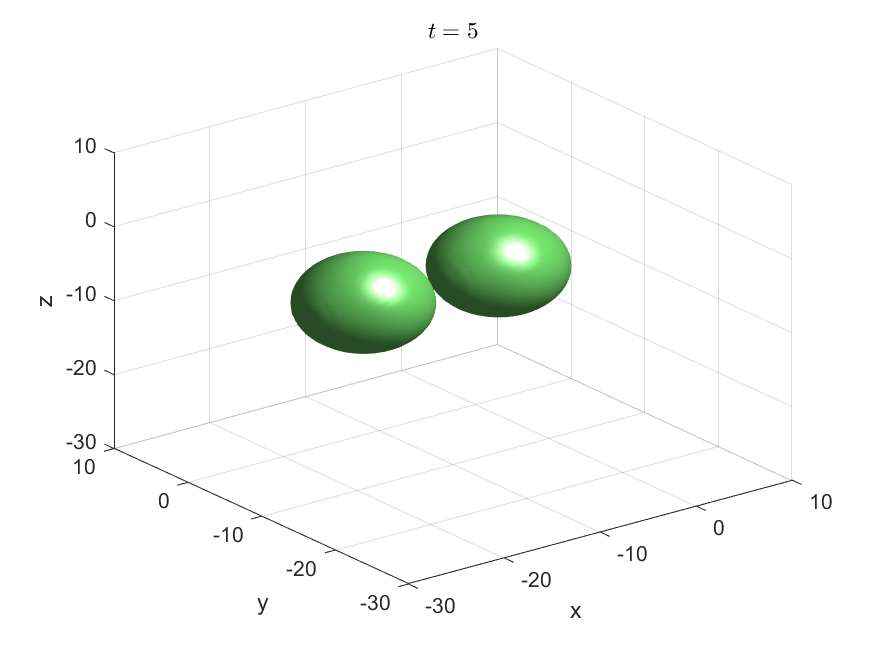}
\end{minipage}
}
\subfigure{
\begin{minipage}[t]{0.3\textwidth}
\centering
\includegraphics[width=5cm]{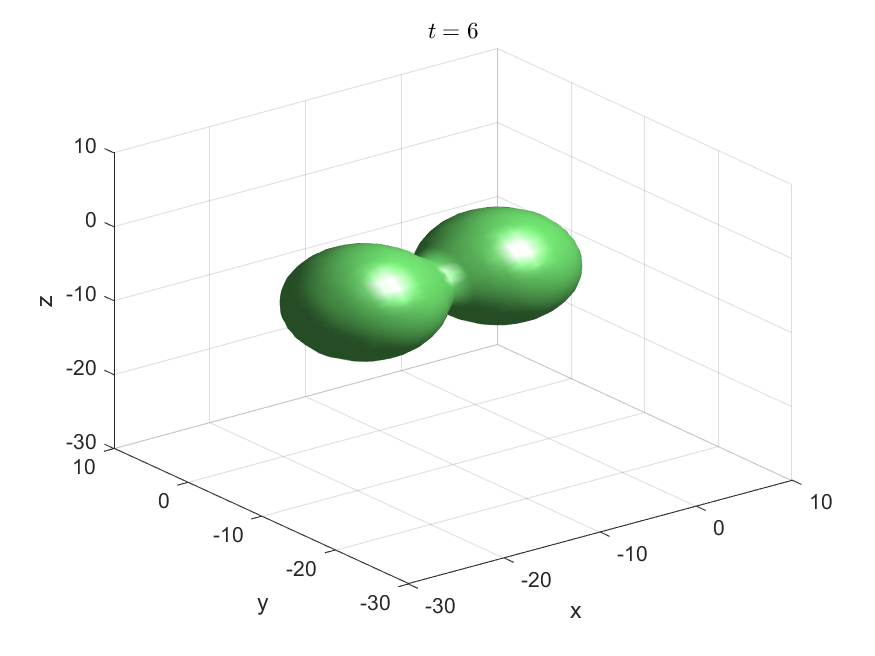}
\end{minipage}
}
\caption{ Isosurfaces for the collisions of two circular solitons in 3D when $\alpha=1.2$. }
\label{afig9}
\end{figure*}

\subsubsection{Collisions of four circular solutions in 3D}
We consider the circular  solitons in 3D with the  initial conditions as
\begin{equation}
\begin{split}
u_0(x, y, z)&=4 \tan ^{-1}\left[\exp \left(\frac{4-\sqrt{(x+3)^2+y^2+(z+3)^2}}{0.436}\right)\right],\quad -30 \leq x, y, z \leq 10, \\
u_1(x, y, z)&=4.13 \operatorname{sech}\left(\frac{4-\sqrt{(x+3)^2+y^2+(z+3)^2}}{0.436}\right),\quad -30 \leq x, y, z\leq 10.
\end{split}
\end{equation}
We choose $ N=64, \tau=10^{-3}$ to discrete the space and time domain. The extension over $x=-10$, $y=0$ and $z=-10$ is included in the solution  based on  the symmetry of the problem. The  isosurfaces of four expanding circular ring solitons in 3D  for different $\alpha$ at time $t=0, 1, 2, 3, 4, 4.5$ are presented in Figures \ref{afig10}-\ref{afig12}. We can see that these four circular solitons are initially independent of each other, gradually colliding and merging, and then separating into four new solitons. Moreover,  as the fractional order $\alpha$ decreases, the shapes of circular solitons change more rapidly and the entire collision period becomes shorter.

\begin{figure*}[htbp]
\centering
\subfigure{
\begin{minipage}[t]{0.3\textwidth}
\centering
\includegraphics[width=5cm]{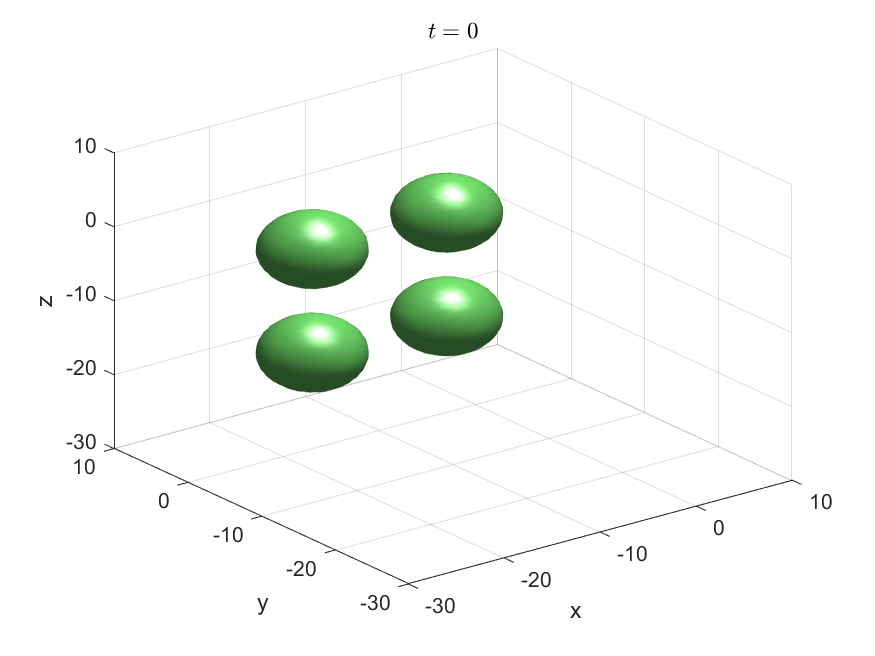}
\end{minipage}
}
\subfigure{
\begin{minipage}[t]{0.3\textwidth}
\centering
\includegraphics[width=5cm]{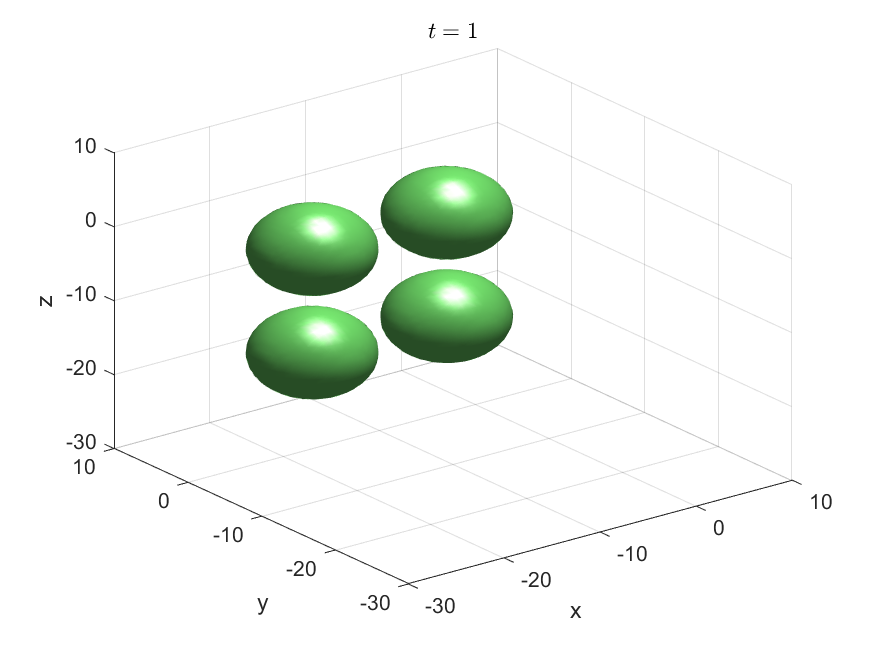}
\end{minipage}
}
\subfigure{
\begin{minipage}[t]{0.3\textwidth}
\centering
\includegraphics[width=5cm]{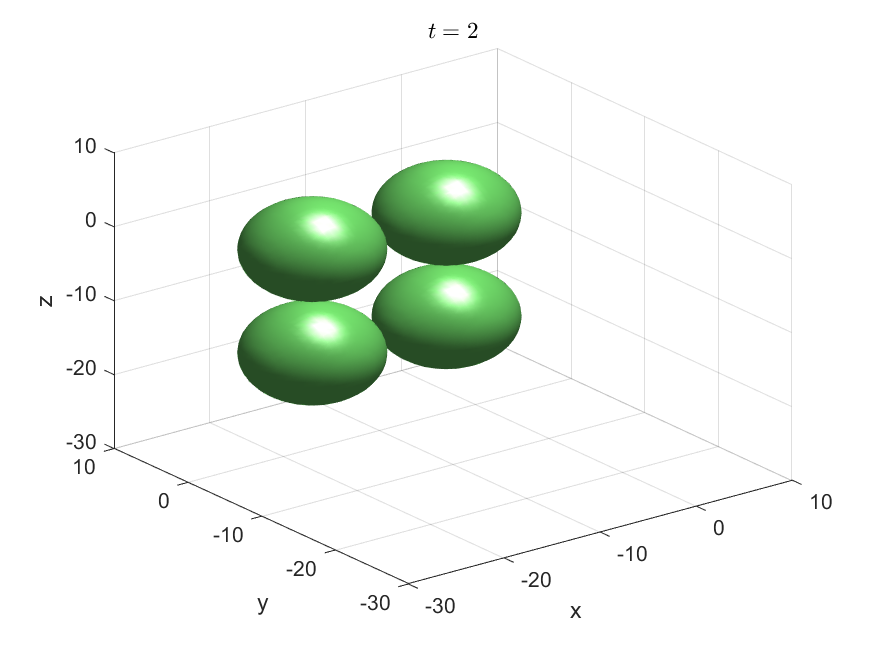}
\end{minipage}
}
\\
\subfigure{
\begin{minipage}[t]{0.3\textwidth}
\centering
\includegraphics[width=5cm]{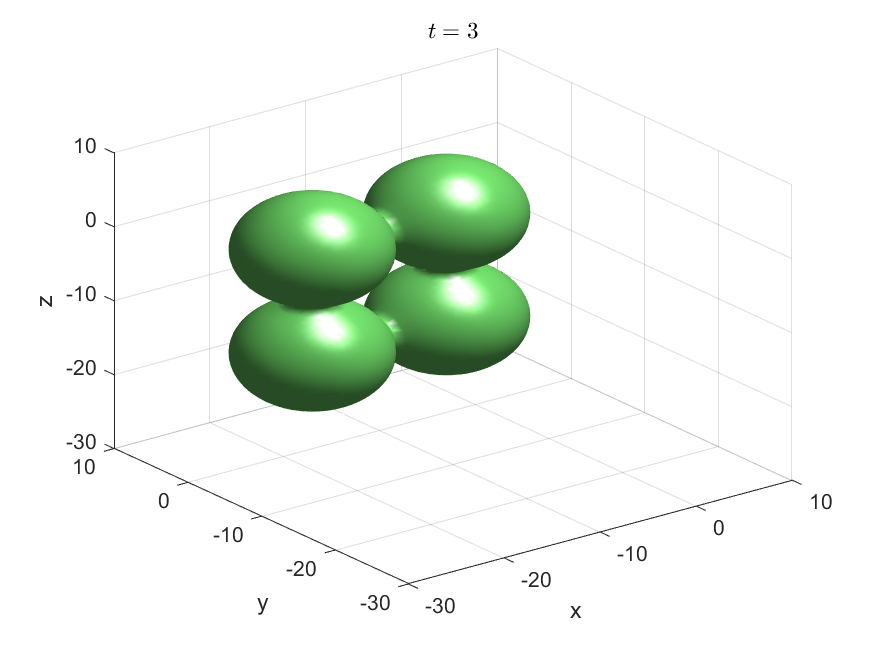}
\end{minipage}
}
\subfigure{
\begin{minipage}[t]{0.3\textwidth}
\centering
\includegraphics[width=5cm]{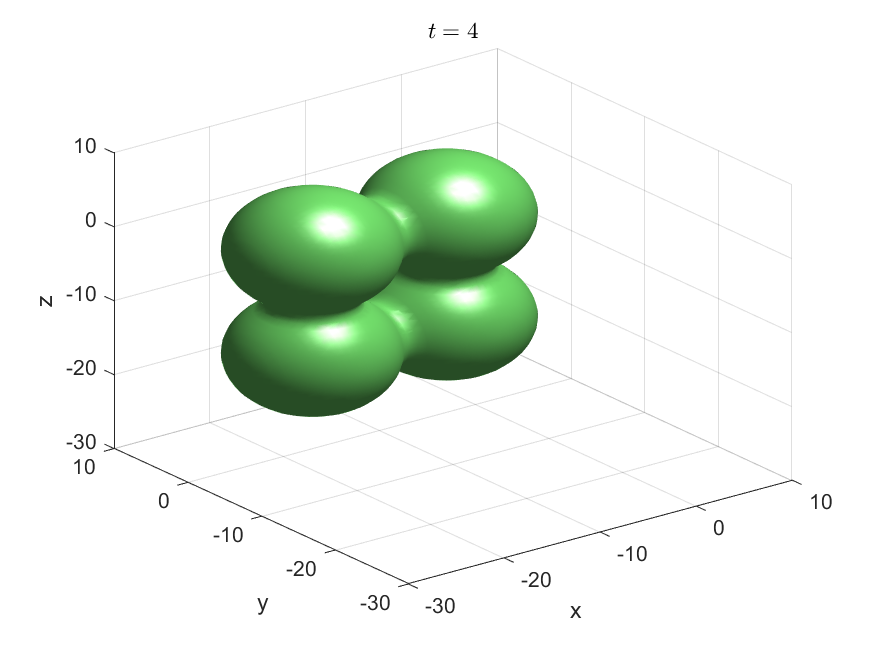}
\end{minipage}
}
\subfigure{
\begin{minipage}[t]{0.3\textwidth}
\centering
\includegraphics[width=5cm]{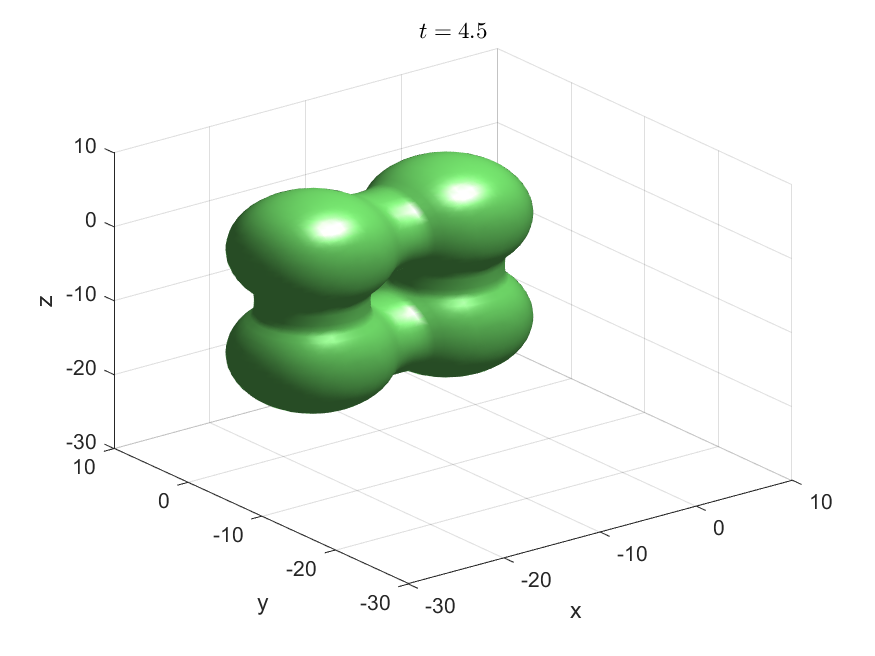}
\end{minipage}
}
\caption{ Isosurfaces for the collisions of four circular solitons in 3D when $\alpha=2$. }
\label{afig10}

\end{figure*}
\begin{figure*}[htbp]
\centering
\subfigure{
\begin{minipage}[t]{0.3\textwidth}
\centering
\includegraphics[width=5cm]{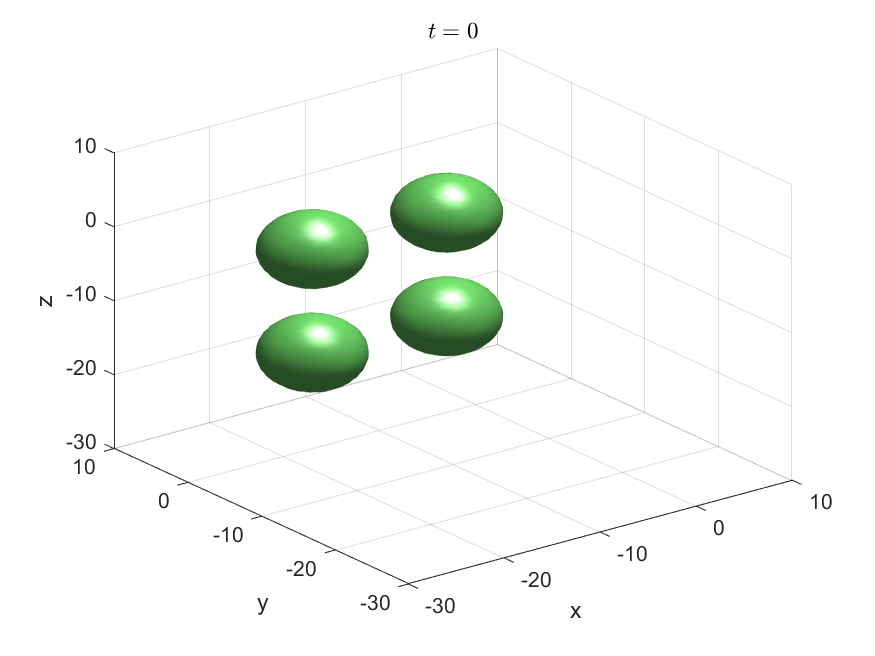}
\end{minipage}
}
\subfigure{
\begin{minipage}[t]{0.3\textwidth}
\centering
\includegraphics[width=5cm]{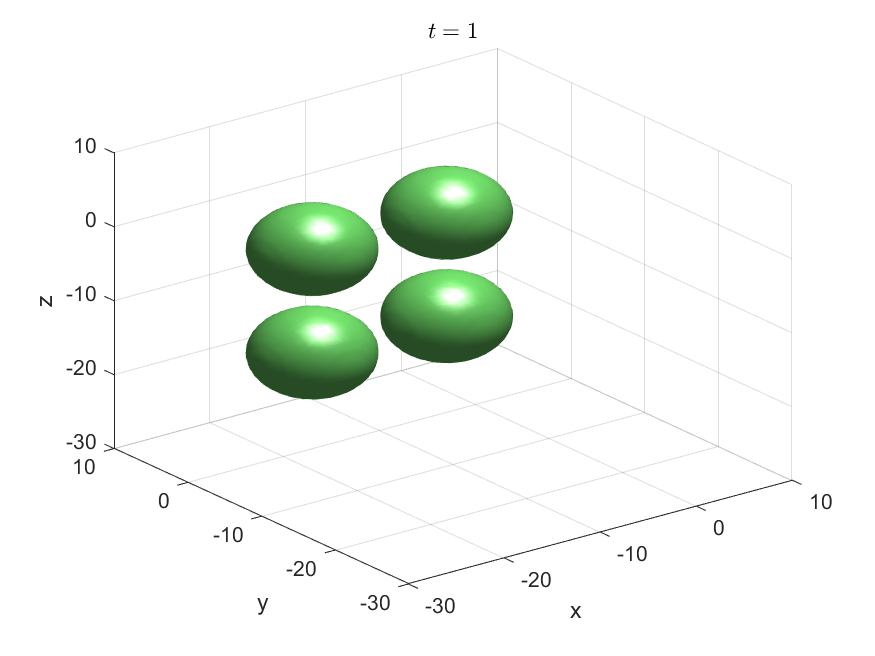}
\end{minipage}
}
\subfigure{
\begin{minipage}[t]{0.3\textwidth}
\centering
\includegraphics[width=5cm]{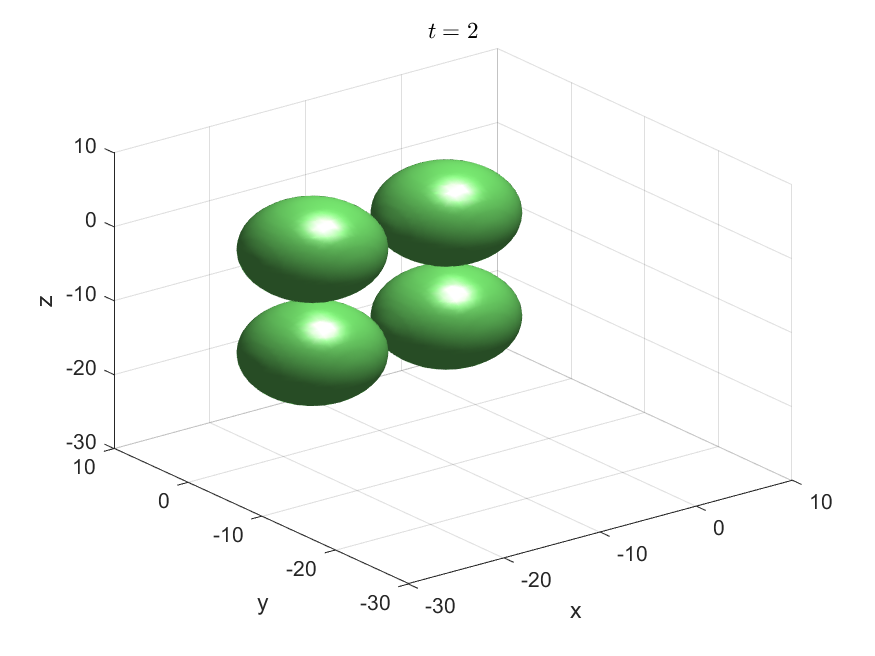}
\end{minipage}
}
\\
\subfigure{
\begin{minipage}[t]{0.3\textwidth}
\centering
\includegraphics[width=5cm]{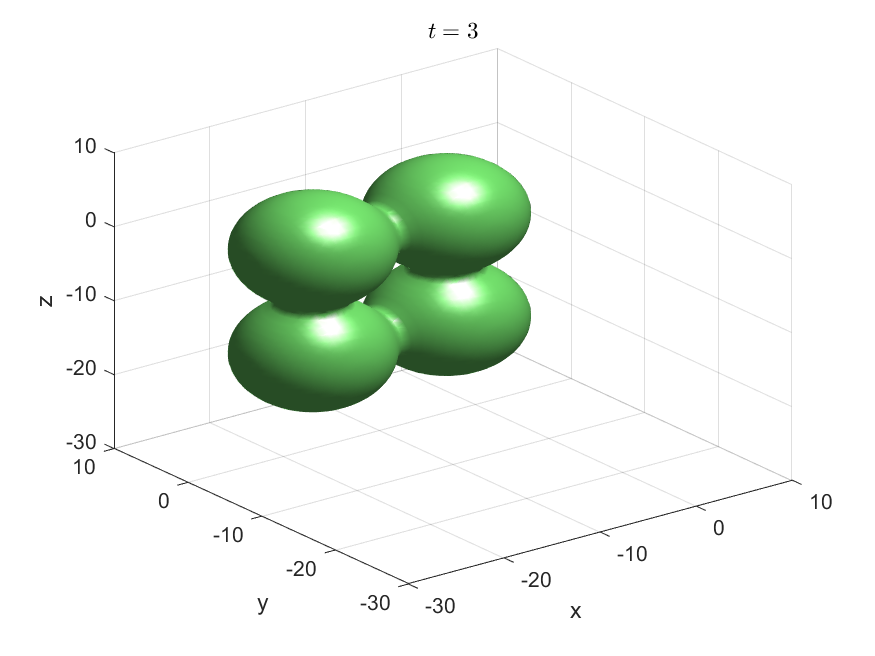}
\end{minipage}
}
\subfigure{
\begin{minipage}[t]{0.3\textwidth}
\centering
\includegraphics[width=5cm]{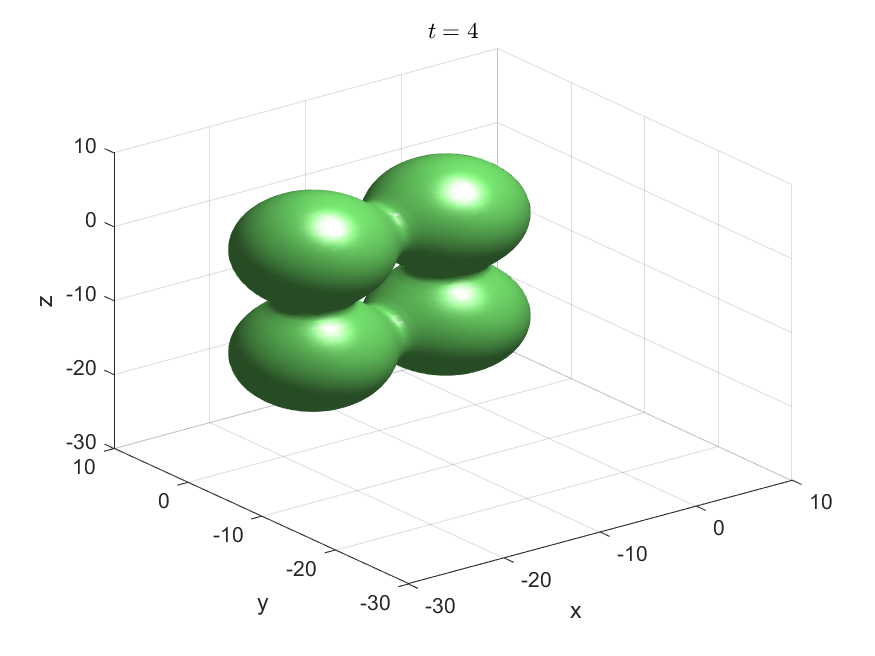}
\end{minipage}
}
\subfigure{
\begin{minipage}[t]{0.3\textwidth}
\centering
\includegraphics[width=5cm]{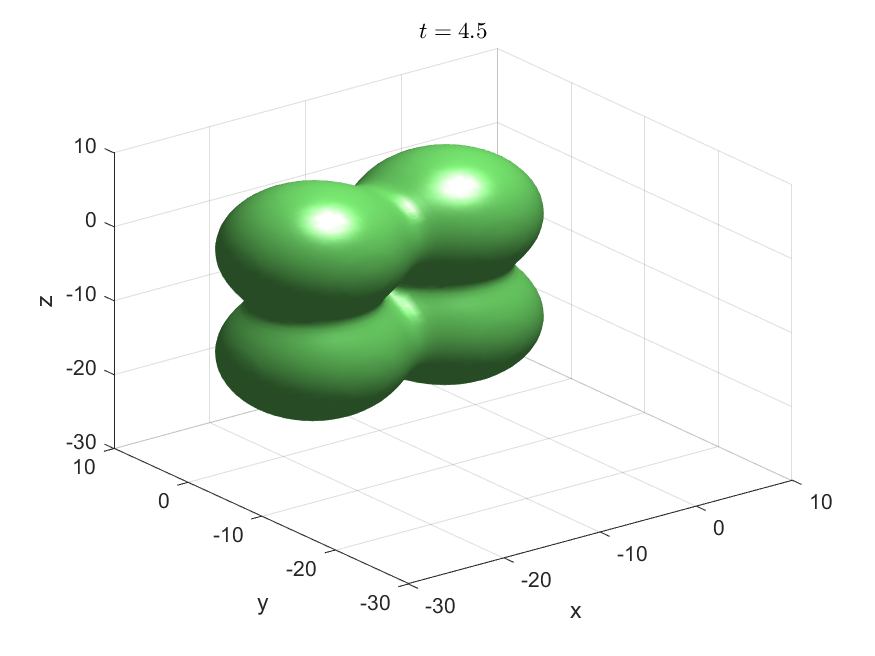}
\end{minipage}
}
\caption{ Isosurfaces for the collisions of four circular solitons in 3D when $\alpha=1.5$. }
\label{afig11}
\end{figure*}

\begin{figure*}[htbp]
\centering
\subfigure{
\begin{minipage}[t]{0.3\textwidth}
\centering
\includegraphics[width=5cm]{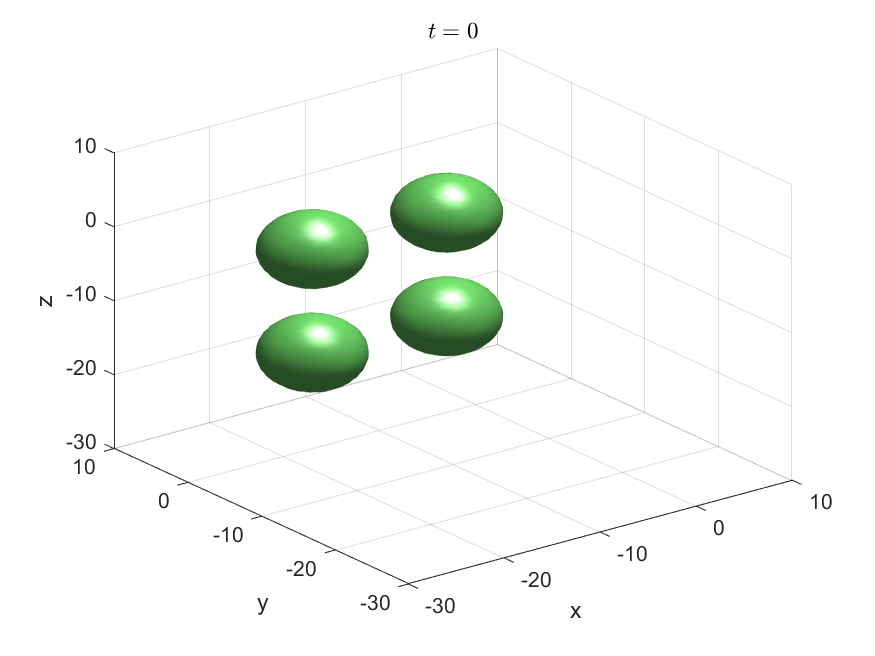}
\end{minipage}
}
\subfigure{
\begin{minipage}[t]{0.3\textwidth}
\centering
\includegraphics[width=5cm]{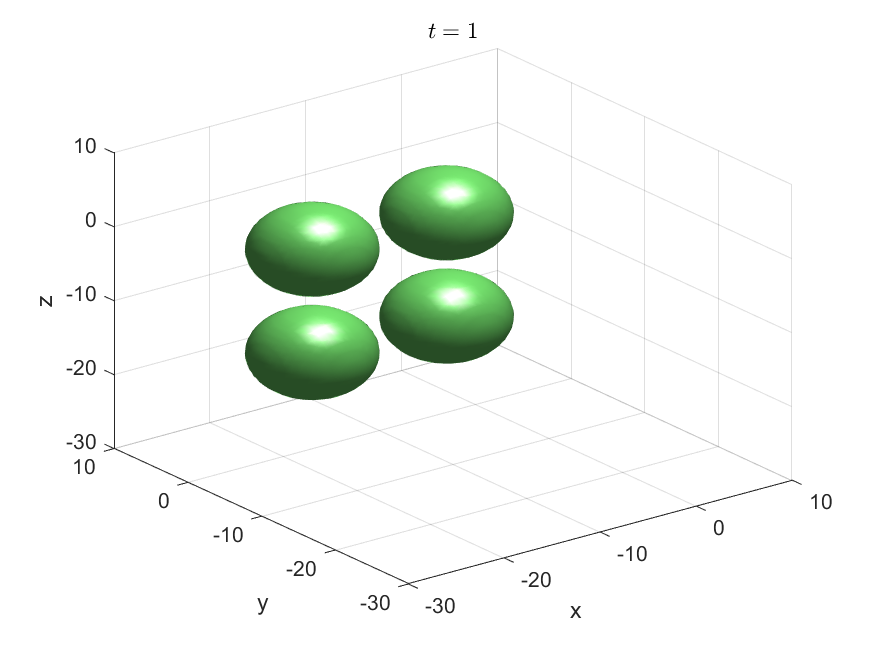}
\end{minipage}
}
\subfigure{
\begin{minipage}[t]{0.3\textwidth}
\centering
\includegraphics[width=5cm]{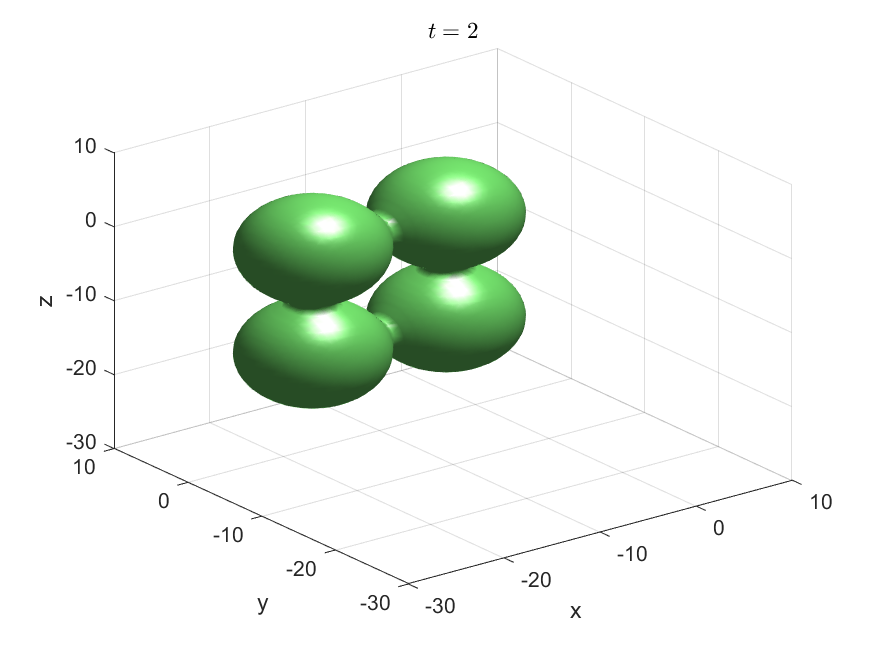}
\end{minipage}
}
\\
\subfigure{
\begin{minipage}[t]{0.3\textwidth}
\centering
\includegraphics[width=5cm]{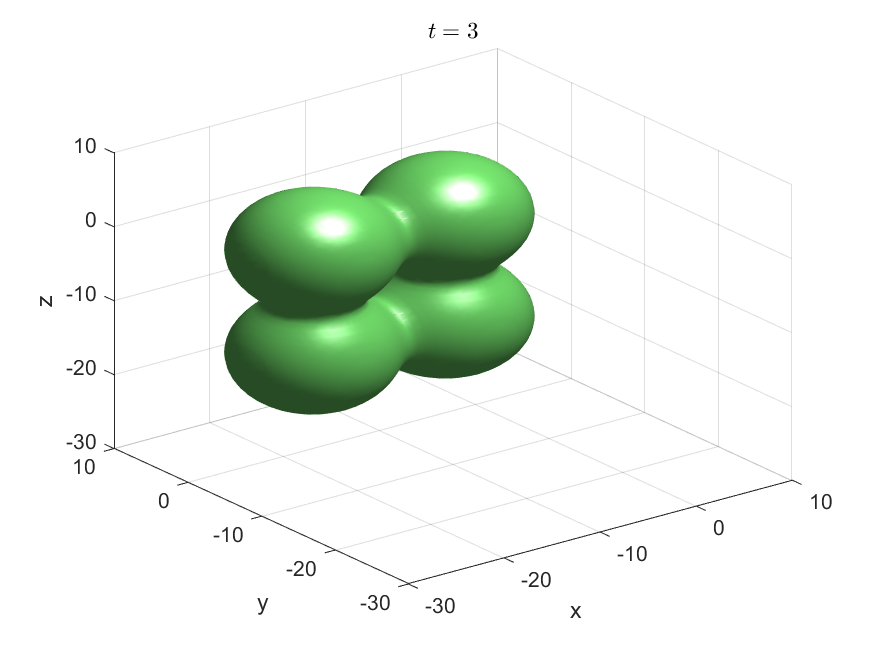}
\end{minipage}
}
\subfigure{
\begin{minipage}[t]{0.3\textwidth}
\centering
\includegraphics[width=5cm]{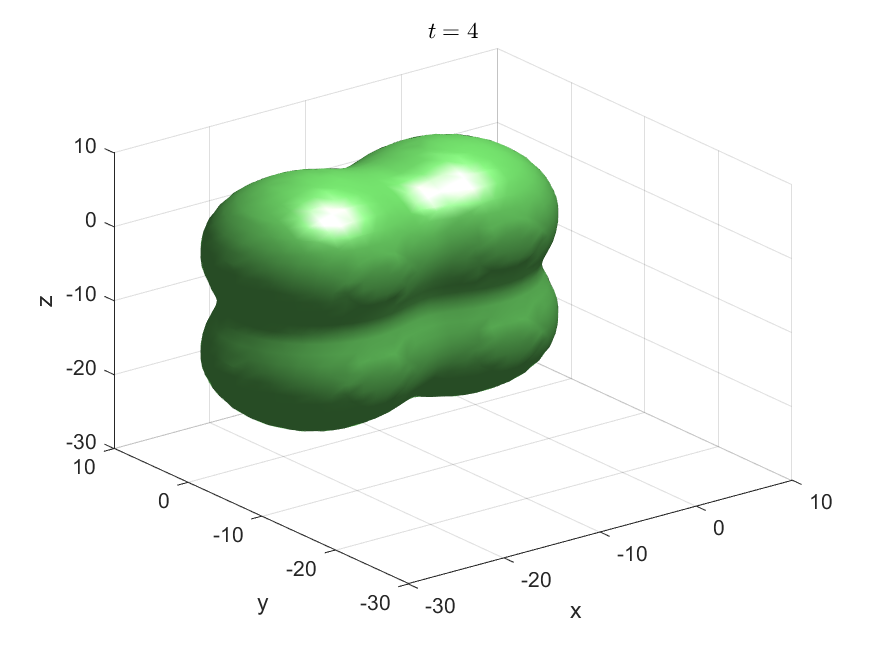}
\end{minipage}
}
\subfigure{
\begin{minipage}[t]{0.3\textwidth}
\centering
\includegraphics[width=5cm]{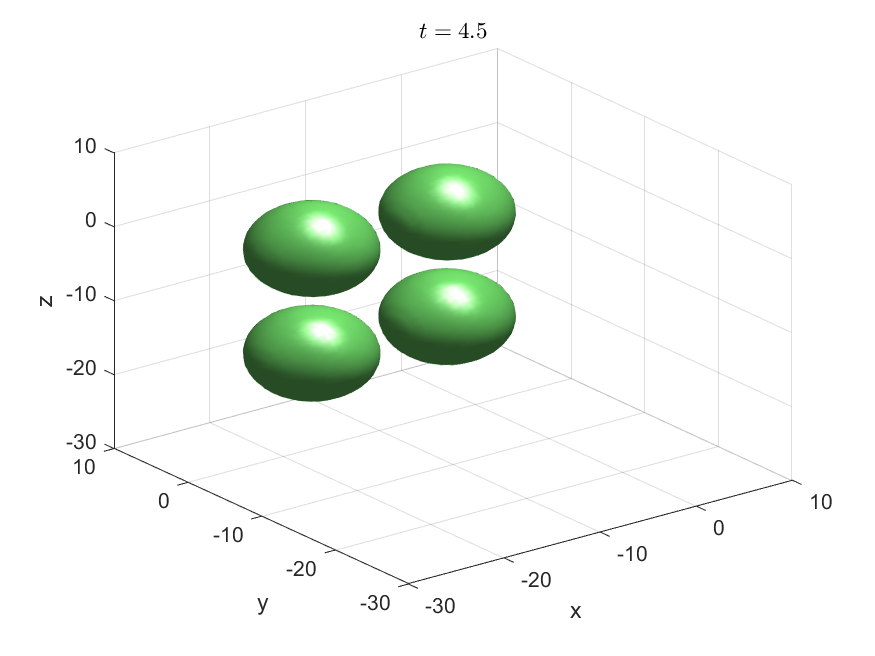}
\end{minipage}
}
\caption{ Isosurfaces for the collisions of four circular solitons in 3D when $\alpha=1.2$. }
\label{afig12}
\end{figure*}
\section{Conclusions}
\label{sec:6}
The improved uniform error bounds $O(h^m+\varepsilon^2\tau^2)$ in $H^{\alpha/2}-$norm are established for the long-time dynamics of the high-dimensional NSFSGE.  We first separate a linear part from the sine function  of the NSFSGE and transform the NSFSGE to an equivalent NSFSE. Then the  numerical scheme based on the time-splitting method in time and the Fourier pseudo-spectral method in space is developed, which can obtain the second-order convergence accuracy in time and  the spectral convergence accuracy in space. By employing the RCO technique, the improved uniform error bounds $O(\varepsilon^2\tau^2)$ for the semi-discrete scheme and $O(h^m +\varepsilon^2\tau^2)$ for the full discrete scheme  are obtained at time $T_\varepsilon$,  which shows the  explicit relationship between the error and $\varepsilon$.  The error bounds $O(h^m +\varepsilon^2\tau^2)$ for the discrete energy are also given. Further, the TSFP method and the improved error bounds are extended to the complex NSFSGE and the oscillatory complex NSFSGE. Finally, we give some numerical examples  in 2D or 3D verifying that the error bounds in the theoretical analysis are sharp. We also give some  applications  to demonstrate the differences in the dynamic behaviors between the fractional sine-Gordon equation and classical  sine-Gordon equation, which implies that the fractional order $\alpha$  has an obvious effect on dynamic behaviors of the NSFSGE. The numerical method and techniques used in this paper can be extended to the long-time dynamics research of other nonlinear fractional equations.

\section*{CRediT authorship contribution statement}
\textbf{Junqing Jia:} Conceptualization, Methodology, Writing - original draft, Software.
\textbf{Xiaoqing Chi:} Conceptualization, Methodology, Writing - review \& editing, Software.
\textbf{Xiaoyun Jiang:} Conceptualization, Methodology, Writing - review \& editing, Supervision.

\section*{Declaration of competing interest}
The authors declare that they have no known competing financial interests or personal relationships that could have appeared to influence the work reported in this paper.

%\section*{Declaration of generative AI in scientific writing}
%The authors declare that they do not use artificial intelligence and AI-assisted technologies in the writting.

\section*{Data Availability}
The datasets analysed during the current study are available from the corresponding author on reasonable request.

%\clearpage
\section*{Acknowledgments}
This work has been supported by the Key International (Regional) Cooperative Research Projects of the National Natural Science Foundation of China (Grants No. 12120101001), the National Natural Science Foundation of China (Grants No. 12301516), the Major Basic Research Project of Natural Science Foundation of Shandong Province (Grants No. ZR2021ZD03), the  Natural Science Youth Foundation of Shandong Province (Grants No. ZR2023QA072), the Postdoctoral Fellowship Program of CPSF (Grants No. GZC20231474).

\end{document}